\documentclass[11pt,letterpaper]{article}
\usepackage[margin=1in]{geometry}

\usepackage{amssymb, amsmath, amsthm}
\usepackage{amsfonts,mathtools,mathrsfs,mathtools,color}
\usepackage{dsfont}
\usepackage[colorlinks]{hyperref}
\usepackage{graphicx,subfigure}
\usepackage{multirow}
\usepackage{soul}
\usepackage[normalem]{ulem}
\usepackage[ruled]{algorithm2e}
\usepackage{algpseudocode}
\usepackage{tabularx}
\usepackage{bm}
\usepackage{enumerate}
\usepackage{xcolor}
\usepackage{comment}

\newcommand{\cA}{\mathcal{A}}
\newcommand{\cB}{\mathcal{B}}
\newcommand{\cC}{\mathcal{C}}
\newcommand{\cD}{\mathcal{D}}
\newcommand{\cE}{\mathcal{E}}
\newcommand{\cF}{\mathcal{F}}
\newcommand{\cG}{\mathcal{G}}

\newcommand{\cL}{\mathcal{L}}
\newcommand{\cM}{\mathcal{M}}
\newcommand{\cN}{\mathcal{N}}

\newcommand{\cP}{\mathcal{P}}
\newcommand{\cQ}{\mathcal{Q}}
\newcommand{\cR}{\mathcal{R}}
\newcommand{\cS}{\mathcal{S}}

\newcommand{\cZ}{\mathcal{Z}}

\newcommand{\EE}{\mathbb{E}}

\newcommand{\NN}{\mathbb{N}}

\newcommand{\PP}{\mathbb{P}}

\newcommand{\RR}{\mathbb{R}}
\newcommand{\SSS}{\mathbb{S}} %

\def\abs#1{\left| #1 \right|}
\newcommand{\norm}[1]{\left\lVert#1\right\rVert}
\newcommand{\floor}[1]{\left\lfloor\, {#1}\,\right\rfloor}

\newcommand{\inparen}[1]{\left(#1\right)}             
\newcommand{\inbraces}[1]{\left\{#1\right\}}           
\newcommand{\insquare}[1]{\left[#1\right]}             

\newtheorem{theorem}{Theorem}[section]
\newtheorem{lemma}[theorem]{Lemma}
\newtheorem{proposition}[theorem]{Proposition}
\newtheorem{corollary}[theorem]{Corollary}

\theoremstyle{definition}
\newtheorem{definition}[theorem]{Definition}

\newtheorem{assumption}[theorem]{Assumption}

\newcommand{\E}{\mathbb{E}}

\newcommand{\R}{\mathbb{R}} 

\newcommand{\N}{\mathbb{N}}

\newcommand{\Z}{\mathbb{Z}}
\renewcommand{\P}{\mathbb{P}}
\newcommand{\eps}{\varepsilon}
\newcommand{\ud}{\textnormal{d}}
\newcommand{\1}{\mathbf{1}}
\renewcommand{\mark}{\dagger}

\newcommand{\leb}{\textnormal{Leb}}
\newcommand{\perm}{\textnormal{perm}}
\newcommand{\dom}{\mathrm{dom}}
\newcommand{\range}{\mathrm{range}}

\newcommand{\Poisson}{\mathrm{Poisson}}
\newcommand{\Unif}{\operatorname{\mathsf{Unif}}}

\newcommand{\PPP}{\textnormal{PPP}}

\newcommand{\Var}{\operatorname{\mathsf{Var}}}

\newcommand{\convWeakly}{\stackrel{\text{w}}{\longrightarrow}}
\newcommand{\convVaguely}{\stackrel{\text{v}}{\longrightarrow}}

\newcommand{\TV}{\operatorname{\mathsf{TV}}}

\newcommand{\XpointPPP}{\operatorname{\mathsf{X}}}
\newcommand{\YpointPPP}{\operatorname{\mathsf{Y}}}
\newcommand{\XSetPPP}{\operatorname{\bm{\mathsf{X}}}}
\newcommand{\YSetPPP}{\operatorname{\bm{\mathsf{Y}}}}
\newcommand{\ProbOnPPP}{\operatorname{\mathsf{Q}}}

\title{Bayesian inference of planted matchings: Local posterior approximation and infinite-volume limit}
\author{Zhou Fan\thanks{Department of Statistics and Data Science, Yale
University, zhou.fan@yale.edu} \and
Timothy L.\ H.\ Wee\thanks{School of Mathematics, Georgia Institute
of Technology, timothy.wee@gatech.edu} \and
Kaylee Y.\ Yang\thanks{Department of Statistics and Data Science,
Yale University, yingxi.yang@yale.edu}}
\date{}

\begin{document}
\begingroup
\renewcommand{\thefootnote}{}
\footnotetext{Authors are listed in alphabetical order.}
\endgroup
\setcounter{footnote}{0}

\maketitle
\begin{abstract}
We study Bayesian inference of an unknown matching $\pi^*$
between two correlated random point sets $\{X_i\}_{i=1}^n$ and
$\{Y_i\}_{i=1}^n$ in $[0,1]^d$, under a critical scaling
$\|X_i-Y_{\pi^*(i)}\|_2 \asymp n^{-1/d}$, in both an
exact matching model where all points are observed
and a partial matching model where a fraction of points may be missing.
Restricting to the simplest setting of $d=1$, in this work, we address the
questions of (1) whether the posterior distribution over matchings is
approximable by a local algorithm, and (2) whether marginal statistics of this
posterior have a well-defined limit as $n \to \infty$. We answer both questions
affirmatively for partial matching, where a decay-of-correlations
arises for large $n$. For exact matching, we show that
the posterior is approximable locally only after a global sorting of the points,
and that defining a large-$n$ limit of marginal statistics requires a
careful indexing of points in the Poisson point process limit of the data,
based on a notion of flow. We leave as an open question the extensions of such
results to dimensions $d \geq 2$.
\end{abstract}

\tableofcontents

\section{Introduction}
Identifying an unknown matching or alignment between two collections of
objects is a fundamental task across many domains of science and engineering,
arising in batch correction and data integration for single-cell genomics
\cite{tran2020benchmark}, particle tracking and trajectory inference
\cite{meijering2012methods}, image matching \cite{ma2021image}, record
linkage and database alignment \cite{sayers2016probabilistic}, and network
alignment \cite{bayati2009algorithms}. Many methods for estimating latent
matchings have been developed in these applications, often using combinations
of ideas around local or nearest neighbor featurization
\cite{haghverdi2018batch}, PCA/CCA/NMF for dimension reduction
\cite{stuart2019comprehensive,korsunsky2019fast}, linear assignment or MAP
estimation \cite{chen2022one,zhu2023robust}, and/or optimal
transport \cite{schiebinger2019optimal,demetci2022scot}. However, while all
such methods return an estimate of a ``most likely'' matching, most do not
provide a calibrated quantification of uncertainty for the underlying match.

Motivated by a need for better theoretical understanding of the algorithmic
feasibility of uncertainty quantification in these problems, in this work we
introduce and study two simple Bayesian models for matching
point sets in $\R^d$.
In an \emph{exact matching} model, we will assume complete observations of two
point sets $X_1,\ldots,X_n \in \R^d$ and $Y_1,\ldots,Y_n \in \R^d$, related
through a latent bijection $\pi^*:[n] \to [n]$ by
\begin{equation}\label{eq:intro_XYObs}
X_i=\bar X_{\pi^*(i)} \qquad \text{and} \qquad Y_i = \bar Y_i
\end{equation}
where $(\bar X_1,\bar Y_1),\ldots,(\bar X_n, \bar Y_n)$ 
are i.i.d.\ data pairs.
We will consider also a \emph{partial matching} model where
subsets of $\{\bar X_1,\ldots,\bar X_n\}$ and $\{\bar Y_1,\ldots,\bar Y_n\}$
may be independently unobserved.
We will study, both statistically and computationally,
posterior inference for $\pi^*$ under a uniform prior distribution over all
candidate matchings (or equivalently, a uniform random ordering of the given
observations $X_1,\ldots,X_n$ and $Y_1,\ldots,Y_n$).

A Gaussian setting of exact matching where $\bar X_i \sim \cN(0,I_d)$
and $\bar Y_i-\bar X_i \sim \cN(0,\sigma^2 I_d)$
was studied in depth in \cite{dai2020achievability,kunisky2022strong}, who
obtained precise characterizations of information-theoretic limits
for perfect and almost-perfect recovery of $\pi^*$, and of estimation error by
the MLE. In the simplest setting of fixed dimension $d$, results of these works
imply that the threshold for almost-perfect recovery of
$\pi^*$ is $\sigma \ll n^{-1/d}$. In this work, we will study the
critical scaling regime $\|\bar Y_i-\bar X_i\|_2 \asymp n^{-1/d}$, where
each point $X_i$ will
typically match to more than one point $Y_j$ with
non-vanishing posterior probability as $n \to \infty$. In this regime, we ask
two basic questions:
\begin{enumerate}
\item Algorithmically, can one approximate the ``marginals'' of the posterior
law $P(\pi \mid X,Y)$ --- i.e.\ the posterior probability vector for the correct match to
each point $X_i$ or $Y_j$ --- using an efficient local algorithm that sees only
a neighborhood of $O(1)$ closest points to $X_i$ or $Y_j$?
\item Statistically, as $n \to \infty$, can one characterize the limiting
empirical
distribution of these posterior marginals, and hence understand properties of
Bayes-optimal inference of $\pi^*$ such as the fraction of mismatched points in
a Bayes estimate or the average size of Bayes credible sets for uncertainty
quantification?
\end{enumerate}

It is evident that these questions relate to the nature of decay-of-correlations and the existence of an infinite-volume limit for the posterior
measure $P(\pi \mid X,Y)$. We will introduce
our models for a general dimension $d$, where these questions are naturally
posed, although the results of our current paper will pertain
only to dimension $d=1$ (where the answers are already somewhat
nuanced, especially for the exact matching model). We discuss below the
additional challenges that may arise in extending our analyses to dimensions $d
\geq 2$.

\subsection{Summary of results}

\paragraph{Partial matching.} In the \emph{partial matching} model, we will
assume that each point $\bar X_1,\ldots,\bar X_N$ and $\bar Y_1,\ldots,\bar Y_N$
for a latent sample size $N$
is independently and randomly observed with a fixed probability $p$. We
study a posterior distribution over partial bijections $\pi^*$ that match each
observed point $X_1,\ldots,X_{N_X}$ to a corresponding point in the opposite
set $Y_1,\ldots,Y_{N_Y}$, or to the empty label $\varnothing$. This imposes
a softer bijectivity constraint for $\pi^*$ than in the exact matching model. 
We establish that for dimension $d=1$, under mild assumptions for the
joint law of $(\bar X_i,\bar Y_i)$ where $|\bar Y_i-\bar X_i|
\asymp n^{-1/d} = n^{-1}$, the posterior law $P(\pi \mid X,Y)$ exhibits
decay-of-correlations in a suitable sense.

This allows us to provide an affirmative answer to Question 1,
showing that posterior marginal match probabilities for each point $X_i$ can
be approximated by computing a local posterior over partial bijections
between points in a window of size $O(n^{-1})$ around $X_i$. As $n \to \infty$,
the empirical distribution of posterior marginals converges weakly to a natural
infinite-volume limit defined over the Poisson point process limit of the
observed data, providing also a simple answer to Question~2.
\paragraph{Exact matching.} In the \emph{exact matching} model
where all points $\bar X_1,\ldots,\bar X_n$ and $\bar Y_1,\ldots,\bar Y_n$ are
observed and inference for an exact bijection is required, we provide a
partially affirmative answer to Question 1, showing that the posterior
marginals are well-approximated by an algorithm that first performs a global
step of sorting $\{X_i\}_{i=1}^n$ and
$\{Y_i\}_{i=1}^n$, and then computing a local posterior over matchings
of $O(1)$ points having the same sorted indices. Our
analyses imply that a more naive approach of computing a local posterior over,
say, bijections between the $k$ points closest to $X_i$ and the $k$ points
closest to its nearest neighbor $Y_j$ would fail to correctly approximate the
posterior law, even in the limit $k \to \infty$.

The need for this global sorting is related to an incomplete
decay-of-correlations of Gibbs measures in the infinite-volume limit.
In a related model of matching a countably infinite point set on $\R$ to itself,
\cite{Biskup2015} showed that there is an infinite set of
extremal infinite-volume Gibbs measures, one corresponding to each
integer value of a conserved quantity called the \emph{flow} that induces 
long-range dependence. In our model, we define an analogous notion
of flow relative to $\pi^*$ (c.f.\ Definition \ref{def:flow}), and we
give an answer to Question 2 by showing that the empirical distribution of
posterior marginals converges to an infinite-volume limit corresponding to
matchings over the limiting Poisson point process that have 0 flow relative to
$\pi^*$.

Thus the answers to these questions are
already different in dimension $d=1$ for the partial and exact matching models,
where a conserved flow variable in exact matching
presents an obstruction to correlation decay. Extending our analyses to
dimensions $d \geq 2$ will carry additional challenges:
\begin{itemize}
\item For exact matching, there is no direct analogue of sorting or ordering
of points when $d \geq 2$, raising a question of how global
information may be used as an input for local approximation of the
posterior marginals to then succeed.
\item 
For both partial and exact matching, analogues of the boundary variables $\Gamma_x(\pi)$ in our analyses, which form a Markov chain in dimension $d=1$, constitute a more general Markov random field in dimensions $d \geq 2$, raising the question of whether there are additional phase transition phenomena for the long-range correlations of the Gibbs measure.
\end{itemize}
We believe these are interesting open questions, which we leave to future work.

\subsection{Related work}

\paragraph{Statistical inference of planted matchings.}
MLE or MAP estimation in the exact matching model
\eqref{eq:intro_XYObs} may be understood as a permutation optimization problem
over a complete weighted bipartite graph, with a random adjacency matrix whose
entries are the costs of matching $X_i$ to $Y_j$. Variants of this problem where
the edge weights are \emph{independent} have also been studied in
\cite{moharrami2021planted,ding2023planted}, with remarkably precise results.

An early work studying our setting of dependent edge weights in a
context of particle tracking was \cite{chertkov2010inference}, followed by
\cite{dai2020achievability,kunisky2022strong} who gave sharp results on
perfect and near-perfect recovery in a Gaussian model for various
noise variances $\sigma^2$ and dimensions $d$. Permutation recovery in linear regression with shuffled responses was studied in \cite{pananjady2017linear}. Variations of a partial matching model were also studied in
\cite{dai2023gaussian}, motivated by database alignment problems with imperfect overlap. 

A related problem is that of matching geometric random graphs,
where one observes only pairwise distances or similarities $k(X_i,X_j)$ and
$k(Y_i,Y_j)$ between data points. We refer
to \cite{wang2022random} for pointers to this literature and recent results.
\paragraph{Gibbs measures on permutations.} The classical Mallows model falls under this framework and remains an active area of research
\cite{mukherjee2016estimation, zhong2021mallows, MaoWu2022Learning,
alimohammadi2025mallows}. Asymptotics and large deviations of 
the Gibbs measure for a matching model in a different graphon-type scaling
limit were studied in \cite{borga2024large}.

Our model scaling and posterior measures are instead more closely related
to models of
\emph{spatial random permutations} studied in a context of Bose-Einstein
condensates in mathematical physics
\cite{feynman1953atomic,suto1993percolation,betz2011Random}. In dimension $d =
1$, \cite{Biskup2015} gave a characterization of infinite-volume Gibbs
measures for a wide class of such models with convex noise potential $V(\cdot)$,
showing that the extremal Gibbs states are in correspondence with an integer parameter
called the \emph{flow}. The infinite-volume limit of our exact matching model
may be viewed as an extension of this setting to non-convex noise potentials
and two different point sets.
To our knowledge, an analogous complete characterization in dimensions $d \geq 2$
seems to be unresolved. Existence of an infinite-volume limit for permutations
on a regular lattice was shown for periodic boundary conditions
by \cite{betz2014random} and in high temperature by
\cite{armendariz2015finite}. The band structure in some such models
has been studied in \cite{fyodorov2021band}.

\section{Main results}
\subsection{Models}

\paragraph{Data model.}
Let $\Lambda:[0,1]^d \to [0,\infty)$ be a probability density on $[0,1]^d$,
and let $V:\RR^d \rightarrow \RR$ be a noise potential function satisfying
$V(\eps)=V(-\eps)$ for all $\eps \in \R^d$, where both $\Lambda(\cdot)$ and
$V(\cdot)$ are fixed and do not depend on $n$. For each $n$, consider the
exchangeable joint density for a pair $(X_0,Y_0) \in \insquare{0,1}^d \times
\insquare{0,1}^d$ given by 
\begin{align}
    p_n(x,y) &= \frac{1}{\cZ_n} \sqrt{\Lambda(x)\Lambda(y)} \exp\inparen{
-V(n^{1/d}(x-y))  } \label{eq:general_model_XY_density} \\
    \cZ_n &= \int_{[0,1]^d} \int_{[0,1]^d}
\sqrt{\Lambda(x)\Lambda(y)} \exp\inparen{
-V(n^{1/d}(x-y))} \, \ud x \, \ud y. \nonumber
\end{align}
The case of Gaussian noise corresponds to a quadratic potential $V(\cdot)$
(up to an unimportant truncation at the boundaries of $[0,1]^d$).
We will always take the scaling convention that
$\int_{\R^d} \exp(-V(\eps)) \, \ud \eps=1$, so that
$q(\eps)=\exp(-V(\eps))$
defines a probability density on $\RR^d$. Denote $\eps=n^{1/d}(y-x)$, and let
\[p_n(x)=\int_{[0,1]^d} p_n(x,y) dy, \qquad p_n(\eps \mid
x)=\frac{n^{-1}p_n(x,x+n^{-1/d}\eps)}{p_n(x)}\]
be the marginal density of $X_0$ and the conditional density of $n^{1/d}(Y_0 -
X_0)$ given $X_0$ in this model. The following basic proposition shows that
$\Lambda(x)$ and $q(\eps)=\exp(-V(\eps))$ may be understood as the limits
of these marginal and conditional densities as $n \to \infty$.

\begin{proposition}\label{prop:datamodel}
Suppose $\Lambda$ is continuous and bounded away from 0 on $[0,1]^d$,
and $V(\cdot)$ is bounded from below on $\RR^d$. Then as $n \to \infty$:
\begin{enumerate}[(a)]
\item For each fixed $x \in (0,1)^d$, $p_n(x) \to \Lambda(x)$.
\item For each fixed $x \in (0,1)^d$ and bounded domain $S \subset \RR^d$,
$\sup_{\eps \in S} |p_n(\eps \mid x)-q(\eps)| \to 0$.
\item $n \cdot \cZ_n \to 1$.
\end{enumerate}
\end{proposition}
\paragraph{Exact matching model.}
Let $(\bar X_1,\bar Y_1),\ldots,(\bar X_n,\bar Y_n)$ be $n$ i.i.d.\ pairs
with distribution $(\bar X_i,\bar Y_i) \sim p_n(x,y)$.
Let $\pi^*:[n] \to [n]$ be a latent bijection, with uniform prior distribution
$$\pi^* \sim \Unif(\SSS_n)$$
where $\SSS_n$ denotes the set of all bijections $\pi:[n] \to [n]$ (i.e.\ the
symmetric group of all permutations on $[n]$).
We observe the data $X=(X_i)_{i=1}^{n}$ and $Y=(Y_i)_{i=1}^{n}$ where
\begin{equation}\label{eq:exact_matching_model}
X_i=\bar X_{\pi^*(i)} \qquad\text{and}\qquad Y_i=\bar Y_i.
\end{equation}
Thus the latent correspondence is $X_i \leftrightarrow Y_{\pi^*(i)}$.
Our goal is to perform inference on $\pi^*$ based on its
posterior law given $(X,Y)$, which takes the form
\begin{align}\label{eq:posterior_general_model_hard_matching}
    P(\pi \mid X,Y)
=\frac{1}{\cZ_\text{exact}}\exp\Big({-}\sum_{i=1}^n V(n^{1/d}(X_i-Y_{\pi(i)}))\Big)
\end{align}
where $\cZ_\text{exact}=\sum_{\pi \in \SSS_n} \exp({-}\sum_{i=1}^n
V(n^{1/d}(X_i-Y_{\pi(i)})))$.

We remark that the assumption of a uniform prior distribution for $\pi^*$
is natural, corresponding to an assumption that the observed data
$X=(X_i)_{i=1}^{n}$ and $Y=(Y_i)_{i=1}^{n}$ are uniformly randomly ordered or,
equivalently, observed as unordered point sets $\{X_1,\ldots,X_n\}$
and $\{Y_1,\ldots,Y_n\}$ on $\R^d$.
\paragraph{Partial matching model.}
We extend the framework to a partial 
matching setting that allows unmatched points. Given $N \geq 1$,
let $(\bar X_1,\bar Y_1),\ldots,(\bar X_N,\bar Y_N)$ be $N$ i.i.d.\ pairs with
distribution $(\bar X_i,\bar Y_i) \sim p_n(x,y)$. (Here $N$ will be random, and
$n$ is a model parameter that will be proportional to $\E[N]$.)
Each value $\bar X_i$ or $\bar Y_i$ is observed independently with
probability $p \in (0,1)$.
Thus, letting $\cS_{XY},\cS_X,\cS_Y,\cS_\varnothing \subseteq [N]$ be the
sets of indices $i$ where both $\{\bar X_i,\bar
Y_i\}$, only $\bar X_i$, only $\bar Y_i$, or neither $\{\bar X_i,\bar Y_i\}$ are
observed, for each $i=1,\ldots,N$ we have
\[i \in \cS_{XY}, \cS_X, \cS_Y, \cS_\varnothing \quad \text{with probability } \quad p^2, p(1-p), p(1-p), (1-p)^2 \text{ respectively}.
\]
Let $N_X=|\cS_{XY}|+|\cS_X|$ denote the number of observed values of $\bar X_i$,
and likewise let $N_Y=|\cS_{XY}|+|\cS_Y|$. Let
\begin{align}\label{eq:piXYstar}
\pi_X^*:[N_X] \to \cS_{XY} \cup \cS_X \qquad \text{and} \qquad 
\pi_Y^*:[N_Y] \to \cS_{XY} \cup \cS_Y
\end{align}
be two latent bijective maps. We observe the data
\begin{align}\label{eq:partialMatching_oberservation}
    (N_X,N_Y), \quad X_i=\bar X_{\pi_X^*(i)}  \,\,\text{ for } i=1,\ldots,N_X, \quad\text{and}\quad Y_j=\bar Y_{\pi_Y^*(j)} & \,\,\text{ for } j=1,\ldots,N_Y
\end{align}
Let us denote by $\dom(\pi^*)=(\pi_X^*)^{-1}(\cS_{XY}) \subseteq [N_X]$ and $
\range(\pi^*)=(\pi_Y^*)^{-1}(\cS_{XY}) \subseteq [N_Y]$
the index sets representing the observed data values which are matched, and by
\[\pi^*:\dom(\pi^*) \to \range(\pi^*),
\qquad \pi^*(i)=(\pi_Y^*)^{-1}(\pi_X^*(i))\]
the latent bijection $X_i \leftrightarrow Y_{\pi^*(i)}$ between these matched
points. We will use the convention that $\pi^*$ implicitly carries the
information about its domain and range, and write
\[\pi^*(i)=\varnothing \text{ if } i \in [N_X] \setminus \dom(\pi^*),
\quad (\pi^*)^{-1}(j)=\varnothing \text{ if } j \in [N_Y] \setminus
\range(\pi^*).\]
Our goal is to perform inference on $\pi^*$ based on
$X=(X_i)_{i=1}^{N_X}$ and $Y=(Y_j)_{j=1}^{N_Y}$.

Given the sample size parameter $n$, we consider a specific prior law for $N$,
\begin{align}\label{eq:partialMatching_prior_N}
    N \sim \Poisson\inparen{(1-p)^{-2}\cZ_n^{-1}}
\end{align}
where $\cZ_n$ is the normalizing constant of the density $p_n(x,y)$.
In light of Proposition \ref{prop:datamodel}(c), this is approximately the
prior $N \sim \Poisson(n/(1-p)^2)$ for large $n$. Conditional on the partition
$(\cS_{XY},\cS_X,\cS_Y,\cS_\varnothing)$ of $[N]$,
we consider independent uniform priors for $\pi_X^*,\pi_Y^*$ over all possible
bijective maps \eqref{eq:piXYstar}, again encoding an assumption of uniform
random ordering for the observed data $(X_i)_{i=1}^{N_X}$
and $(Y_j)_{j=1}^{N_Y}$.

The somewhat unusual choice of
prior \eqref{eq:partialMatching_prior_N} for $N$ is motivated by the
following particularly simple form for the induced posterior law of
$\pi^*=(\pi_Y^*)^{-1} \circ \pi_X^*$.

\begin{proposition}\label{prop:softMatching_posterior}
Let $\SSS_n$ be the space of all \emph{partial bijections} $\pi$
from $[N_X]$ to $[N_Y]$, i.e.\ bijective maps from a subset
$\dom(\pi) \subseteq [N_X]$ to a
subset $\range(\pi) \subseteq [N_Y]$ of equal cardinality. Denote
$\pi(i)=\varnothing$ if $i \notin \dom(\pi)$ and $\pi^{-1}(j)=\varnothing$ if $j
\notin \range(\pi)$. Let
$U_n(x)=\log \frac{p_n(x)}{\sqrt{\Lambda(x)}}$, and define the Hamiltonian
over $\SSS_n$
\begin{align}\label{eq:partialMatching_Hamiltonian}
    H_n(\pi \mid X,Y)
    =\!\sum_{i \in [N_X]:\pi(i) \neq \varnothing} \!\! V(n^{1/d}(X_i-Y_{\pi(i)}))
    \,\,-\!\!\! \sum_{i \in [N_X]:\pi(i)=\varnothing} \!\!\! U_n(X_i)
    \,\,-\!\!\!\sum_{j \in [N_Y]:\pi^{-1}(j)=\varnothing} \!\!\! U_n(Y_j).
\end{align}
Then the posterior law for $\pi^*$ takes the form
\[P(\pi \mid X,Y)=\frac{1}{\cZ_{\text{partial}}}\exp
\Big({-}H_n(\pi \mid X,Y)\Big),
\quad \cZ_{\text{partial}}=\sum_{\pi \in \SSS_n}
\exp\Big({-}H_n(\pi \mid X,Y)\Big).\]
\end{proposition}

\subsection{Results for partial matching}

In the remainder of this paper, we fix $d=1$
and assume the following properties for $\Lambda(\cdot)$ and $V(\cdot)$.

\begin{assumption}\label{asmpt:pairwiseGenerative}
In dimension $d=1$:
\begin{enumerate}[(i)]
    \item \label{asmpt:pairwiseGenerative_h} (Properties of $\Lambda$.)
$\Lambda$ is a continuous probability density function on $[0,1]$, and
there exist constants $\Lambda_{\min},\Lambda_{\max}>0$ such that
    $\Lambda_{\min} \leq \Lambda(x) \leq \Lambda_{\max}$ for all $x \in [0,1]$.
    \item \label{asmpt:pairwiseGenerative_V}(Properties of $V$.)
$V(\eps)=V(-\eps)$ for all $\eps \in \RR$, and there exist constants
$C,c,\delta>0$, $V_{\min} \in \R$, and non-decreasing
functions $f,g:[0,\infty) \to \R$ such that
\[C(1+|\eps|^C) 
\geq g(|\eps|) \geq V(\eps)-V_{\min} \geq f(|\eps|) \geq c|\eps|^{1+\delta}
\text{ for all } \eps \in \R.\]
\end{enumerate}
\end{assumption}

\subsubsection{Local computation of posterior marginals}
In the partial matching model, for each index
$i \in [N_X]$ and $j \in \{\varnothing\} \cup [N_Y]$, let us write as shorthand
\begin{equation}\label{eq:partial_marginal}
P_i(j)=P(\pi(i)=j \mid X,Y)
\end{equation}
for the marginal posterior probability that $X_i$ matches to $\varnothing$ or
$Y_j$. Thus $P_i(\varnothing)+\sum_{j=1}^{N_Y} P_i(j)=1$.
Given a user-specified locality parameter $M>0$ (constant in $n$),
Algorithm \ref{alg:partial} describes our
procedure for local approximation of these marginal laws
$P_1,\ldots,P_{N_X}$.\footnote{By symmetry, one may apply the same algorithm to
approximate the marginal match probabilities for $Y_1,\ldots,Y_{N_Y}$.} This is
the natural procedure that restricts the Hamiltonian
\eqref{eq:partialMatching_Hamiltonian} to a $O(1/n)$-sized window around each
point $X_i$, and approximates $P_i$ via a local
computation of the posterior over partial matchings of points in this window.

\begin{algorithm}[t]\label{alg:partial}
\DontPrintSemicolon
\caption{Local approximation of the partial matching posterior}
\For{$i \gets 1$ \KwTo $N_X$}{
Let $x_-^{(i)}=\max(0,\frac{\floor{n X_i}-M}{n})$ and
$x_+^{(i)}=\min(1,\frac{\floor{nX_i}+M}{n})$.\;
Let $\SSS^{(i)}$ be the space of all partial bijections $\pi$
from $S_X^{(i)}:=\{j:X_j \in [x_-^{(i)},x_+^{(i)}]\}$
to $S_Y^{(i)}:=\{j:Y_j \in [x_-^{(i)},x_+^{(i)}]\}$ (i.e.\ bijections between
subsets of equal cardinality).\;
Compute the local posterior law over $\pi \in \SSS^{(i)}$ given by
$P^{(i)}(\pi \mid X,Y)=\frac{1}{\cZ^{(i)}}
\exp({-}H_n(\pi \mid X,Y)),
\quad \cZ^{(i)}=\sum_{\pi \in \SSS^{(i)}} \exp({-}H_n(\pi \mid X,Y))$
where $H_n(\cdot)$ denotes the restriction of the Hamiltonian
\eqref{eq:partialMatching_Hamiltonian} to $S_X^{(i)},S_Y^{(i)}$.\;
Let $\displaystyle \widehat P_i(j)=\begin{cases}
P^{(i)}(\pi(i)=j \mid X,Y) & \text{ if } Y_j \in
[x_-^{(i)},x_+^{(i)}],\\
0 & \text{ otherwise,}\end{cases}
\quad
\widehat P_i(\varnothing)=P^{(i)}(\pi(i)=\varnothing \mid X,Y).$
}
\Return $\widehat P_1,\ldots,\widehat P_{N_X}$
\end{algorithm}

The following theorem provides a
quantitative $\TV$ guarantee for Algorithm \ref{alg:partial}.

\begin{theorem}\label{thm:soft_alg}
Suppose Assumption \ref{asmpt:pairwiseGenerative} holds.
Let $\widehat{P}_1,\dots,\widehat{P}_{N_X}$ denote the output from
Algorithm~\ref{alg:partial}, applied with $M=KL$ for some $K,L>0$.
Then there exists a constant $C>0$, depending on $\Lambda(\cdot),V(\cdot)$
but not on $K,L$, such that for any $\kappa>1$,
with probability at least $1 - \exp\inparen{-(\log n)^\kappa}$ over $(X,Y)$
for all $n \geq n_0(\kappa,K,L)$,
\begin{equation}\label{eq:TVpartialmatching}
\frac{1}{N_X}\sum_{i=1}^{N_X} \TV(P_i,\widehat P_i) 
\le CL^{-\delta/2}
+ C\inparen{ 1 - e^{-CL\log L}}^{K/3}.
\end{equation}
\end{theorem}

In particular, this implies that for any $\iota>0$,
one may choose $L \equiv L(\iota)$ followed by $K \equiv K(L,\iota)$ large
enough to ensure that this averaged TV error in \eqref{eq:TVpartialmatching} is
at most $\iota$ with high probability
for large $n$.\footnote{For any $D>0$, 
also with probability at least $1-\exp(-(\log n)^\kappa)$ for $n \geq
n_0(\kappa,K,L,D)$, the
fraction of indices $i \in [N_X]$ for which $[x_-^{(i)},x_+^{(i)}]$ contains
more than $D$ observed data values
is at most a constant $o_D(1)$ vanishing as $D \to \infty$.
Bypassing the (possibly intractable) computation of $\widehat P_i$ for such
indices $i \in [N_X]$ yields an additional $o_D(1)$ error in \eqref{eq:TVpartialmatching}.
Thus for any fixed $\iota>0$, we may also choose $D \equiv D(\iota,K,L)$ large
enough to give a linear-time local algorithm that achieves
$\frac{1}{N_X}\sum_{i=1}^{N_X} \TV(P_i,\widehat P_i)<\iota$.}

\subsubsection{Asymptotics of the posterior law}

Let $\PPP(\mu)$ denote the law of a homogeneous Poisson point process on $\RR$
with rate $\mu>0$. Conditional on $(X,Y)$, over a uniform random
choice of index $I \in [N_X]$, the rescaled point processes
$\{n(X_i-X_I)\}_{i \in [N_X]}$ and $\{n(Y_j-X_I)\}_{j \in [N_Y]}$ will converge
weakly a.s.\ as $n \to \infty$ to the following law of
coupled Poisson point processes
(c.f.\ Corollary \ref{cor:partial_matching_PPP_convergence}).

\begin{definition}\label{def:partial_PPP}
Given $p \in (0,1)$ and the probability density $\Lambda(x)$ on $[0,1]$,
let $(\XSetPPP_{\Lambda,p}, \YSetPPP_{\Lambda,p})$
be two stochastic point processes on $\R$, and let $\pi_{\Lambda,p}^*$ be a
bijective map from a subset of points $\dom(\pi_{\Lambda,p}^*)
\subseteq \XSetPPP_{\Lambda,p}$ to a subset of points $\range(\pi_{\Lambda,p}^*)
\subseteq \YSetPPP_{\Lambda,p}$, constructed as follows:
\begin{enumerate}
\item Sample $x \sim \Lambda$.
\item Sample $\XSetPPP_{\Lambda,p}=\{0\} \cup \widetilde \XSetPPP_{\Lambda,p}$ where $\widetilde
\XSetPPP_{\Lambda,p} \sim \PPP(\Lambda(x)p/(1-p)^2)$.
\item Independently for each $\XpointPPP \in \XSetPPP_{\Lambda,p}$, include $\XpointPPP \in
\dom(\pi_{\Lambda,p}^*)$ with probability $p$.
\item Independently for each $\XpointPPP \in \dom(\pi_{\Lambda,p}^*)$,
sample $\eps \sim q(\cdot)$ and let $\YpointPPP=\XpointPPP+\eps$.
Set $\pi_{\Lambda,p}^*(\XpointPPP)=\YpointPPP$ and
$\range(\pi_{\Lambda,p}^*)=\{\pi_\Lambda^*(\XpointPPP):\XpointPPP
\in \dom(\pi_{\Lambda,p}^*)\}$.
\item Sample $\bar \YSetPPP_{\Lambda,p} \sim \PPP(\Lambda(x)p/(1-p))$, and set
$\YSetPPP_{\Lambda,p}=\bar\YSetPPP_{\Lambda,p} \cup \range(\pi_{\Lambda,p}^*)$.
\end{enumerate}
\end{definition}
\noindent
Note that the point $0 \in \XSetPPP_{\Lambda,p}$ corresponds to the
centering point $X_I \in \{X_1,\ldots,X_{N_X}\}$.

We will describe a notion of a weak limit for the empirical distribution of
the posterior marginals and true matches $\{(P_i,\pi^*(i))\}_{i \in [N_X]}$ 
via the following class of test
functions $f(P,j^*)$, where $P$ represents the probability vector of the match
for a given point, and $j^*$ represents its true match.

\begin{definition}[Matching cost functions]
Let $\N=\{1,2,3,\ldots\}$ and let $\cP(\N)$ be the space of probability
distributions on $\N$. A \emph{permutation-invariant and TV-continuous
matching cost} is a function $f:\cP(\N) \times \N \to \R$ satisfying the conditions:
\begin{enumerate}
\item (Permutation invariance) Consider any $P \in \cP(\N)$, $j^* \in \N$,
bijection $\sigma:\N \to \N$, and let $P^\sigma \in \cP(\N)$ be given by
$P^\sigma(j)=P(\sigma^{-1}(j))$. Then
$f(P,j^*)=f(P^\sigma,\sigma(j^*))$.\footnote{I.e., $f(P,j^*)$ depends
only on the probability of the true match $P(j^*)$ and the set of remaining
probability values $\{P(i):i \neq j^*\}$ irrespective of their ordering.}
\item (TV-continuity)
There is a continuous function $\iota:[0,1] \to [0,\infty)$ satisfying
$\lim_{x \to 0} \iota(x)=0$ such that, for any $P,P' \in \cP(\N)$ and $j^* \in \N$,
\[|f(P,j^*)-f(P',j^*)| \leq \iota(\TV(P,P'))
=\iota\bigg(\frac{1}{2}\sum_{j \in \N} |P(j)-P'(j)|\bigg).\]
\end{enumerate}
\end{definition}
We will then write $f(P,Y^*)$ for the application of $f$
to any probability vector $P \in \cP(S)$ over a finite or countable
set $S$ and an element $Y^* \in S$, extending $P \in \cP(S)$
by zero-padding if $|S|<\infty$.

Examples of such matching cost functions include:
\begin{itemize}
\item The probability that $P$ assigns to the true match $j^*$: $f(P,j^*)=P(j^*)$.
\item The expected cardinality of a Bayes credible set $f(P,j^*)=\E|C(P)|$:
Sorting the probability values of $P$ as
$P_{s(1)} \geq P_{s(2)} \geq \ldots$ (ordering tied values uniformly at random)
and fixing a target coverage level $1-\alpha \in (0,1)$,
\[C(P)=\begin{cases}
\{s(1),\ldots,s(k)\} \text{ with probability } \xi \\
\{s(1),\ldots,s(k-1)\} \text{ with probability } 1-\xi
\end{cases}\]
where $k \geq 1$ and $\xi \in (0,1]$ are such that
$P_{s(1)}+\ldots+P_{s(k-1)}+\xi \cdot P_{s(k)}=1-\alpha$. (Here $\E$ is over
the random construction of $C(P)$, and in this example $f(P,j^*)$
does not depend on $j^*$.)
\item The expected coverage of this set $C(P)$ for the true match:
$f(P,j^*)=\P[j^* \in C(P)]$
(where $\P$ is again the probability over the random construction of $C(P)$).
\end{itemize}
The following theorem shows that the average of such functions
evaluated at $\{(P_i,\pi^*(i))\}_{i \in [N_X]}$ --- e.g.\ the
average posterior probability of the true match $P_i(\pi^*(i))$, the average
size of the Bayes credible set $\E|C(P_i)|$, etc. --- converges
to a limit defined under the above Poisson point processes.

\begin{theorem}\label{thm:partial_asymptotics}
Suppose Assumption \ref{asmpt:pairwiseGenerative} holds. Let
$\ProbOnPPP \equiv
\ProbOnPPP(\XSetPPP_{\Lambda,p},\YSetPPP_{\Lambda,p},\pi_{\Lambda,p}^*)
\in \cP(\{\varnothing\} \cup \YSetPPP_{\Lambda,p})$
be the probability vector given by Proposition \ref{prop:partial_infvolume}
below.

Let $f:\cP(\N) \times \N \to \R$ be any bounded, permutation-invariant and
TV-continuous matching cost. Then, almost surely
\[\lim_{n \to \infty} \frac{1}{N_X}\sum_{i=1}^{N_X} f(P_i,\pi^*(i))
=\E_{\Lambda,p}[f(\ProbOnPPP,\pi_{\Lambda,p}^*(0))]\]
where $\E_{\Lambda,p}$ is the expectation over the law of 
$(\XSetPPP_{\Lambda,p},\YSetPPP_{\Lambda,p},\pi_{\Lambda,p}^*)$ in
Definition \ref{def:partial_PPP}.
\end{theorem}
The probability vector
$\ProbOnPPP \in \cP(\{\varnothing\} \cup \YSetPPP_{\Lambda,p})$
of Theorem \ref{thm:partial_asymptotics} may be understood as the
infinite-volume limit of
the vectors of marginal match probabilities for $0 \in \XSetPPP_{\Lambda,p}$
under local restrictions of an asymptotic posterior law, defined as follows: For
any $K>0$, denote $\XSetPPP_K = \XSetPPP_{\Lambda,p} \cap [-K,K]$ and
$\YSetPPP_K = \YSetPPP_{\Lambda,p} \cap [-K,K]$. Let $\SSS_K$ be the space
of all partial bijections from $\XSetPPP_K$ to $\YSetPPP_K$.
Set $U(x)=\sqrt{\Lambda(x)}$, and
define the local Gibbs measure over $\pi \in \SSS_K$ by
\begin{align*}
Q_K(\pi) \, &\propto \,
\exp\bigg({-}\sum_{\XpointPPP \in \XSetPPP_K:\pi(\XpointPPP) \neq \varnothing} V(\XpointPPP-\pi(\YpointPPP))
+\sum_{\XpointPPP \in \XSetPPP_K:\pi(\XpointPPP)=\varnothing} U(\XpointPPP)
+\sum_{\YpointPPP \in \YSetPPP_K:\pi^{-1}(\YpointPPP)=\varnothing} U(\YpointPPP)\bigg).
\end{align*}
Define the probability vector $Q_K^0 = \{Q_K(\pi(0)=\YpointPPP)\}_{\YpointPPP
\in \{\varnothing\} \cup \YSetPPP_K} \in \cP(\{\varnothing\} \cup \YSetPPP_K)$ for the match of the
point $0 \in \XSetPPP_K$ under $Q_K$.

\begin{proposition}\label{prop:partial_infvolume}
$Q_K^0$ converges in $\P_{\Lambda,p}$-probability to a unique limit $\ProbOnPPP
\in \cP(\{\varnothing\} \cup \YSetPPP_{\Lambda,p})$ in $\TV$ as $K \to \infty$,
i.e. for any $\eps > 0$,
\[\lim_{K \to \infty}\P_{\Lambda,p} [\TV(\ProbOnPPP,Q_K^0) \geq \eps] = 0\]
where $\P_{\Lambda,p}$ is the probability over
$(\XSetPPP_{\Lambda,p},\YSetPPP_{\Lambda,p},\pi_{\Lambda,p}^*)$.
\end{proposition}

\subsection{Results for exact matching}
\subsubsection{Local computation of posterior marginals}

In the exact matching model,
let $s,t:[n] \to [n]$ be the permutations that sort $X$ and $Y$, i.e.
\[X_{s(1)}<X_{s(2)}<\ldots<X_{s(n)},\qquad Y_{t(1)}<Y_{t(2)}<\ldots<Y_{t(n)}.\]
Conditioning on the points $X,Y$,
for any matching $\pi$, let $\pi_{st} \coloneqq t^{-1} \circ \pi \circ s$ be 
the corresponding matching under the sorted coordinates, i.e.\
$\pi_{st}(i)=j$ if $\pi(s(i))=t(j)$ (the $i^\text{th}$ ordered value in $X$
matches to the $j^\text{th}$ ordered value in $Y$).
For each index $i \in [n]$, let us write as shorthand
\begin{equation}\label{eq:exact_marginal}
P_i(j)=P(\pi(i)=j \mid X,Y),
\quad P_i^\text{sort}(j)=P_{s(i)}(t(j))=P(\pi_{st}(i)=j \mid X,Y).
\end{equation}
Thus $P_i^\text{sort}(j)$ is the marginal probability that the $i^\text{th}$
ordered value in $X$ matches with the $j^\text{th}$ ordered value in $Y$
under the posterior law \eqref{eq:posterior_general_model_hard_matching}.
Fixing a (integer) locality parameter $M \geq 1$,
Algorithm \ref{alg:exact} describes our procedure for
local approximation of these marginal laws $P_1,\ldots,P_n$ after the above
global sorting of the data values. This algorithm approximates $P_i^\text{sort}$
by computing a local posterior over bijections between $O(1)$ sorted values
around $X_{s(i)}$ to a corresponding set of $O(1)$ sorted values around
$Y_{t(i)}$.

\RestyleAlgo{ruled}
\begin{algorithm}[hbt!]
\DontPrintSemicolon
\caption{Local approximation of the exact matching posterior}
\label{alg:exact}
\For{$i \gets 1$ \KwTo $n$}{
     Let $i_-=\max(i-M,1)$ and $i_+=\min(i+M,n)$.\;
     Let $\SSS^{(i)}$ be the set of bijections $\pi:\{i_-,\ldots,i_+\} \to \{i_-,\ldots,i_+\}.$\;
    Compute the local posterior for matchings $\pi \in \SSS^{(i)}$ given by
    $P^{(i)}(\pi \mid X,Y)
    \propto \exp\bigg(-\sum_{k=i_-}^{i_+} V(n(X_{s(k)}-Y_{t(\pi(k))}))\bigg).
    $\;
     Let $\displaystyle \widehat P_{i}^\text{sort}(j)=\begin{cases}
    P^{(i)}(\pi(i)=j \mid X,Y) & \text{ if } j \in \{i_-,\ldots,i_+\},\\
    0 & \text{ otherwise.}\end{cases}$\;
    Let $\widehat P_i(j)=\widehat P_{s^{-1}(i)}^\text{sort}(t^{-1}(j))$.
}
\Return $\widehat P_1,\ldots,\widehat P_n$.
\end{algorithm}

The following theorem establishes a TV guarantee for this algorithm.
\begin{theorem}\label{thm:exact_alg}
Suppose Assumption \ref{asmpt:pairwiseGenerative} holds.
Let $\widehat{P}_1,\dots,\widehat{P}_n$ denote the output from
Algorithm~\ref{alg:exact}, applied with $M=KL$ for some $K,L \geq 1$.
Then there exists a constant $C>0$, depending on $\Lambda(\cdot),V(\cdot)$ but
not on $K,L$, such that for any $\kappa>1$, with
probability at least $1 - \exp\inparen{-(\log n)^\kappa}$ over $(X,Y)$
for all $n \geq n_0(\kappa,K,L)$,
\begin{align}\label{eq:TVexactmatching}
    \frac{1}{n}\sum_{i=1}^n \TV(P_i,\widehat P_i) 
    \leq CL^{-\delta}+C\inparen{1 - e^{-CL g(4L/\Lambda_{\min})}}^{K/3}.
\end{align}
\end{theorem}
Again, this implies that for any $\iota>0$, one may choose $L \equiv L(\iota)$
followed by $K \equiv K(L,\iota)$ large enough to ensure that this averaged TV
error is at most $\iota$ for large $n$.

\subsubsection{Asymptotics of the posterior law}

As in the partial matching setting, conditional on $(X,Y)$, over a uniform
random choice of index $I \in [n]$ the rescaled
point processes $\{n(X_i-X_I)\}_{i \in [n]}$ and $\{n(Y_i-X_I)\}_{i \in [n]}$
converge weakly a.s.\ to the following law of coupled Poisson point
processes (c.f.\ Lemma \ref{lemma:ExactM_convergenceToPPP2}).

\begin{definition}\label{def:exactMatching_PPP}
Given the probability density $\Lambda(x)$ on $[0,1]$,
let $(\XSetPPP_\Lambda,\YSetPPP_\Lambda)$ be two point processes on $\R$ and
$\pi_\Lambda^*:\XSetPPP_\Lambda \mapsto \YSetPPP_\Lambda$ a bijection between
$\XSetPPP_\Lambda$ and $\YSetPPP_\Lambda$, constructed as follows: 
\begin{enumerate}
\item Sample $x \sim \Lambda$.
\item Sample $\XSetPPP_\Lambda=\{0\} \cup \widetilde \XSetPPP_{\Lambda}$ where
$\widetilde \XSetPPP_{\Lambda} \sim \PPP(\Lambda(x))$.
\item Independently for each $\XpointPPP \in \XSetPPP_\Lambda$, sample $\eps
\sim q(\cdot)$ and let 
$\YpointPPP =\XpointPPP+\eps$.
Set $\pi_\Lambda^*(\XpointPPP)=\YpointPPP$ and
$\YSetPPP_\Lambda=\{\pi_\Lambda^*(\XpointPPP)
:\XpointPPP \in \XSetPPP_\Lambda\}$.
\end{enumerate}
\end{definition}

The following result on the empirical distribution of posterior
marginals $\{(P_i,\pi^*(i))\}_{i \in [n]}$ is
analogous to Theorem \ref{thm:partial_asymptotics} in
the partial matching setting. However, the construction of the limit
$\ProbOnPPP$ is more intricate, and described below.

\begin{theorem}\label{thm:exact_asymptotics}
Suppose Assumption \ref{asmpt:pairwiseGenerative} holds. 
Let $\ProbOnPPP \equiv
\ProbOnPPP(\XSetPPP_\Lambda,\YSetPPP_\Lambda,\pi_\Lambda^*) \in
\cP(\YSetPPP_\Lambda)$ be the probability vector given
by Proposition \ref{prop:exact_infvolume} below.

Let $f:\cP(\N) \times \N \to \R$ be any bounded, permutation-invariant, and TV-continuous matching cost. Then, almost surely
\[\lim_{n \to \infty} \frac{1}{n}\sum_{i=1}^n f(P_i,\pi^*(i))
=\E_{\Lambda}[f(\ProbOnPPP,\pi_\Lambda^*(0))],\]
where $\E_\Lambda$ is the expectation over the law of 
$(\XSetPPP_{\Lambda},\YSetPPP_{\Lambda},\pi_{\Lambda}^*)$ in Definition
\ref{def:exactMatching_PPP}.
\end{theorem}

The probability vector $\ProbOnPPP \in \cP(\YSetPPP_\Lambda)$
is defined as follows:
Restricting to the probability-1 event where $\XSetPPP_\Lambda$ and $\YSetPPP_\Lambda$
are locally finite, we index the points of $\XSetPPP_\Lambda,\YSetPPP_\Lambda$
as
\[\begin{gathered}
\ldots<\XpointPPP_{-2}<\XpointPPP_{-1}< 
\XpointPPP_0=0<\XpointPPP_1<\XpointPPP_2<\ldots \\
\ldots<\YpointPPP_{-2}<\YpointPPP_{-1}
<\YpointPPP_0=\pi_\Lambda^*(0)<\YpointPPP_1
<\YpointPPP_2<\ldots\end{gathered}\]
where the distinguished points $0 \in \XSetPPP_\Lambda$ 
and $\pi_\Lambda^*(0) \in \YSetPPP_\Lambda$ have index 0.
We identify the bijection $\pi_\Lambda^*:\XSetPPP_\Lambda \to
\YSetPPP_\Lambda$ with the bijection of their indices
$\pi_\Lambda^*:\Z \to \Z$.

\begin{proposition}
\label{prop:flux_finiteness_PPP}
For any $a \in \Z$, define
\begin{align*}
\mathsf{L}_a(\pi)=\sum_{i \in \Z:i \leq a} \sum_{j \in \Z:j>a} \1\{\pi(i)=j\},
\qquad \mathsf{R}_a(\pi)=\sum_{i \in \Z:i>a} \sum_{j \in \Z:j \leq a}
\1\{\pi(i)=j\}.
\end{align*}
Then with probability 1 over
$(\XSetPPP_\Lambda,\YSetPPP_\Lambda,\pi_\Lambda^*)$,
we have that
$\mathsf{L}_a(\pi_\Lambda^*),\mathsf{R}_a(\pi_\Lambda^*)<\infty$ 
for every $a \in \Z$, and
\begin{equation}\label{eq:zeroflow}
\mathsf{F}(\pi_\Lambda^*)
:=\mathsf{L}_a(\pi_\Lambda^*)-\mathsf{R}_a(\pi_\Lambda^*) \in \Z
\end{equation}
has the same value at all $a \in \Z$.
\end{proposition}

\begin{definition}\label{def:flow}
A bijection $\pi:\Z \to \Z$ has \emph{flow $k$ relative to $\pi_\Lambda^*$}
if $\mathsf{L}_a(\pi),\mathsf{R}_a(\pi)<\infty$ and
$\mathsf{L}_a(\pi)-\mathsf{R}_a(\pi)=k+\mathsf{F}(\pi_\Lambda^*)$
for every $a \in \Z$.
\end{definition}

For any integer $K \geq 1$, let $\SSS_K$ be the space of bijections
$\pi:\{-K,\ldots,K\} \to
\{-K+\mathsf{F}(\pi_\Lambda^*),\ldots,K+\mathsf{F}(\pi_\Lambda^*)\}$,
and define the local Gibbs measure over $\SSS_K$
\begin{equation}\label{eq:QKdef}
Q_K(\pi)
\, \propto \,
\exp\bigg({-}\sum_{i \in \{-K,\ldots,K\}} V(\XpointPPP_i-\YpointPPP_{\pi(i)})\bigg).
\end{equation}
The limit of $Q_K$ as $K \to \infty$ may be understood as 
an infinite-volume posterior law corresponding to the formal 
Hamiltonian $H(\pi \mid \XSetPPP_\Lambda,\YSetPPP_\Lambda)
=\sum_{i \in \Z} V(\XpointPPP_i-\YpointPPP_{\pi(i)})$, restricted to
bijections $\pi:\Z \to \Z$ with flow 0 relative to $\pi_\Lambda^*$.
Let $Q_K^{0} \in \cP(\YSetPPP_\Lambda) \equiv \cP(\Z)$
be the probability vector defined by
$Q_K^0(j) = Q_K(\pi(0) = j)
\equiv Q_K(\pi(\XpointPPP_0)=\YpointPPP_j)$ for each $j \in \Z$.

\begin{proposition}\label{prop:exact_infvolume}
$Q_K^0$ converges in probability to a unique limit $\ProbOnPPP \in
\cP(\YSetPPP_\Lambda)$ as $K \to \infty$, i.e.~for any $\eps > 0$,
\[
\lim_{K \to \infty}\PP_{\Lambda}[\TV(\ProbOnPPP, Q_K^0) \geq \eps] = 0
\]
where $\PP_{\Lambda}$ is the probability over $(\XSetPPP_\Lambda,\YSetPPP_\Lambda,\pi_\Lambda^*)$.
\end{proposition}

\section{Proof overview for Theorems \ref{thm:exact_alg}
and \ref{thm:exact_asymptotics}}
\label{sec:proofOverview}

We explain our main ideas for proving Theorems \ref{thm:exact_alg}
and \ref{thm:exact_asymptotics} for exact matching, which draw upon insights
of \cite{Biskup2015}: For each index $l \in [n]$,
define the boundary variable $\Gamma_{l+\frac{1}{2}}(\pi)$ of $\pi$ at
$l+\frac{1}{2}$ as
\begin{align*}
     \Gamma_{l+\frac{1}{2}}(\pi)
     \coloneqq \inbraces{ (k, \pi_{st}(k)) \colon k \leq l, \pi_{st}(k) > l}
    \cup  \inbraces{ (k, \pi_{st}(k)) \colon k > l, \pi_{st}(k) \leq l}. 
\end{align*}
Thus $\Gamma_{l+\frac{1}{2}}(\pi)$ is the set of pairs of
indices $(a,b)$ for the \emph{sorted} data values, such that $a \leq l$ and
$b>l$ or vice versa, and $X_{s(a)}$ matches to $Y_{t(b)}$ under $\pi$.
Fixing (large) locality parameters $K,L>0$,
the match probabilities $\widehat P_i^\text{sort}$ computed in
Algorithm \ref{alg:exact} for $M=KL$ are precisely those under the conditional
posterior law $P(\pi \mid X,Y,\Gamma_{i \pm (KL+\frac{1}{2})}(\pi)=\varnothing)$.
\paragraph{Locality of the Gibbs measure.}
Recalling the lower bound
$\Lambda(x) \geq \Lambda_{\min}$, for each $l \in [n]$, we define the $(X,Y)$-dependent event
\begin{align*}
\cA_l&=\bigg\{\sup_{j:|j-(l+\frac{1}{2})|\le L} \{n|X_{s(j)}-X_{s(l)}|,
n|Y_{t(j)}-Y_{t(l)}|\} \leq \frac{3L}{2\Lambda_{\min}},
\; n\!\abs{X_{s(l)} - Y_{t(l)}} \leq L\bigg\}
\end{align*}
and the $\pi$-dependent events
\[\begin{gathered}
\cC_l(\pi)=\left\{|k-m| \leq L \text{ for all } (k,m) \in
\Gamma_{l+\frac{1}{2}}\right\},
\qquad \cL_l(\pi)=\cA_l \cap \cC_l(\pi),\\
\cG_i(\pi)=\inbraces{\sum_{k=1}^{KL}
\1\{\cL_{i-k}(\pi)\} \geq \frac{2KL}{3}}
\cap \inbraces{\sum_{k=0}^{KL-1}
\1\{\cL_{i+k}(\pi)\} \geq \frac{2KL}{3}}
\end{gathered}\]
i.e.\ $\cG_i(\pi)$ is the event that there are sufficiently many sites
$l \in [i-KL,i+KL]$ where the points near $X_{s(l)}$ and $Y_{t(l)}$
are sufficiently regular and
$\pi$ does not have a long-range jump 
(of size $>L$ in the sorted indices)
at the boundary $\Gamma_{l+\frac{1}{2}}(\pi)$.

A key lemma (Lemma \ref{lemma:locality_G_whp})
is that for large $L$, most sites $i \in [n]$ will
satisfy $\cG_i(\pi)$ with high probability over $(X,Y)$
under both $P(\pi \mid X,Y)$ and the local posterior
$P(\pi \mid X,Y,\Gamma_{i \pm (KL+\frac{1}{2})}=\varnothing)$.
Concretely, we will show
\begin{equation}\label{eq:localityprobintro}
\frac{1}{n}\sum_{i=1}^n P(\cG_i^c(\pi) \mid X,Y) \leq CL^{-\delta},
\qquad \frac{1}{n}\sum_{i=1}^n P(\cG_i^c(\pi) \mid X,Y,\Gamma_{i \pm
(KL+\frac{1}{2})}=\varnothing) \leq CL^{-\delta}
\end{equation}
where $\delta>0$ is the exponent for the lower envelope of $V(\eps)$ in
Assumption \ref{asmpt:pairwiseGenerative}.
\paragraph{Correlation decay.}
Fixing $i \in [n]$, let $\cQ_i(\pi)$ be the list of sites $l \in [i-KL,i+KL]$
where $\cL_l(\pi)$ holds.
For each site $m \in [i-KL,i+KL]$,
let $\cF_m(\pi)$ be the information set consisting of $\cQ_i(\pi)$
and the values of the boundary variables 
$\{\Gamma_{l+\frac{1}{2}}(\pi):l \in \cQ_i(\pi) \setminus [m-L,m+L]\}$.
We will establish a quantitative lower bound for
the probability that $\Gamma_{m+\frac{1}{2}}(\pi)$ is empty:
\[P(\Gamma_{m+\frac{1}{2}}(\pi)=\varnothing \mid X,Y,\cL_m(\pi),\cF_m(\pi))
\geq \iota(L):=e^{-CLg(4L/\Lambda_{\min})}\]
where $g(\cdot)$ is the monotonic
upper envelope in Assumption \ref{asmpt:pairwiseGenerative}.
We will show this by constructing a map from bijections $\pi$ with
$\Gamma_{m+\frac{1}{2}}(\pi) \neq \varnothing$ to ones with
$\Gamma_{m+\frac{1}{2}}(\pi)=\varnothing$,
which importantly relies on $\Gamma_{m+\frac{1}{2}}(\pi)$ being defined via the sorted
indices. (The analogous argument in the partial matching model is simpler,
where we will define a boundary variable $\Gamma_x(\pi)$ simply as the index
pairs $(k,\pi(k))$ where $X_k \leq x$ and $Y_{\pi(k)}>x$ or vice versa,
without need for sorting.)

Then, applying an idea in \cite[Lemma 2.14]{Biskup2015} of independently coupling two
posterior samples $\pi,\pi' \sim P(\,\cdot \mid X,Y)$,
this establishes that on the event $\cG_i(\pi) \cap \cG_i(\pi')$,
with probability at least $1-(1-\iota(L)^2)^{K/3}$ over this coupling,
there exists an index $m \in [i-KL,i)$ and also an index $m \in [i,i+KL)$
where both
$\Gamma_{m+\frac{1}{2}}(\pi)=\Gamma_{m+\frac{1}{2}}(\pi')=\varnothing$. This
suffices to imply a correlation decay
\begin{align*}
& \TV\inparen{(P(\pi_{st}(i)=j \mid X,Y,\Gamma_{i\pm
(KL+\frac{1}{2})}=\varnothing))_{j=1}^n, (P(\pi_{st}(i)=j \mid X,Y)_{j=1}^n} \\
& \le C\Big((1-\iota(L)^2)^{K/3}+P(\cG_i^c(\pi) \mid X,Y)
+P(\cG_i^c(\pi) \mid X,Y,\Gamma_{i \pm (KL+\frac{1}{2})}=\varnothing)\Big)
\end{align*}
which together with \eqref{eq:localityprobintro}
leads to Theorem \ref{thm:exact_alg}.
\paragraph{Asymptotics of the Gibbs measure.}

To show Theorem \ref{thm:exact_asymptotics}, we apply Theorem
\ref{thm:exact_alg} to approximate each posterior marginal vector $P_i$
by its estimate $\widehat P_i$ from Algorithm \ref{alg:exact},
\[\frac{1}{n}\sum_{i=1}^n f(P_i,\pi^*(i))
\approx \frac{1}{n}\sum_{i=1}^n f(\widehat P_i,\pi^*(i)).\]
This estimate $\widehat P_i$ is not a local function
of only the points around $X_i$, because its calculation requires knowledge
of which point $Y_j$ has the same sorted rank as $X_i$. We show, however,
that with high probability for large $D>0$,
$\widehat P_i=\widetilde P_i^D$
for a surrogate $\widetilde P_i^D$ that depends only on $\pi^*$ and the point
processes $\{X_1,\ldots,X_n\}$ and $\{Y_1,\ldots,Y_n\}$ restricted to
$[X_i-\frac{D}{n},X_i+\frac{D}{n}]$. This surrogate is given by an analogue of
Algorithm \ref{alg:exact} that computes a local approximation for the flow of
the true bijection $\pi^*$ at the site $i \in [n]$
(c.f.\ Algorithm \ref{alg:exact_localApprox}).
Then, making the further approximation
\[\frac{1}{n}\sum_{i=1}^n f(\widehat P_i,\pi^*(i))
\approx \frac{1}{n}\sum_{i=1}^n f(\widetilde P_i^D,\pi^*(i)),\]
we may apply weak convergence to the Poisson
point process limit to show Theorem \ref{thm:exact_asymptotics}.

\newpage
\appendix
\section{Proofs of Propositions \ref{prop:datamodel} and \ref{prop:softMatching_posterior}}
\begin{proof}[Proof of Proposition \ref{prop:datamodel}]
    We establish $(c)$ first. 
    Using the substitution $\eps = n^{1/d}(y-x)$ in $\R^d$ and $V(\eps)=V(-\eps)$, it follows that
    \begin{align*}
        n \cZ_n &= n \int_{[0,1]^d} \int_{[0,1]^d} \sqrt{\Lambda(x)\Lambda(y)}
\exp\inparen{ -V(n^{1/d}(x-y))  } \, \ud x \, \ud y \\
        &= \int_{[0,1]^d}  \ud x \int_{\R^d} \ud \varepsilon \,
        \sqrt{\Lambda(x) \Lambda\inparen{x + n^{-1/d}\varepsilon}  } 
        \exp\inparen{ -V(\eps) } 
        \boldsymbol{1}\inbraces{x+n^{-1/d}\varepsilon \in [0,1]^d}.
    \end{align*}
    As $n \to \infty$, the
integrand converges pointwise to $\Lambda(x) e^{-V(\varepsilon)}$, and 
    $\int_{[0,1]^d} \Lambda(x) \ud x \int_{\R^d} e^{-V(\varepsilon)}
    \ud \varepsilon = 1$.
    Thus $\lim_{n \to \infty} n \cdot \cZ_n=1$ follows 
    by the dominated convergence theorem.
    
    Statement (a) follows similarly, by applying (c):
    for any fixed $x \in (0,1)^d$,
    \begin{align*}
        p_n(x) &= \frac{1}{\cZ_n} \int_{[0,1]^d}
        \sqrt{\Lambda(x) \Lambda(y)} \exp\inparen{
        -V(n^{1/d}(x-y))
        }  \ud y \\
        &= \frac{1}{n \cZ_n} \int_{\R^d}
        \sqrt{\Lambda(x) \Lambda(x + n^{-1/d}\varepsilon)}
        \exp\inparen{
        -V(\varepsilon)
        } \boldsymbol{1}\inbraces{x+n^{-1/d}\varepsilon \in [0,1]^d}
\ud \varepsilon \\
        &\to \Lambda(x) \int_{\R^d} e^{-V(\varepsilon)} 
        = \Lambda(x).
    \end{align*}
    For (b), by definition,
    \begin{align*}
        p_n(\varepsilon \mid x) = 
        \frac{(n\cZ_n)^{-1}\sqrt{\Lambda(x)\Lambda(x+n^{-1/d}\varepsilon)}
        e^{-V(\varepsilon)}\1\{x+n^{-1/d}\varepsilon \in [0,1]^d\}}{p_n(x)}.
    \end{align*}
Applying (a) and (c), as $n \to \infty$,
for any fixed $x \in (0,1)^d$ and bounded domain $S \subset \R^d$,
this converges uniformly over $\eps \in S$ to $q(\eps)=e^{-V(\eps)}$.
\end{proof}  

\begin{proof}[Proof of Proposition \ref{prop:softMatching_posterior}]
We first compute the full posterior law over all latent variables, conditioned upon
the observations \eqref{eq:partialMatching_oberservation}. We write $\propto$ to
absorb any numerical constants depending only on the observed data
$(N_X,N_Y,X,Y)$. Then
\begin{align*}
    &P\inparen{\pi_X, \pi_Y, \cS_{XY}, \cS_{X}, \cS_{Y}, \cS_{\varnothing}, N
\mid N_X,N_Y,X, Y}\\
    &\propto \, P\inparen{\pi_X, \pi_Y, \cS_{XY}, \cS_{X}, \cS_{Y}, \cS_{\varnothing}, N, N_X, N_Y, X, Y} \\
    &= \int P\inparen{\pi_X, \pi_Y, \cS_{XY}, \cS_{X}, \cS_{Y}, \cS_{\varnothing}, N, N_X, N_Y, X, Y, \bar{X}, \bar{Y}} \, \ud \bar{X} \, \ud \bar{Y}\\
    &= \int P\inparen{X, Y \mid \pi_X, \pi_Y, \cS_{XY}, \cS_{X}, \cS_{Y}, \cS_{\varnothing}, N, N_X, N_Y,\bar{X}, \bar{Y}}  \\
    &\qquad \phantom{\int} \times P\inparen{\pi_X, \pi_Y \mid \cS_{XY}, \cS_{X}, \cS_{Y}, \cS_{\varnothing}, N, N_X, N_Y, \bar{X}, \bar{Y}}  \\
    &\qquad \phantom{\int} \times P\inparen{N_X, N_Y \mid N, \cS_{XY}, \cS_{X}, \cS_{Y}, \cS_{\varnothing}, \bar{X}, \bar{Y}} \\
    &\qquad \phantom{\int} \times P\inparen{\cS_{XY}, \cS_{X}, \cS_{Y}, \cS_{\varnothing} \mid N, \bar{X}, \bar{Y}}  P\inparen{ \bar{X}, \bar{Y} \mid N} P(N)  \, \ud \bar{X} \, d\bar{Y} 
\end{align*}
It follows that 
\begin{align*}
    &P\inparen{\pi_X, \pi_Y, \cS_{XY}, \cS_{X}, \cS_{Y}, \cS_{\varnothing}, N
\mid N_X,N_Y,X, Y} \\
    &\propto \, \int \Bigg\{\prod_{i=1}^{N_X} \boldsymbol{1}\inbraces{X_i =
\bar{X}_{\pi_X(i)}} \prod_{i=1}^{N_Y} \boldsymbol{1}\inbraces{Y_i =
\bar{Y}_{\pi_Y(i)}} \times \frac{1}{N_X ! N_Y!}\\
&\qquad \times \boldsymbol{1}\inbraces{\abs{\cS_{XY} \cup \cS_{X}} = N_X, \,
\abs{\cS_{XY} \cup \cS_{Y}} = N_Y} \\
    &\qquad \times \, 
\boldsymbol{1}\inbraces{\abs{\cS_{XY} \cup \cS_X \cup \cS_Y \cup \cS_\varnothing}=N}
p^{2\abs{\cS_{XY}}} \cdot (p(1-p))^{\abs{S_X}} \cdot
(p(1-p))^{\abs{S_Y}} \cdot (1-p)^{2 \abs{S_\varnothing}}\\
&\qquad \times \prod_{i=1}^N
p_n\!\inparen{\bar{X}_i, \bar{Y}_i} \times P(N) \Bigg\}\ud \bar{X} \, \ud \bar{Y}\\
    &\propto \, \boldsymbol{1}\inbraces{\abs{\cS_{XY} \cup \cS_{X}} = N_X, \,
\abs{\cS_{XY} \cup \cS_{Y}} = N_Y,\, \abs{\cS_{XY}\cup\cS_X \cup \cS_Y \cup
\cS_\varnothing}=N}\\
&\qquad \times p^{2 \abs{\cS_{XY}} + \abs{\cS_X} + \abs{\cS_Y}} (1-p)^{2 \abs{\cS_\varnothing} + \abs{\cS_X} + \abs{\cS_Y}} P(N) \\
    &\qquad \times \prod_{k \in \cS_{XY}} p_n\inparen{X_{\pi_X^{-1}(k)}, Y_{\pi_Y^{-1}(k)}} \prod_{k \in \cS_{X}} p_n\inparen{X_{\pi_X^{-1}(k)}}  \prod_{k \in \cS_{Y}} p_n\inparen{Y_{\pi_Y^{-1}(k)}} \\
    &\propto \, \boldsymbol{1}\inbraces{\abs{\cS_{XY} \cup \cS_{X}} = N_X, \,
\abs{\cS_{XY} \cup \cS_{Y}} = N_Y,\, \abs{\cS_{XY}\cup\cS_X \cup \cS_Y \cup
\cS_\varnothing}=N}\\
&\qquad \times p^{N_X+N_Y} (1-p)^{2N-(N_X+N_Y)} P(N) \\
    &\qquad \times \prod_{k \in \cS_{XY}} \frac{p_n\inparen{X_{\pi_X^{-1}(k)}, Y_{\pi_Y^{-1}(k)}}}{    p_n\inparen{X_{\pi_X^{-1}(k)}}  \cdot p_n\inparen{Y_{\pi_Y^{-1}(k)}}} \prod_{k \in \cS_{XY} \cup \cS_X} p_n\inparen{X_{\pi_X^{-1}(k)}}  \prod_{k \in \cS_{XY} \cup \cS_Y} p_n\inparen{Y_{\pi_Y^{-1}(k)}}.
\end{align*}
Note that $N_X+N_Y$ is constant in this computation,
and that the last two factors can be written as $\prod_{i=1}^{N_X}
p_n(X_i)$ and $\prod_{i=1}^{N_Y} p_n(Y_i)$ and so are also constant in this
computation (depending only on $N_X,N_Y,X,Y$). Substituting in the joint density
\eqref{eq:general_model_XY_density} for $p_n(x,y)$, we then have
\begin{align*}
    &P\inparen{\pi_X, \pi_Y, \cS_{XY}, \cS_{X}, \cS_{Y}, \cS_{\varnothing}, N \mid N_X, N_Y, X, Y} \\
    &\propto \, \boldsymbol{1}\inbraces{\abs{\cS_{XY} \cup \cS_{X}} = N_X, \,
\abs{\cS_{XY} \cup \cS_{Y}} = N_Y,\,
\abs{\cS_{XY} \cup \cS_X \cup \cS_Y \cup \cS_\varnothing}=N} (1-p)^{2N} P(N) \\
    &\, \times \cZ_n^{-\abs{\cS_{XY}}} \exp \inparen{ \sum_{k \in \cS_{XY}} -V\!\inparen{n^{1/d}\inparen{ X_{\pi_X^{-1}(k)} - Y_{\pi_Y^{-1}(k)}  }} + \log \frac{\sqrt{\Lambda\inparen{X_{\pi_X^{-1}(k)}}}}{p_n\!\inparen{X_{\pi_X^{-1}(k)}}} + \log \frac{\sqrt{\Lambda\inparen{Y_{\pi_Y^{-1}(k)}}}}{p_n\!\inparen{Y_{\pi_Y^{-1}(k)}}} } \\
    &\propto \, \boldsymbol{1}\inbraces{\abs{\cS_{XY} \cup \cS_{X}} = N_X, \,
\abs{\cS_{XY} \cup \cS_{Y}} = N_Y,\,
\abs{\cS_{XY} \cup \cS_X \cup \cS_Y \cup \cS_\varnothing}=N} (1-p)^{2N} \cZ_n^{-\abs{\cS_{XY}}} P(N) \\
    &\, \times \exp \inparen{ \sum_{k \in \cS_{XY}} -V\!\inparen{n^{1/d}\inparen{ X_{\pi_X^{-1}(k)} - Y_{\pi_Y^{-1}(k)}  }} - \sum_{k \in \cS_{X}}\log \frac{\sqrt{\Lambda\inparen{X_{\pi_X^{-1}(k)}}}}{p_n\!\inparen{X_{\pi_X^{-1}(k)}}} - \sum_{k \in \cS_{Y}} \log \frac{\sqrt{\Lambda\inparen{Y_{\pi_Y^{-1}(k)}}}}{p_n\!\inparen{Y_{\pi_Y^{-1}(k)}}} },
\end{align*}
where the last line uses the fact that $\sum_{i = 1}^{N_X} \log \inparen{
\sqrt{\Lambda(X_i)} / p_n(X_i) }$ and $\sum_{i = 1}^{N_Y} \log \inparen{
\sqrt{\Lambda(Y_i)} / p_n(Y_i) }$ are also constant in this computation.
Denoting $\pi=\pi_Y^{-1} \circ \pi_X$, which by convention
carries the information about $\dom(\pi)=\pi_X^{-1}(\cS_{XY})$ and
$\range(\pi)=\pi_Y^{-1}(\cS_{XY})$, and recalling the definition of the
Hamiltonian $H_n(\pi \mid X,Y)$ in \eqref{eq:partialMatching_Hamiltonian},
this is equivalently
\begin{align*}
&P\inparen{\pi_X, \pi_Y, \cS_{XY}, \cS_{X}, \cS_{Y}, \cS_{\varnothing}, N \mid
N_X, N_Y, X, Y}\\
&\propto \boldsymbol{1}\inbraces{\abs{\cS_{XY} \cup \cS_{X}} = N_X, \,
\abs{\cS_{XY} \cup \cS_{Y}} = N_Y,\,
\abs{\cS_{XY} \cup \cS_X \cup \cS_Y \cup \cS_\varnothing}=N}\\
&\qquad \times (1-p)^{2N} \cZ_n^{-\abs{\cS_{XY}}} P(N)
\exp\big({-}H_n(\pi \mid X,Y)\big).
\end{align*}

The marginal posterior probability $P(\pi \mid X,Y) \equiv P(\pi \mid
N_X,N_Y,X,Y)$ is given by further summing over all
$(\pi_X,\pi_Y,\cS_{XY},\cS_X,\cS_Y,\cS_\varnothing,N)$ satisfying
$|\cS_{XY} \cup \cS_X|=N_X$, 
$|\cS_{XY} \cup \cS_Y|=N_Y$,
$|\cS_{XY} \cup \cS_X \cup \cS_Y \cup \cS_\varnothing|=N$, and
$\pi_Y^{-1} \circ \pi_X=\pi$. Note that fixing $\pi$ fixes also
$|\dom(\pi)|=|\range(\pi)|=|\cS_{XY}|$. Thus, for fixed $(N,N_X,N_Y,\pi)$, all
cardinalities $|\cS_{XY}|,|\cS_X|,|\cS_Y|,|\cS_\varnothing|$ are determined.
The number of tuples
$(\pi_X,\pi_Y,\cS_{XY},\cS_X,\cS_Y,\cS_\varnothing)$ fixing
$(N,N_X,N_Y,\pi)$ is
\[ \binom{N}{\abs{\cS_{XY}}, \abs{\cS_{X}}, \abs{\cS_{Y}}, N - \abs{\cS_{XY}} -
\abs{\cS_{X}} - \abs{\cS_{Y}}} \abs{\cS_{XY}}! \abs{\cS_{X}}! \abs{\cS_{Y}}! =
\frac{N!}{(N-N_X-N_Y+|\cS_{XY}|)!}.\]
(The multinomial coefficient is the number of ways to choose
$\cS_{XY},\cS_X,\cS_Y,\cS_\varnothing$. Then there are $|\cS_X|!$ ways to choose
the restriction of $\pi_X$ to a bijection $[N_X] \setminus \dom(\pi) \to \cS_X$,
$|\cS_Y|!$ ways to choose the restriction of $\pi_Y$ to a bijection $[N_Y]
\setminus \range(\pi) \to \cS_Y$, $|\cS_{XY}|!$ ways to choose the restriction
of $\pi_X$ to a bijection
$\dom(\pi) \to \cS_{XY}$, and this determines via $\pi=\pi_Y^{-1} \circ \pi_X$
the last restriction
of $\pi_Y$ to a bijection $\range(\pi) \to \cS_{XY}$.) Then for fixed
$(N_X,N_Y,\pi)$, marginalizing also over $N$ which takes possible values
$N \geq N_X+N_Y-|\cS_{XY}|=|\cS_X|+|\cS_Y|+|\cS_{XY}|$, this shows
\begin{align*}
&P(\pi \mid N_X,N_Y,X,Y)\\
&\propto \sum_{N \geq N_X+N_Y-|\cS_{XY}|}
(1-p)^{2N}\cZ_n^{-|\cS_{XY}|} \cdot \frac{N!}{(N-N_X-N_Y+|\cS_{XY}|)!} \cdot
P(N) \exp\big({-}H_n(\pi \mid X,Y)\big).
\end{align*}
Finally, with the choice of Poisson prior \eqref{eq:partialMatching_prior_N} for
$N$, the sum
\begin{align*}
    &\sum_{N \geq N_X+N_Y-|\cS_{XY}|}
    (1-p)^{2N}\cZ_n^{-|\cS_{XY}|} \cdot \frac{N!}{(N-N_X-N_Y+|\cS_{XY}|)!} \cdot P(N) \\
    &\quad = \cZ_n^{-|\cS_{XY}|} \exp\inparen{ - \frac{1}{(1-p)^2 \cZ_n} } \sum_{N \geq N_X+N_Y-|\cS_{XY}|} \frac{\cZ_n^{-N}}{(N-N_X-N_Y+|\cS_{XY}|)!} \\
    &\quad = \cZ_n^{- N_X - N_Y} \exp\inparen{ - \frac{1}{(1-p)^2 \cZ_n} } \sum_{N' \geq 0} \frac{\cZ_n^{-N'}}{N'!} 
\end{align*}
is the same for all $\pi$, not depending on
$|\cS_{XY}|=|\dom(\pi)|=|\range(\pi)|$.
Thus $P(\pi \mid X,Y) \equiv
P(\pi \mid N_X,N_Y,X,Y) \propto \exp({-}H_n(\pi \mid X,Y))$.
\end{proof}

\section{Proofs for exact matching}
In the remainder of these appendices, we restrict to $d=1$.
Throughout, $C,C',c,c'>0$ etc.\ denote constants whose values may change from
line to line. We will write $\P[\cdot], \E[\cdot]$ for probabilities and
expectations over the data $(X,Y)$, and $P(\,\cdot \mid X,Y)$, $E(\,\cdot \mid X,Y)$ for probabilities
and expectations under the posterior law $P(\pi \mid X,Y)$ given $X,Y$.

For a $(X,Y)$-dependent event $\cE$, we will say that
\[\cE \text{ holds w.h.p.}\]
if, for any constant $\kappa>1$,
\[\P[\cE] \geq 1-\exp({-}(\log n)^\kappa) \text{ for all } 
n \geq n_0(\kappa).\]
Here, $n_0(\kappa)>0$ is a constant that
may depend on $\kappa$, the functions $\Lambda(\cdot)$, $V(\cdot)$ and
their associated bounds in Assumption \ref{asmpt:pairwiseGenerative},
and locality constants $K,L,D$ etc.\ that may define the event $\cE$.

We will assume Assumption \ref{asmpt:pairwiseGenerative} and the setting of the
exact matching model \eqref{eq:exact_matching_model} throughout the lemmas and
analyses of this section.

\subsection{Proof of Theorem \ref{thm:exact_alg}}

\subsubsection{Mean potential}

We first establish a bound for the mean value of the
potential under the posterior law $P(\pi \mid X,Y)$ of
\eqref{eq:posterior_general_model_hard_matching}.

\begin{lemma}
\label{lemma:mean-potential-O(1)-whp}
There exists a constant $C>0$ such that
\begin{align*}
E \insquare{\frac{1}{n} \sum_{i=1}^{n}V(n(X_i-Y_{\pi(i)}))\;\Bigg\rvert\; X,Y}
\le C \text{ w.h.p.}
\end{align*}
\end{lemma}
\begin{proof}
Consider the following model (in dimension $d=1$)
with $\beta V$ in place of $V$ in \eqref{eq:general_model_XY_density}:
\begin{align}
    p_{n,\beta}(x,y) &= \frac{1}{\cZ_n(\beta)} \sqrt{\Lambda(x) \Lambda(y)} \exp\inparen{ -\beta V(n(x-y))  } \label{eq:general_model_XY_density_with_beta} \\
    \cZ_n(\beta) &= \int_0^1 \int_0^1 \sqrt{\Lambda(x)\Lambda(y)} \exp\inparen{ - \beta V(n(x-y))  } \, \ud x \, \ud y \nonumber
\end{align}
The posterior \eqref{eq:posterior_general_model_hard_matching} in this setting becomes
\begin{align}\label{eq:posterior_general_model_hard_matching_with_beta}
    P_\beta(\pi \mid X,Y)
&=\frac{1}{\cZ_\text{exact}(\beta)}\exp\Big({-\beta}\sum_{i=1}^n V(n(X_i-Y_{\pi(i)}))\Big)
\end{align}
where $\cZ_\text{exact}(\beta)=\sum_{\pi \in \SSS_n} \exp({- \beta}\sum_{i=1}^n
V(n(X_i-Y_{\pi(i)})))$.
We denote expectations under $P_\beta$ by $E_{\beta}\insquare{\, \cdot \mid X,
Y}$. Define the corresponding free energy and mean potential
\begin{align}
    \cF_n(\beta) \coloneqq \frac{1}{n} 
    \log \cZ_{\text{exact}}(\beta), \qquad\cD_n(\beta) \coloneqq 
    E_\beta \insquare{\frac{1}{n} 
    \sum_{i=1}^n V(n(X_i - Y_{\pi(i)})) \;\Bigg\rvert\; X, Y
    }.
\end{align}
It is then standard to check that
\[\cD_n(\beta)={-}\cF_n'(\beta),
\qquad
\cF_n''(\beta)=\frac{1}{n} \Var_\beta
    \insquare{\sum_{i=1}^n V(n(X_i - Y_{\pi(i)})) \,\Bigg|\, X,Y } \geq 0.\]
In particular $\cF_n$ is convex, so we have 
\[\cD_n(1) \le 2\inparen{\cF_n(1/2)-\cF_n(1)}.\]

A lower bound for $\cF_n(1)$ is obtained by lower-bounding the sum over $\pi \in
\SSS_n$ by the single summand $\pi=\pi^*$. This leads to 
\begin{align*}
    \cF_n(1) \ge  \frac{1}{n} \log \exp\inparen{ -
    \sum_{i=1}^n V\inparen{n \inparen{\bar{X}_i - \bar{Y}_i}} } 
    = -\frac{1}{n}\sum_{i=1}^n V\inparen{n \inparen{\bar{X}_i - \bar{Y}_i}} 
\end{align*}
For any $\mu>0$,
\begin{align*}
    \P \inparen{
    \frac{1}{n}\sum_{i=1}^n V\inparen{n \inparen{\bar{X}_i - \bar{Y}_i}} \ge t
    } 
    &
    \le e^{-\mu t n} \E \exp\inparen{
    \mu \sum_{i=1}^n V\inparen{n \inparen{\bar{X}_i - \bar{Y}_i}} 
    } \\
    &= e^{-\mu t n} \insquare{\E \exp\inparen{
    \mu V\inparen{\varepsilon} 
    } }^n 
    = e^{-\mu t n}
    \insquare{\int_{-\infty}^{\infty} e^{(\mu-1) 
    V(\varepsilon)} \ud \varepsilon}^n
\end{align*}
Choosing $\mu=\frac{1}{2}$ and $t>0$ large enough (depending on $V$) shows
\begin{align}
\label{eq:lower_bound_F_n(1)}
\cF_n(1) \ge -C \text{ w.h.p.}
\end{align}

An upper bound for $\cF_n(1/2)$ is obtained as follows:
Let $A$ be a $n \times n$ matrix with $(i,j)$-entry
\begin{align*}
    A_{ij} := \exp\inparen{ -\frac{1}{2}
    V(n( \bar X_i -  \bar Y_j))  }.
\end{align*}
Observe that $\cF_n(1/2)=n^{-1}\log \perm(A)$. Since $A$ has non-negative entries, its permanent can be bounded from above by the product of its row sums:
\begin{align}
\label{eq:pairwiseGenerative_hard_meanDisplacement_afterPerm}
    \cF_n(1/2) \leq \frac{1}{n} \sum_{i=1}^{n} 
    \log \sum_{j=1}^{n}
    \exp\inparen{-\frac{1}{2} V(n(\bar X_i - \bar Y_j))} .
\end{align}
Fix any $\kappa>1$. For each $i \in [n]$, let $\cN_\kappa(\bar Y_i) \subset [n]$
denote the set of indices of the $\floor{(\log n)^\kappa}$ 
nearest neighbors of $\bar Y_i$ in $\{\bar Y_1,\ldots,\bar Y_n\}$. 
Note that
\begin{align}
    \sum_{j \notin \cN_\kappa(\bar Y_i)} \exp\inparen{-\frac{1}{2} V(n(\bar X_i - \bar Y_j))}
    &\le \sum_{j \notin \cN_\kappa(\bar Y_i)}
    \exp\inparen{-\frac{n}{2}\abs{\bar X_i - \bar Y_j}+C'} 
    \nonumber \\ 
    &\le \exp\inparen{C' + \frac{n}{2} \abs{ \bar X_i - \bar Y_i}} \sum_{j \notin \cN_\kappa(\bar Y_i)}
    \exp\inparen{-\frac{n}{2}\abs{\bar Y_i - \bar Y_j} }
    \label{eq:non-(log n)^kappa-neighbor-bound}
\end{align}
where the first inequality uses $V(\eps) \ge c|\eps|^{1+\delta}+V_{\min} \geq
|\eps|-C'$ for some constant $C'>0$, and 
the second inequality uses
$\abs{\bar X_i - \bar Y_j} \ge 
\abs{\bar Y_i - \bar Y_j} - \abs{\bar X_i - \bar Y_i}$.
Let
\[B_i=\sum_{j:j \neq i} \boldsymbol{1}\inbraces{n \abs{
\bar Y_i - \bar Y_j} \le 2 \log n}.\]
The law of $B_i$ conditional on $\bar Y_i$ is
$\text{Binomial}(n-1,p_{n,i})$
where $p_{n,i} \leq (C\log n)/n$ for a constant $C>0$.
Then by the Chernoff bound, for any $t > 0$
\begin{align}
    \P\inparen{B_i \geq (\log n)^\kappa \mid \bar Y_i}
       \le & \; \exp\inparen{ - t (\log n)^\kappa} 
    \E[\exp\inparen{t B_i} \mid \bar Y_i] \quad 
    \nonumber \\
    \le & \;\exp\inparen{ - t (\log n)^\kappa + 
    n p_n (e^t - 1)}\nonumber\\
    \le &\; \exp\inparen{ - t (\log n)^\kappa + 
   C(e^t - 1)\log n}.
   \label{eq:chernoff_bound_binomial}
\end{align}
Taking $t = \log[(\log n)^{\kappa-1}] = (\kappa-1)\log \log n$, we obtain
\begin{align}
\label{eq:chernoff_bound_binomial_opt}
\P\inparen{B_i \geq (\log n)^\kappa \mid \bar Y_i}
\leq \exp({-}c(\log n)^\kappa \log \log n).
\end{align}
If $B_i<(\log n)^\kappa$, then the distance from $\bar Y_i$ 
to each point $\{\bar Y_j:j \notin \cN_\kappa(\bar Y_i)\}$ is greater than
$(2\log n)/n$. Thus by \eqref{eq:non-(log n)^kappa-neighbor-bound},
with probability at least 
$1 - \exp\insquare{{-} c(\log n)^\kappa  \log \log n}$
over $(\bar Y_j)_{j \neq i}$,
\begin{align*}
    \sum_{j \notin \cN_\kappa(\bar Y_i)} \exp\inparen{-\frac{1}{2} V(n(\bar X_i - \bar Y_j))} 
    &\le \exp\inparen{C' + \frac{n}{2} \abs{ \bar X_i - \bar Y_i}} n
\exp\inparen{- \frac{n}{2} \cdot \frac{2\log n}{n}} \\
    & \le \exp\inparen{C' + \frac{n}{2} \abs{ \bar X_i - \bar Y_i}}
\end{align*}
Taking a union bound, with probability
$1-n\exp\insquare{{-} c(\log n)^\kappa  \log \log n}$,
\begin{align}
\label{eq:upper_bd_Fn1/2_bdd_diff_fxc}
    \cF_n(1/2) \le \frac{1}{n} \sum_{i=1}^n \log \insquare{
    \exp\inparen{C' + \frac{n}{2} \abs{ \bar X_i - \bar Y_i}}  + 
    \sum_{j \in \cN_\kappa(\bar Y_i)}\exp\inparen{-\frac{1}{2} V(n(\bar X_i -
\bar Y_j))}}
\end{align}
Note that the given condition
$q(\eps)=e^{-V(\eps)} \leq e^{-c|\eps|^{1+\delta}-V_{\min}}$ implies the tail
bound
\begin{equation}\label{eq:epstailbound}
\P[|\eps| \geq t] \leq Ce^{-ct^{1+\delta}} \text{ for all } t \geq 0.
\end{equation}
Then
\begin{align*}
    \P\Big(\max_{i\in [n]} \abs{\varepsilon_i} \ge (\log n)^\kappa\Big) 
    \le \sum_{i=1}^n \P(\abs{\varepsilon_i} \ge (\log n)^\kappa) \le Cn
\exp\inparen{-c(\log n)^{\kappa(1+\delta)}}.
\end{align*}
It follows that with probability
$1-n\exp\insquare{{-} c(\log n)^\kappa  \log \log n}
-Cn \exp\inparen{-c(\log n)^{\kappa(1+\delta)}}$,
\begin{align}
\label{eq:upper_bound_free_energy_by_bdd_diff_fxc}
    \cF_n(1/2) \le 
    \frac{1}{n}
    \underbrace{\sum_{i=1}^n \log \insquare{
    \exp\inparen{C' +
    \frac{1}{2}\inparen{
    n\abs{ \bar X_i - \bar Y_i} \wedge (\log n)^\kappa
    }
    }  + 
    \sum_{j \in \cN_\kappa(\bar Y_i)}\exp\inparen{-\frac{1}{2} V(n(\bar X_i - \bar Y_j))} 
    }}_{:=\mathscr{F}\inparen{(\bar X_1, \bar Y_1), \ldots,
    (\bar X_n, \bar Y_n)}}
\end{align}
This function $\mathscr{F} \colon ([0,1]^2)^n \to \R$
satisfies the bounded differences property as follows. For any $k \in [n]$,
$(\bar X_1, \bar Y_1), \ldots,(\bar X_k, \bar Y_k), \ldots, 
(\bar X_n, \bar Y_n)$ and $(\bar X_k', \bar Y_k') \in [0,1]^2$,
noting that at most $C(\log n)^\kappa$ summands $i \in [n]$ defining
$\mathscr{F}$ can change and applying a uniform upper and lower bound for
each summand, we have
\begin{align*}
   & \abs{
    \mathscr{F}\inparen{(\bar X_1, \bar Y_1), \ldots, (\bar X_k, \bar Y_k), \ldots,
    (\bar X_n, \bar Y_n)}  - 
    \mathscr{F}\inparen{(\bar X_1, \bar Y_1), \ldots,
    (\bar X_k', \bar Y_k'), \ldots,
    (\bar X_n, \bar Y_n)}
    }
    \le C'(\log n)^{2\kappa},
\end{align*}
almost surely (excluding the measure zero event where $\bar Y_k$ or $\bar Y_k'$ have more than $(\log n)^{\kappa}$ neighbors). Thus for any constant $t>0$, by the bounded differences inequality,
\begin{align}
\label{eq:bounded_diff_inequality}
    \P\inparen{ \frac{\mathscr{F} - \E \mathscr{F}}{n} > t }
    \le \exp\inparen{- \frac{cnt^2}{(\log n)^{4\kappa}}}.
\end{align}
We can bound $\frac{1}{n} \E \mathscr{F}$ from above as follows:
\begin{align*}
    \frac{1}{n} \E \mathscr{F} 
    &\le \frac{1}{n}  \sum_{i=1}^n \E \log \insquare{
    \exp\inparen{C' +
    \frac{n}{2}
    \abs{ \bar X_i - \bar Y_i}
    }  + 
    \sum_{j=1}^n\exp\inparen{-\frac{1}{2} V(n(\bar X_i - \bar Y_j))} 
    } \\
    &\le \frac{1}{n}  \sum_{i=1}^n \log \insquare{
    \E \exp\inparen{C'+
    \frac{n}{2}
    \abs{ \bar X_i - \bar Y_i}
    }  + 
    \sum_{j=1}^n \E \exp\inparen{-\frac{1}{2} V(n(\bar X_i - \bar Y_j))} 
    }
\end{align*}
For $i \neq j$, applying Proposition \ref{prop:datamodel} we may check that
\[\E\left[\exp\inparen{-\frac{1}{2} V(n(\bar X_i - \bar Y_j))}\;\Bigg|\;\bar
X_i\right]
=\int_0^1 p_n(y) \exp\inparen{-\frac{1}{2} V(n(\bar X_i - y))} \ud y
\le \frac{C}{n}.\]
Applying also a constant bound for $\E[\exp(|\eps|/2)]$
and $\E[\exp(-V(\eps)/2)]$ for $i=j$, we obtain
that $\frac{1}{n}\E \mathscr{F} \leq C$ for a constant $C>0$.
Combining with
\eqref{eq:upper_bound_free_energy_by_bdd_diff_fxc} and
\eqref{eq:bounded_diff_inequality}, this shows $\cF_n(1/2) \le C'$ w.h.p.
Together with \eqref{eq:lower_bound_F_n(1)}, this implies
$\cD_n(1) \le 2(\cF_n(1/2)-\cF_n(1)) \leq C''$ as desired.
\end{proof}

We next show a similar bound under a local restriction of the posterior measure.
Conditional on the data $X,Y$, recall the identification of $\pi$ with
$\pi_{st}$ on the sorted indices. For each $l \in \Z$, define the boundary
variable
\begin{align}
     \Gamma_{l+\frac{1}{2}}(\pi)
     =\inbraces{(k,\pi_{st}(k)) \colon k \leq l, \pi_{st}(k)>l}
    \cup  \inbraces{ (k,\pi_{st}(k)) \colon k>l, \pi_{st}(k)\leq l}\label{eq:definition_Gamma_pi_st}
\end{align}
(where by definition $\Gamma_{l+\frac{1}{2}}=\varnothing$ if $l \leq 0$ or $l
\geq n$).
Note that $\widehat P_i^\text{sort}(j)$ computed in Algorithm
\ref{alg:exact} is precisely the posterior probability $P(\pi_{st}(i)=j \mid
X,Y,\Gamma_{i \pm (KL+\frac{1}{2})}(\pi)=\varnothing)$, conditioning on the event
that both boundaries $KL+\frac{1}{2}$ to the left and right of $i$ are empty.

\begin{lemma}
\label{lemma:local_mean_potential}
Fix integers $K,L>0$. For each $i\in [n]$, 
denote $i_-=\max(i-KL,1)$ and $i_+=\min(i+KL,n)$. Then
there exists a constant $C>0$ not depending on $K,L$ such that
\begin{align*}
\frac{1}{n} \sum_{i=1}^n E\Bigg[\frac{1}{KL}\sum_{j=i_-}^{i_+}
V(n(X_{s(j)}-Y_{\pi(s(j))}))\;\Bigg|\;X,Y,\Gamma_{i \pm
(KL+\frac{1}{2})}(\pi)=\varnothing \Bigg] \le C \text{ w.h.p.}
\end{align*}
\end{lemma}
\begin{proof}
Conditional on $X,Y$, define
\[P_\beta^i(\pi \mid X,Y)
=\frac{1}{\cZ_i(\beta)}\exp\left({-}\beta\sum_{j=i_-}^{i_+}
V(n(X_{s(j)}-Y_{\pi(s(j))}))\right),
\quad \cF_i(\beta)=\frac{1}{KL}\log \cZ_i(\beta)\]
Denote its corresponding expectation by $E_\beta^i[\,\cdot\mid X,Y]$. Then the
quantity to be bounded is precisely
$\frac{1}{n}\sum_{i=1}^n \cD_n^i$, where
\begin{align*}
     \cD_n^i:=E_{\beta=1}^i\Bigg[\frac{1}{KL}\sum_{j=i_-}^{i_+}
V(n(X_{s(j)}-Y_{\pi(s(j))}))\;\Bigg|\;X,Y\Bigg]
={-}\cF_i'(1)
\leq 2(\cF_i(1/2)-\cF_i(1)).
\end{align*}

Here $\frac{1}{n}\sum_{i=1}^n \cF_i(1/2)$ can be bounded from 
above similar to the proof of Lemma \ref{lemma:mean-potential-O(1)-whp}:
Let $A^i$ be the matrix with rows/columns indexed by $\{i_-,\ldots,i_+\}$,
with entries $A^i_{jl} = \exp\inparen{-\frac{1}{2}V(n(X_{s(j)}-Y_{t(l)}))}$.
Then
\begin{align*}
    \frac{1}{n}\sum_{i=1}^n \cF_i(1/2) &= \frac{1}{n}\sum_{i=1}^n  \frac{1}{KL} \log \perm (A^i) \\
    &\le \frac{1}{n} \sum_{i=1}^n \frac{1}{KL} \sum_{j=i_-}^{i_+}
    \log \sum_{l=i_-}^{i_+} \exp \inparen{-\frac{1}{2} 
    V(n(X_{s(j)} - Y_{t(l)}))} \\
    &\le \frac{1}{n} \sum_{i=1}^n \frac{1}{KL} \sum_{j=i_-}^{i_+}
    \log \sum_{l=1}^{n} \exp \inparen{-\frac{1}{2} 
    V(n(X_{s(j)} - Y_{t(l)}))} \\
    &\le \frac{C}{n} \sum_{i=1}^n
    \log \sum_{l=1}^{n} \exp \inparen{-\frac{1}{2} 
    V(n(X_{s(i)} - Y_{t(l)}))}.
\end{align*}
The right side is identical to
\eqref{eq:pairwiseGenerative_hard_meanDisplacement_afterPerm}, so applying the
bound in Lemma \ref{lemma:mean-potential-O(1)-whp} we obtain
\begin{align}
    \frac{1}{n}\sum_{i=1}^n \cF_i(1/2) \le C \text{ w.h.p.}
    \label{eq:upper_bound_Fi_1/2}
\end{align}

A lower bound of $ \frac{1}{n}\sum_{i=1}^n \cF_i(1)$ is obtained by
lower-bounding the sum over bijections
$\pi_{st} \colon [i_-,i_+] \to [i_-,i_+] $ by the single summand $\pi_{st}=\text{id}$:
\begin{align}
    \frac{1}{n}\sum_{i=1}^n \cF_i(1) &\ge \frac{1}{n} \sum_{i=1}^n \frac{1}{KL} \log \exp 
    \inparen{-\sum_{j=i_-}^{i_+} V\inparen{n(X_{s(j)} - Y_{t(j)})}}\nonumber \\
    &\ge  {-}\frac{1}{n} \sum_{i=1}^n \frac{1}{KL}
    \sum_{j=i_-}^{i_+} V\inparen{n(X_{s(j)} - Y_{t(j)})}
    \ge {-}\frac{C}{n} \sum_{i=1}^n V\inparen{n(X_{s(i)} - Y_{t(i)})}.
    \label{eq:lower_bound_Fi_1_by_avg_V_sorted}
\end{align}
By Assumption \ref{asmpt:pairwiseGenerative}\eqref{asmpt:pairwiseGenerative_V},
there exists a constant $C \geq 1+\delta$ for which
\begin{align}
    \frac{1}{n} \sum_{i=1}^n V\inparen{n(X_{s(i)} - Y_{t(i)})}
&\le C+\frac{C}{n}
\sum_{i=1}^n |n(X_{s(i)} - Y_{t(i)})|^C\nonumber\\
    &\overset{(*)}{=} C+\frac{C}{n}\min_{\pi \in \SSS_n} \sum_{i=1}^n
|n(X_{i} - Y_{\pi(i)})|^C
    \le C+\frac{C}{n} \sum_{i=1}^n |n(\bar X_i - \bar Y_i)|^C.
\label{eq:lower_bound_Fi_1_by_g}
\end{align}
Here $(*)$ uses a rearrangement inequality for the convex function
$g(x)=(n|x|)^C$. (For any $x_1 \le x_2$ and $y_1 \le y_2$,
let $r= x_1 - y_2,\; s=x_2 - y_1,\; p = x_1 - y_1,\; q = x_2 - y_2$.
Then $p+q = r+s$ and $p,q \in [r,s]$, implying there exists $\alpha \in [0,1]$
such that
$p = \alpha r + (1-\alpha) s,\,q = (1-\alpha) r + \alpha s$. So
$g(x_1-y_1) + g(x_2-y_2)
=g(p)+g(q) \le g(r)+g(s)=g(x_1-y_2)+g(x_2-y_1)$.)
Recalling that $\eps_i:=n(\bar Y_i - \bar X_i)$ satisfies the tail bound
\eqref{eq:epstailbound}, the right side of \eqref{eq:lower_bound_Fi_1_by_g} is
bounded above by a constant $C'>0$ w.h.p. Applying this
in \eqref{eq:lower_bound_Fi_1_by_avg_V_sorted} shows
\[\frac{1}{n}\sum_{i=1}^n \cF_i(1) \ge {-}C \text{ w.h.p.}\]
Together with \eqref{eq:upper_bound_Fi_1/2}, this implies the desired bound
$\frac{1}{n}\sum_{i=1}^n \cD_n^i \leq C'$ w.h.p.
\end{proof}

\subsubsection{Regularity of point processes and locality of the Gibbs measure}

Fix (large) locality parameters $K,L>0$ and recall the lower bound
$\Lambda(x) \geq \Lambda_{\min}$, for each $l \in [n]$. We repeat the event definitions from Section \ref{sec:proofOverview} for convenience. Define
the $(X,Y)$-dependent event 
\begin{align*}
\cA_l&=\left\{\sup_{j:|j-(l+\frac{1}{2})| \le L} \{n|X_{s(j)}-X_{s(l)}|,
n|Y_{t(j)}-Y_{t(l)}|\} \leq \frac{3L}{2\Lambda_{\min}},
\; n\!\abs{X_{s(l)} - Y_{t(l)}} \leq L\right\}.
\end{align*}
Conditional on $(X,Y)$, define the $\pi$-dependent events
\begin{align*}
\cC_l &\equiv \cC_l(\pi)=\left\{|k-m| \leq L \text{ for all }
(k,m) \in \Gamma_{l+\frac{1}{2}}(\pi)\right\},\\
\cL_l &\equiv \cL_l(\pi) \equiv \cA_l \cap \cC_l(\pi).
\end{align*}
Here $\cC_l(\pi)$ is a ``locality'' event that all matches in the boundary set
$\Gamma_{l+\frac{1}{2}}(\pi)$ are of short range ($\leq L$)
in terms of the sorted indices of the data values.

Define the event, for each $i \in [n]$,
\begin{align*}
\cG_i(\pi)&=\inbraces{\sum_{k=1}^{KL}
\1\{\cL_{i-k}(\pi)\} \geq \frac{2KL}{3}}
\cap \inbraces{\sum_{k=0}^{KL-1}
\1\{\cL_{i+k}(\pi)\} \geq \frac{2KL}{3}}.
\end{align*}
We write $\cG_i^c(\pi)$ for its complement, and take the convention that
$\cG_i(\pi)$ does not hold if $i \leq KL$ or $i>n-KL$.
In this section, we show the following lemma.

\begin{lemma}
\label{lemma:locality_G_whp}
For a constant $C>0$ not depending on $K,L$, w.h.p.
\[\frac{1}{n}\sum_{i=1}^n P(\cG_i^c(\pi) \mid X,Y) \leq CL^{-\delta},
\qquad \frac{1}{n}\sum_{i=1}^n P(\cG_i^c(\pi) \mid X,Y,\Gamma_{i \pm
(KL+\frac{1}{2})}(\pi)=\varnothing) \leq CL^{-\delta}.\]
\end{lemma}

To first show that $\cA_i$ holds with high probability for most sites $i \in
[n]$, we use the following tail bound for order statistics under a bounded
density.
\begin{lemma}
\label{thm:tail_bound_order_statistics_bdd_density}
For any $1\le i<j\le n$,
\begin{align*}
    \P\inparen{X_{s(j)}-X_{s(i)} \le \frac{j-i}{2\Lambda_{\max}n}} 
    \le \exp\inparen{-\frac{j-i}{8}}, 
    \quad \P\inparen{X_{s(j)}-X_{s(i)} \ge \frac{3(j-i)}{2\Lambda_{\min}n}} 
    \le \exp\inparen{-\frac{j-i}{8}}.
\end{align*}
\end{lemma}
\begin{proof}
Let $F$ be the CDF of $X$ and set 
$U_i \coloneqq F(X_i)$. Then $U_i \sim \Unif[0,1]$ and 
$U_{s(1)} < \cdots < U_{s(n)}$.
Moreover, \[
 U_{s(j)}-U_{s(i)} \sim \textnormal{Beta}(j-i, n-(j-i)+1).
\]
Note that $X_{s(i)} = F^{-1}\inparen{U_{s(i)}}$,
and since $\Lambda(x)$ is bounded from below and above, $\frac{1}{\Lambda_{\max}} \le \inparen{F^{-1}}'(u) \le \frac{1}{\Lambda_{\min}}$. Thus for any $u,v \in [0,1]$,
    \[\frac{1}{\Lambda_{\max}}\abs{u-v} 
\le \abs{F^{-1}(u) -F^{-1}(v)}
\le \frac{1}{\Lambda_{\min}}\abs{u-v}.\]
It follows that
\begin{align*}
   \P\inparen{X_{s(j)}-X_{s(i)} \le \frac{j-i}{2\Lambda_{\max}n}} 
   &\le \P\inparen{U_{s(j)}-U_{s(i)} \le \frac{j-i}{2n}} \\
   &\le \exp\inparen{-\frac{(i-j)^2/4n^2}{2(j-i)(n-(j-i)+1)/n^3}}
   \le \exp\inparen{-\frac{j-i}{8}}
\end{align*}
where the second inequality is a tail bound for the
Beta distribution \cite[Theorem 1]{skorski2023bernstein}.
Similarly,
\begin{equation*}
   \P\inparen{X_{s(j)}-X_{s(i)} \ge \frac{3(j-i)}{2\Lambda_{\min}n}} 
   \le \P\inparen{U_{s(j)}-U_{s(i)} \ge \frac{3(j-i)}{2n}} 
   \le \exp\inparen{-\frac{j-i}{8}}.
   \qedhere
\end{equation*}
\end{proof}

The following two lemmas then control the fraction of sites $i \in [n]$ that do
not satisfy $\cA_i$.
\begin{lemma}
\label{lemma:finiteLine_card_B1(D)_bound}
Define
\begin{align*}
\cB_1 &\equiv \cB_1(L) = \inbraces{i \in [n-L]\colon X_{s(i+L)} - X_{s(i)} >
\frac{3L}{2\Lambda_{\min} n}},\\
\cB_1' &\equiv \cB_1'(L) = \inbraces{i \in [n-L]\colon Y_{t(i+L)} - Y_{t(i)} > \frac{3L}{2 \Lambda_{\min} n}}.
\end{align*}
Then
\[\frac{1}{n}\max(|\cB_1|,|\cB_1'|) \le 2 e^{-L/8} \text{ w.h.p.}\]
\end{lemma}
\begin{proof}
Note that the cardinality $|\cB_1|$ as a function of the unordered points
$\bar X_1,\ldots,\bar X_n$ satisfies the bounded differences property:
\begin{align*}
    \abs{|\cB_1(\bar X_1, \cdots, \bar X_i, \cdots, \bar X_n)| - 
    |\cB_1(\bar X_1, \cdots, \bar X_i', \cdots, \bar X_n)|}
    \le 2L
\end{align*}
for any $X_1, \cdots, X_i,X_i' \cdots, X_n \in [0,1]$. Therefore,
by the bounded differences inequality, for any $t > 0$,
\begin{align*}
    \P\inparen{
    \frac{|\cB_1| - \E|\cB_1|}{n}
    \ge t
    } \le
    \exp\inparen{- \frac{n t^2}{2L^2}}.
\end{align*}
Furthermore, by Lemma 
\ref{thm:tail_bound_order_statistics_bdd_density},
\begin{align*}
    \frac{1}{n}\E|\cB_1|
    = \frac{1}{n} \sum_{i=1}^{n-L} \P\inparen{
    X_{s(i+L)} - X_{s(i)} > \frac{3L}{2\Lambda_{\min}n} } 
       \le e^{-L/8}.
   \end{align*}
Taking $t=e^{-L/8}$ shows the bound for $\cB_1$, and the bound for $\cB_1'$ is
the same.
\end{proof}

\begin{lemma}
\label{lemma:finiteLine_card_B2(D)_bound}
Define
\begin{align*}
\cB_2 \equiv \cB_2(L) = \inbraces{i \in [n]: n \abs{X_{s(i)}-Y_{t(i)}}>L}.
\end{align*}
For some constants $C,c>0$ independent of $L$,
\begin{align*}
    \frac{1}{n}|\cB_2| \le Ce^{-cL} \text{ w.h.p.}
\end{align*}
\end{lemma}
\begin{proof}
Note that $\inbraces{n \abs{X_{s(i)}-Y_{t(i)}} \geq L} = \cE_1^i \cup \cE_2^i$ 
where
\begin{align}
\label{eq:closeness_rank_statistics_XY_equivalent_events}
    \cE_1^i = \inbraces{\left|\inbraces{j: Y_j < X_{s(i)} - \frac{L}{n}} \right|\geq i },
    \qquad \cE_2^i = \inbraces{\left|\inbraces{j: 
    Y_j>X_{s(i)} + \frac{L}{n} }\right| \ge n-i+1}
\end{align}
For $\cE_1^i$, we have
\begin{align}
    \frac{1}{n} \sum_{i=1}^n \boldsymbol{1}\{\cE_1^i\}
    &= \frac{1}{n} \sum_{i=1}^n \boldsymbol{1} \inbraces{
    \sum_{j = 1}^n \boldsymbol{1} \inbraces{
    Y_{(\pi^*)^{-1} \circ s(j)} < X_{s(i)} - \frac{L}{n}
    } \ge i
    } \notag\\
&    \le \frac{1}{n} \sum_{i=1}^n \boldsymbol{1} \inbraces{
    \sum_{j = i}^n \boldsymbol{1} \inbraces{
    Y_{(\pi^*)^{-1} \circ s(j)} < X_{s(i)} - \frac{L}{n}
    } \ge 1
    } \notag\\
    &\le \frac{1}{n} \sum_{i=1}^n 
    \sum_{j = i}^n \boldsymbol{1} \inbraces{
    Y_{(\pi^*)^{-1} \circ s(j)} < X_{s(i)} - \frac{L}{n}
    } \label{eq:cE1bound}\\
    &= \frac{1}{n} \sum_{i=1}^n 
    \sum_{j = i}^n \boldsymbol{1} \inbraces{
    Y_{(\pi^*)^{-1} \circ s(j)} - X_{s(j)} < 
    \underbrace{X_{s(i)} - X_{s(j)}}_{\le 0} - \frac{L}{n}
    } \notag\\
    &\le (A) + (B) + (C)\notag
\end{align}
where 
\begin{align*}
    (A) &\coloneqq \frac{1}{n} \sum_{i=1}^n 
    \sum_{j \colon i \leq j<i+L} \boldsymbol{1} \inbraces{
    Y_{(\pi^*)^{-1} \circ s(j)} - X_{s(j)} <  - \frac{L}{n}
    },\\
    (B) &\coloneqq \frac{1}{n} \sum_{i=1}^n 
    \sum_{j \colon j \geq i + L} \boldsymbol{1} \inbraces{
     X_{s(j)} - X_{s(i)} \le \frac{j-i}{2\Lambda_{\max}n}
    },\\
    (C) &\coloneqq  \frac{1}{n} \sum_{i=1}^n 
    \sum_{j \colon j \geq i + L} \boldsymbol{1} \inbraces{
    n \inparen{ Y_{(\pi^*)^{-1} \circ s(j)} - X_{s(j)} }
    \le -\frac{j-i}{2\Lambda_{\max}} - L}.
\end{align*}
Setting $\eps_i=n(\bar Y_i-\bar X_i)=n(Y_i-X_{\pi^*(i)})$, note that
\[(A)=\frac{1}{n} \sum_{i=1}^n 
    \sum_{j \colon i \leq j<i+L} \boldsymbol{1} \inbraces{
    \eps_{(\pi^*)^{-1} \circ s(j)} <  - L}
\le 
    \frac{L}{n} \sum_{i=1}^n  \boldsymbol{1} \inbraces{
    \varepsilon_i <  - L}.\]
By the tail bound \eqref{eq:epstailbound} and a binomial tail inequality, we
have
\begin{align}
\label{eq:high_prob_bound_termA_noise}
    (A) \le CL e^{-cL^{1+\delta}} \leq C'e^{-c'L} \quad \text{w.h.p.}
\end{align}
For $(B)$, fixing any constant $\kappa>1$, note by Lemma 
\ref{thm:tail_bound_order_statistics_bdd_density} that
\begin{equation}\label{eq:quantilebound}
    \sum_{i=1}^n 
    \sum_{j \colon j-i \ge (\log n)^\kappa} \P\inparen{ 
     X_{s(j)} - X_{s(i)} \le \frac{j-i}{2 \Lambda_{\max} n}
    } \le n^2 \exp\inparen{{-}(\log n)^{\kappa}/8}<\exp({-}c(\log n)^\kappa).
\end{equation}
Thus, with probability $1-\exp({-}c(\log n)^\kappa)$,
\begin{align*}
    (B) = \frac{1}{n} \sum_{i=1}^n 
    \sum_{j \colon j-i \in [L,(\log n)^\kappa)} \boldsymbol{1} \inbraces{
     X_{s(j)} - X_{s(i)} \le \frac{j-i}{2 \Lambda_{\max} n}
    }.
\end{align*}
This quantity satisfies the bounded differences property in $\bar
X_1,\ldots,\bar X_n$, where changing any $\bar X_i$ may change the value by
at most $\frac{2(\log n)^\kappa}{n}$. By
Lemma \ref{thm:tail_bound_order_statistics_bdd_density},
the expectation of this quantity is at most
\[\frac{1}{n}\sum_{i=1}^n \sum_{j-i \in [L, (\log n)^\kappa)} 
\exp\inparen{-\frac{j-i}{8}} \le Ce^{-cL}.\]
Thus, together with the bounded differences inequality, we obtain
\begin{align}
\label{eq:high_prob_bound_termC}
(B) \le Ce^{-cL} \text{ w.h.p.} 
\end{align}
Similarly, writing again $\eps_i=n(Y_i-X_{\pi^*(i)})$
and noting that the tail bound \eqref{eq:epstailbound} implies, for any
$\kappa>1$,
\begin{align*}
    \sum_{i=1}^n 
    \sum_{j \colon j-i \geq (\log n)^\kappa} \P \inparen{
    \eps_{(\pi^*)^{-1} \circ s(j)} \le -\frac{j-i}{2 \Lambda_{\max}} - L
    } & \le n^2 \exp \inparen{{-}c(\log n)^{\kappa(1+\delta)}}
    \le \exp \inparen{{-}c'(\log n)^{\kappa(1+\delta)}},
\end{align*}
we have with probability $1-\exp({-}c'(\log n)^{\kappa(1+\delta)})$ that
\begin{align*}
    (C) &=  \frac{1}{n} \sum_{i=1}^n 
    \sum_{j \colon j-i \in [L,(\log n)^\kappa)} \boldsymbol{1} \inbraces{
    \eps_{(\pi^*)^{-1} \circ s(j)} \le -\frac{j-i}{2 \Lambda_{\max}} - L
    }.
\end{align*}
Conditional on $\bar X_1,\ldots,\bar X_n$
(and hence on the sorting $s:[n] \to [n]$),
this quantity satisfies the bounded differences property in
$\eps_1,\ldots,\eps_n$
since changing any $\eps_i$ will change its
value by at most $\frac{(\log n)^{\kappa}}{n}$. Bounding its expectation 
(conditional on $\bar X$) again using the tail bound \eqref{eq:epstailbound},
and applying the bounded differences inequality, we obtain 
\begin{align}
\label{eq:high_prob_bound_termD}
    (C) \le Ce^{-cL^{1+\delta}} \text{  w.h.p.}
\end{align}
Combining \eqref{eq:high_prob_bound_termA_noise},
\eqref{eq:high_prob_bound_termC}, and \eqref{eq:high_prob_bound_termD}
shows
$\frac{1}{n} \sum_{i=1}^n \boldsymbol{1}\{\cE_1^i\} \le Ce^{-cL}$ w.h.p.
For $\cE_2^i$, observe that
\[\frac{1}{n}\sum_{i=1}^n \boldsymbol{1}\{\cE_2^i\}
\leq \frac{1}{n}\sum_{i=1}^n \boldsymbol{1}\left\{\sum_{j=1}^i
\left\{Y_{(\pi^*)^{-1} \circ s(j)}>X_{s(i)}+\frac{L}{n}\right\} \geq
1\right\}.\]
Then the same argument applies to bound
$ \frac{1}{n} \sum_{i=1}^n \boldsymbol{1}\{\cE_2^i\}$.
\end{proof}

We next turn to analyzing the expected fraction of sites $l \in [n]$
that do not satisfy
$\cC_l(\pi)$ under $P(\pi \mid X,Y)$ or under a local restriction of the
posterior law.
This uses the bounds on the mean potential obtained previously 
in Lemmas \ref{lemma:mean-potential-O(1)-whp} and
\ref{lemma:local_mean_potential}.

\begin{lemma}
\label{lemma:locality_pi}
Fix constants $K,L>0$, and set
$i_-=\max(i-KL,1)$ and $i_+=\min(i+KL,n)$. Then
for a constant $C>0$ not depending on $K,L$, w.h.p.
\begin{align*}
    \frac{1}{n} \sum_{i=1}^n P\inparen{\cC_i^c(\pi)\mid X,Y}
    \le CL^{-\delta}, \quad 
    \frac{1}{n}\sum_{i=1}^n \frac{1}{KL} \sum_{l=i_-}^{i_+-1}
P(\cC_l^c(\pi) \mid X,Y,\Gamma_{i \pm (KL+\frac{1}{2})}(\pi)=\varnothing)
\leq CL^{-\delta}.
\end{align*}
\end{lemma}
\begin{proof}
For the global Gibbs measure, note that
\[\cC_i^c(\pi)=\{\text{there exists } j:|\pi_{st}(j)-j|>L,\,
j \leq i<\pi_{st}(j) \text{ or } \pi_{st}(j) \leq i<j\}.\]
Then
\begin{align*}
\frac{1}{n}\sum_{i=1}^n \1\{\cC_i^c(\pi)\}
&\leq \frac{1}{n} \sum_{i=1}^n \sum_{j=1}^n
\1\{|\pi_{st}(j)-j|>L,\,
j \leq i<\pi_{st}(j) \text{ or } \pi_{st}(j) \leq i<j\}\\
&\leq \frac{1}{n} \sum_{i,j=1}^n (A_{ij})+(B_{ij})+(C_{ij})
\end{align*}
where
\begin{align*}
(A_{ij})&=\1\{|\pi_{st}(j)-j|>L,\,
j \leq i<\pi_{st}(j) \text{ or } \pi_{st}(j) \leq i<j\}
\1\left\{|Y_{t \circ \pi_{st}(j)}-Y_{t(j)}|
\leq \frac{|\pi_{st}(j)-j|}{2\Lambda_{\max} n}\right\}\\
(B_{ij}) &=\1\{|\pi_{st}(j)-j|>L,\,
j \leq i<\pi_{st}(j) \text{ or } \pi_{st}(j) \leq i<j\}
\1\left\{|X_{s(j)}-Y_{t(j)}|
>\frac{|\pi_{st}(j)-j|}{4 \Lambda_{\max} n}\right\}\\
(C_{ij})&=\1\{|\pi_{st}(j)-j|>L,\,
j \leq i<\pi_{st}(j) \text{ or } \pi_{st}(j) \leq i<j\}
\1\left\{|Y_{t \circ \pi_{st}(j)}-X_{s(j)}|
>\frac{|\pi_{st}(j)-j|}{4 \Lambda_{\max} n}\right\}
\end{align*}
Denote $(A)=\frac{1}{n}\sum_{i,j=1}^n (A_{ij})$,
$(B)=\frac{1}{n}\sum_{i,j=1}^n (B_{ij})$, and
$(C)=\frac{1}{n}\sum_{i,j=1}^n (C_{ij})$.

Choosing any $\kappa>1$, we can bound $(A)$ by a $\pi$-independent quantity
\begin{align*}
    (A) 
    &\le \frac{2}{n}\sum_{i=1}^n \mathop{\sum_{j:j \leq i}
    \;\sum_{k:k>i}}_{|k-j|>L}
    \boldsymbol{1}\inbraces{
    \abs{Y_{t(k)} - Y_{t(j)}} 
    \le \frac{ \abs{k - j}}{2 \Lambda_{\max} n}
    } \\
    &=\frac{2}{n}\sum_{i=1}^n \mathop{\sum_{j:j \leq i}
    \;\sum_{k:k>i}}_{(\log n)^\kappa \geq |k-j|>L}
    \boldsymbol{1}\inbraces{
    \abs{Y_{t(k)} - Y_{t(j)}} 
    \le \frac{ \abs{k - j}}{2 \Lambda_{\max} n}
    }\\
    &\hspace{1in}+ \frac{2}{n}\sum_{i=1}^n \mathop{\sum_{j:j \leq i}
    \;\sum_{k:k>i}}_{|k-j|>(\log n)^\kappa}
    \boldsymbol{1}\inbraces{
    \abs{Y_{t(k)} - Y_{t(j)}} 
    \le \frac{ \abs{k - j}}{2 \Lambda_{\max} n}
    }.
\end{align*}
Applying Lemma \ref{thm:tail_bound_order_statistics_bdd_density}
and a union bound over all index tuples $(i,j,k)$,
the second term is 0 with probability $1-Ce^{-c(\log n)^\kappa}$. The first term
satisfies a bounded differences property in $\bar Y_1,\ldots,\bar Y_n$,
i.e.\ changing any $\bar Y_i$ changes its value by at most
$\frac{C(\log n)^{2\kappa}}{n}$. Furthermore, also by
Lemma \ref{thm:tail_bound_order_statistics_bdd_density}, the expectation
of the first term is at most
\begin{equation}\label{eq:meanboundA}
\frac{2}{n}\sum_{i=1}^n \mathop{\sum_{j:j \leq i}
    \;\sum_{k:k>i}}_{|k-j|>L} Ce^{-c|k-j|}
\leq \frac{C'}{n}\sum_{i=1}^n \Big(
\sum_{j \in [i-L,i]} e^{-cL}
+\sum_{j:j<i-L} e^{-c|i-j|}\Big) \leq C''e^{-c'L}.
\end{equation}
Then, applying the bounded differences inequality, we obtain
\begin{align}
\label{eq:high_prob_bound_termE}
    (A) \le C e^{-cL} \text{  w.h.p.}
\end{align}

Similarly, we can also bound $(B)$ by a $\pi$-independent quantity
\begin{align*}
    (B) \le  \frac{1}{n}\sum_{i=1}^n \mathop{\sum_{j:j \leq i}
    \;\sum_{k:k>i}}_{|k-j|>L} \boldsymbol{1}\inbraces{
    \abs{X_{s(j)} - Y_{t(j)}} 
    > \frac{ \abs{k - j}}{4 \Lambda_{\max} n}}
    +\boldsymbol{1}\inbraces{
    \abs{X_{s(k)} - Y_{t(k)}} 
    > \frac{ \abs{k - j}}{4 \Lambda_{\max} n}}
\end{align*}
The same arguments as in Lemma \ref{lemma:finiteLine_card_B2(D)_bound}
show that for any fixed $\kappa>1$, 
with probability at least $1-Ce^{-c(\log n)^\kappa}$,
all summands with $|k-j| \geq (\log n)^\kappa$ on
the right side above are 0. Then on this high probability event,
\begin{align*}
    (B) &\leq \frac{1}{n}\sum_{i=1}^n \mathop{\sum_{i-(\log n)^\kappa \leq
j,k \leq i+(\log n)^\kappa}}_{|k-j|>L}
    \boldsymbol{1}\inbraces{\abs{X_{s(j)} - Y_{t(j)}}>\frac{\abs{k-j}}
{4\Lambda_{\max}n}}\\
&\leq \underbrace{\frac{1}{n}\sum_{i=1}^n \mathop{\sum_{i-(\log n)^\kappa \leq
j,k \leq i+(\log n)^\kappa}}_{|k-j|>L}
\boldsymbol{1}\inbraces{\left|\left\{l:Y_l<X_{s(j)}-
\frac{\abs{k-j}}{4\Lambda_{\max}n}\right\}\right| \geq j}}_{:=(B_1)}\\
&\qquad\qquad+\frac{1}{n}\sum_{i=1}^n \mathop{\sum_{i-(\log
n)^\kappa \leq j,k \leq i+(\log n)^\kappa}}_{|k-j|>L}
\underbrace{\boldsymbol{1}\inbraces{\left|\left\{l:Y_l>X_{s(j)}+
\frac{\abs{k-j}}{4\Lambda_{\max}n}\right\}\right| \geq n-j+1}}_{:=(B_2)},
\end{align*}
where the last inequality decomposes this event analogously to
\eqref{eq:closeness_rank_statistics_XY_equivalent_events}. Applying the bound of
\eqref{eq:cE1bound}, we have 
\begin{align*}
     (B_1) \leq \frac{1}{n}\sum_{i=1}^n 
\mathop{\sum_{i-(\log n)^\kappa \leq
j,k \leq i+(\log n)^\kappa}}_{|k-j|>L}
     \sum_{l=j}^n \boldsymbol{1}\inbraces{
      Y_{(\pi^*)^{-1} \circ s(l)} \le X_{s(j)}
      - \frac{ \abs{k - j}}{4 \Lambda_{\max} n}
     } \le (B_1') + (B_1'') 
\end{align*}
where
\begin{align*}
    & (B_1') \coloneqq \frac{1}{n}\sum_{i=1}^n 
\mathop{\sum_{i-(\log n)^\kappa \leq
j,k \leq i+(\log n)^\kappa}}_{|k-j|>L}
     \sum_{l=j}^n \boldsymbol{1}\inbraces{
     X_{s(l)} - X_{s(j)} \le \frac{l-j}{2\Lambda_{\max}n},\;
      Y_{(\pi^*)^{-1} \circ s(l)} - X_{s(l)}  \le 
      - \frac{ \abs{k - j}}{4\Lambda_{\max}n}
     }  \\
    & (B_1'') \coloneqq \frac{1}{n} \sum_{i=1}^n 
\mathop{\sum_{i-(\log n)^\kappa \leq
j,k \leq i+(\log n)^\kappa}}_{|k-j|>L}
     \sum_{l=j}^n \boldsymbol{1}\inbraces{
      Y_{(\pi^*)^{-1} \circ s(l)} - X_{s(l)}  \le 
      - \frac{l-j}{2\Lambda_{\max}n}
      - \frac{ \abs{k - j}}{4\Lambda_{\max}n}
     }.
\end{align*}
Here, $n(Y_{(\pi^*)^{-1} \circ s(l)}-X_{s(l)})=\eps_{(\pi^*)^{-1} \circ s(l)}$.
Again with probability at least $1-Ce^{{-}c(\log n)^\kappa}$, by a
union bound, the summands defining $(B_1')$ with $l-j \geq (\log n)^\kappa$ are
0. On this high probability event,
\begin{align*}
    (B_1')=\frac{1}{n}\sum_{i=1}^n \mathop{\sum_{i-(\log n)^\kappa \leq
j,k \leq i+(\log n)^\kappa}}_{|k-j|>L}
     \sum_{l \in [j,j+(\log n)^\kappa]}
\boldsymbol{1}\inbraces{X_{s(l)} - X_{s(j)} \le \frac{l-j}{2\Lambda_{\max}n},\;
\eps_{(\pi^*)^{-1} \circ s(l)} \le {-}\frac{\abs{k-j}}{4\Lambda_{\max}}}.
\end{align*}
Conditional on $(\bar X_1,\ldots,\bar X_n)$ (and hence on the sorting $s:[n] \to
[n]$), this satisfies the bounded differences property as a
function of $\eps_1,\ldots,\eps_n$, where
changing $\eps_i$ may change the value by
at most $\frac{C(\log n)^{3 \kappa}}{n}$. By the tail bound
\eqref{eq:epstailbound}, its conditional expectation
over $\eps_1,\ldots,\eps_n$ is bounded above by
\[\frac{1}{n}\sum_{i=1}^n
\mathop{\sum_{i-(\log n)^\kappa \leq
j,k \leq i+(\log n)^\kappa}}_{|k-j|>L}
\sum_{l \in [j,j+(\log n)^\kappa]}
\boldsymbol{1}\inbraces{X_{s(l)}-X_{s(j)} \le \frac{l-j}{2 \Lambda_{\max} n}}
\cdot Ce^{-c|k-j|^{1+\delta}}.\]
This in turn satisfies the bounded differences property as a
function of $\bar X_1,\ldots,\bar X_n$, where changing $\bar X_i$ may change the
value by at most $\frac{C'(\log n)^{3\kappa}}{n}$. By
Lemma \ref{thm:tail_bound_order_statistics_bdd_density} 
and an argument similar to \eqref{eq:meanboundA}, its expectation over
$\bar X_1,\ldots,\bar X_n$ is
bounded above by $Ce^{-cL^{1+\delta}}$.
Then, applying the bounded differences inequality twice, first to show
concentration over $\eps_1,\ldots,\eps_n$ and then to show concentration
over $\bar X_1,\ldots,\bar X_n$, we obtain
\[(B_1') \le Ce^{-cL^{1+\delta}} \text{  with probability } 
1 - Ce^{-c(\log n)^\kappa}.\]
Similarly for $(B_1'')$, with probability at least 
$1 - Ce^{-c(\log n)^{\kappa}}$, the summands with $l-j \geq (\log
n)^\kappa$ are 0, so
\begin{align*}
    (B_1'')=\frac{1}{n}\sum_{i=1}^n 
\mathop{\sum_{i-(\log n)^\kappa \leq
j,k \leq i+(\log n)^\kappa}}_{|k-j|>L}
     \sum_{l \in [j,j+(\log n)^\kappa]}
     \boldsymbol{1}\inbraces{
      \eps_{\pi_*^{-1} \circ s(l)}  \le 
      - \frac{l-j}{2\Lambda_{\max}}
      - \frac{ \abs{k-j}}{4\Lambda_{\max}}
     }
\end{align*}
This satisfies the bounded differences property over
$\eps_1,\ldots,\eps_n$ conditional on $\bar X_1,\ldots,\bar X_n$.
By the tail bound \eqref{eq:epstailbound} and an argument
similar to \eqref{eq:meanboundA},
it has conditional mean at most $Ce^{-cL^{1+\delta}}$, so
\[(B_1'') \le Ce^{-cL^{1+\delta}} \text{  with probability }
1-Ce^{-c(\log n)^\kappa}.\]
The quantity $(B_2)$ is bounded analogously, and combining these bounds shows
\begin{align}
\label{eq:high_prob_bound_termB}
    (B) \le C e^{-cL^{1+\delta}} \text{  w.h.p.}
\end{align}

To bound the last term $(C)$, observe that
\begin{align*}
(C) &=\frac{1}{n}\sum_{j=1}^n |\pi_{st}(j)-j|
\cdot \1\{|\pi_{st}(j)-j|>L\}\1\left\{|Y_{\pi \circ s(j)}-X_{s(j)}|
>\frac{|\pi_{st}(j)-j|}{4 \Lambda_{\max} n}\right\}\\
&\leq 4 \Lambda_{\max} \cdot \frac{1}{n}\sum_{j=1}^n n|Y_{\pi \circ s(j)}-X_{s(j)}|
\cdot \1\left\{n|Y_{\pi \circ s(j)}-X_{s(j)}|
>\frac{L}{4\Lambda_{\max}}\right\}
\end{align*}
Note that for any nonnegative random variable $Z$, 
by H\"older's inequality
\begin{align}
\label{eq:first_mmoment_tail_estimate}
    \E[Z\1\{Z>a\}]
\leq (\E Z^{1+\delta})^{\frac{1}{1+\delta}}
\P[Z>a]^{\frac{\delta}{1+\delta}}
\leq (\E Z^{1+\delta})^{\frac{1}{1+\delta}}
\left(\frac{1}{a^{1+\delta}}\E Z^{1+\delta}\right)^{\frac{\delta}{1+\delta}}
\leq \frac{\E Z^{1+\delta}}{a^\delta}.
\end{align}
Lemma \ref{lemma:mean-potential-O(1)-whp} and the lower bound
$V(\eps)-V_{\min} \geq c|\eps|^{1+\delta}$ imply
\[\frac{1}{n}\sum_{j=1}^n E\left[\left(n|Y_{\pi \circ
s(j)}-X_{s(j)}|\right)^{1+\delta} \;\Big|\; X,Y\right] \leq C \text{ w.h.p.}\]
Thus, applying \eqref{eq:first_mmoment_tail_estimate},
\begin{equation}\label{eq:high_prob_bound_EC}
E[(C) \mid X,Y]
\leq 4\Lambda_{\max} \cdot \frac{C}{(L/4\Lambda_{\max})^{\delta}} \leq
C'L^{-\delta} \text{ w.h.p.}
\end{equation}
Combining
\eqref{eq:high_prob_bound_termE}, \eqref{eq:high_prob_bound_termB}, and
\eqref{eq:high_prob_bound_EC},
\[\frac{1}{n}\sum_{i=1}^n P(\cC_i^c(\pi) \mid X,Y)
\leq Ce^{-cL}+Ce^{-cL^{1+\delta}}+CL^{-\delta}
\leq C'L^{-\delta} \text{ w.h.p.}\]

The argument for the local Gibbs measure is similar: We want to bound
\[\frac{1}{n}\sum_{i=1}^n \frac{1}{KL} \sum_{l=i_-}^{i_+-1}
P(\cC_l^c(\pi) \mid X,Y,\Gamma_{i \pm (KL+\frac{1}{2})}(\pi)=\varnothing)\]
For each fixed $i \in [n]$, we have as above
\[\frac{1}{KL}\sum_{l=i_-}^{i_+-1} \1\{\cC_i^c(\pi)\}
\leq \frac{1}{KL}\sum_{l,j=i_-}^{i_+}
(A_{lj})+(B_{lj})+(C_{lj})\]
Note that we have the $\pi$-independent bound
\[\frac{1}{n}\sum_{i=1}^n\frac{1}{KL}
\sum_{l,j=i_-}^{i_+}
(A_{lj})+(B_{lj}) \leq C[(A)+(B)] \leq C'e^{-cL} \text{ w.h.p.}\]
by \eqref{eq:high_prob_bound_termE} and
\eqref{eq:high_prob_bound_termB}. 
To bound the last term involving $(C_{lj})$, similar to the above argument under the global Gibbs measure, we observe that
\begin{align*}
    &\frac{1}{n}\sum_{i=1}^n\frac{1}{KL}
\sum_{l,j=i_-}^{i_+}  E\insquare{(C)_{lj}
    \;\Big|\; X,Y,\Gamma_{i \pm (KL+\frac{1}{2})}(\pi) = \varnothing} \\
\le & \; \frac{1}{n}\sum_{i=1}^n \frac{1}{KL}
    \sum_{j=i_-}^{i_+} E\bigg[|\pi_{st}(j)-j|  \; \1\{|\pi_{st}(j)-j|>L\}\\
&\hspace{2in}\1\left\{|Y_{\pi \circ s(j)}-X_{s(j)}|
>\frac{|\pi_{st}(j)-j|}{4 \Lambda_{\max} n}\right\} \bigg\rvert X,Y,
\Gamma_{i \pm (KL+\frac{1}{2})}(\pi)=\varnothing\bigg] \\
\le & \; 4\Lambda_{\max} \; \frac{1}{n}\sum_{i=1}^n \frac{1}{KL}
    \sum_{j=i_-}^{i_+}
    E\insquare{n |Y_{\pi \circ s(j)}-X_{s(j)}| \1\left\{n |Y_{\pi \circ s(j)}-X_{s(j)}|
>\frac{ L}{4 \Lambda_{\max}}\right\}\bigg\rvert X,Y,\Gamma_{i \pm
(KL+\frac{1}{2})}(\pi)=\varnothing} \\
\le & \;
4\Lambda_{\max}\;  \frac{1}{n}\sum_{i=1}^n \frac{1}{KL}
    \sum_{j=i_-}^{i_+} \frac{E[(n|Y_{\pi \circ s(j)}-X_{s(j)}|)^{1+\delta}
\mid X,Y,\Gamma_{i \pm (KL+\frac{1}{2})}(\pi)=\varnothing]}{(L/4 \Lambda_{\max})^\delta}.
\end{align*}
Applying the local mean potential bound from
Lemma \ref{lemma:local_mean_potential}, this shows
\[\frac{1}{n}\sum_{i=1}^n\frac{1}{KL}
\sum_{l,j=i_-}^{i_+}  E\insquare{(C)_{lj}
    \;\Big|\; X,Y,\Gamma_{i \pm (KL+\frac{1}{2})}(\pi)=\varnothing} \le C L^{-\delta} \text{ w.h.p.}\]
and combining with the above gives
\[\frac{1}{n}\sum_{i=1}^n \frac{1}{KL} \sum_{l=i_-}^{i_+-1}
P(\cC_l^c(\pi) \mid X,Y,\Gamma_{i \pm (KL+\frac{1}{2})}(\pi)=\varnothing)
\leq Ce^{-cL}+CL^{-\delta} \leq C'L^{-\delta} \text{ w.h.p.}\]
\end{proof}

We now conclude the proof of Lemma \ref{lemma:locality_G_whp}.
\begin{proof}[Proof of Lemma \ref{lemma:locality_G_whp}]
By Markov's inequality, for each $i \in [KL+1,n-KL]$,
\begin{align*}
    \cG_i^c(\pi) &= \inbraces{\sum_{k=1}^{KL} \1\inbraces{\cL_{i-k}^c(\pi)} \ge
\frac{KL}{3}} \cup \inbraces{\sum_{k=0}^{KL-1} \1\inbraces{\cL_{i+k}^c(\pi)} \ge \frac{KL}{3}} \\
    &\le \frac{3}{KL} \sum_{j=i_-}^{i_+-1} \1\inbraces{\cL_{j}^c(\pi)}
    \le  \frac{3}{KL} \sum_{j=i_-}^{i_+-1} \Big(\1\inbraces{\cA_{j}^c} +
\1\inbraces{\cC_{j}^c(\pi)}\Big)
\end{align*}
Then for the global Gibbs measure, by
Lemmas \ref{lemma:finiteLine_card_B1(D)_bound}, 
\ref{lemma:finiteLine_card_B2(D)_bound},
and \ref{lemma:locality_pi},
\begin{align*}
   \frac{1}{n}\sum_{i=1}^n P(\cG_i^c(\pi) \mid  X, Y) \le \frac{2KL}{n} + 
   \frac{6}{n}\sum_{i=1}^n\boldsymbol{1}\{\cA_i^c\} + \frac{6}{n}\sum_{i=1}^n
P(\cC_i^c(\pi) \mid  X, Y) \le C L^{-\delta} \text{ w.h.p.}
\end{align*}
Similarly, for the local Gibbs measure, 
\begin{align*}
    & \frac{1}{n}\sum_{i=1}^n P(\cG_i^c(\pi) \mid  X, Y, \Gamma_{i\pm
(KL+\frac{1}{2})}(\pi) = \varnothing) \\
    \le& \; \frac{2KL}{n} + 
   \frac{6}{n}\sum_{i=1}^n\boldsymbol{1}\{\cA_i^c\} + \frac{1}{n}\sum_{i=1}^n
\frac{3}{KL} \sum_{j=i_-}^{i_+-1} P(\cC_j^c(\pi) \mid  X, Y,  \Gamma_{i\pm
(KL+\frac{1}{2})}(\pi) = \varnothing) \le C L^{-\delta} \text{ w.h.p.}\qedhere
\end{align*}
\end{proof}

\subsubsection{Correlation decay}
We now conclude the proof of Theorem \ref{thm:exact_alg} by using the preceding
locality estimates to show correlation decay in the posterior law.
We require the following elementary observation.

\begin{lemma}\label{lemma:markov_chain}
Let $i \in [n]$ be fixed. Under the posterior law $P(\pi \mid X,Y)$,
conditional on $\Gamma_{i+\frac{1}{2}}(\pi)=\varnothing$,
the values of $(\pi_{st}(k):k \leq i)$ are independent of
$(\pi_{st}(k):k>i)$.
\end{lemma}
\begin{proof}
Any $\pi$ for which $\Gamma_m(\pi)=\varnothing$
can be decomposed into $\pi=\pi_{\leq i} \oplus \pi_{>i}$ 
where $\pi_{\leq i}$ is a bijection on $\{1,\ldots,i\}$
and $\pi_{>i}$ is a bijection on $\{i+1,\ldots,n\}$.
The Hamiltonian for such $\pi$ decomposes as
\begin{align*}
   H(\pi)=H(\pi_{\le i}) + H(\pi_{>i}):=
\sum_{j=1}^i V\inparen{n\inparen{X_j - Y_{\pi(j)}}}+
\sum_{j=i+1}^n V\inparen{n\inparen{X_j - Y_{\pi(j)}}}
\end{align*}
Therefore,
\begin{align*}
   P\inparen{\pi \mid \Gamma_{i+\frac{1}{2}}(\pi)=\varnothing,X,Y}
   = \frac{\exp\inparen{- H(\pi_{\leq i})}}{\sum_{\pi_{\leq i}}
   \exp\inparen{- H(\pi_{\leq i})}}
   \frac{\exp\inparen{-H(\pi_{>i})}}{\sum_{\pi_{>i}}
   \exp\inparen{-H(\pi_{>i})}},
\end{align*}
implying the desired conditional independence.
\end{proof}

\begin{proof}[Proof of Theorem \ref{thm:exact_alg}]
Fix $i \in [n]$, and let $\pi,\pi'$ be two independent draws from
the posterior law $P(\pi \mid X,Y)$. Throughout this proof, we condition on $X,Y$
and write $P(\cdot)$ for $P(\,\cdot \mid X,Y)$. Define the lists of sites
\begin{align*}
\cQ_i^-(\pi)&=\left\{k \in \{i-KL,\ldots,i-1\}:\cL_k(\pi) \text{
holds}\right\},\\
\cQ_i^+(\pi)&=\left\{k \in \{i,\ldots,i+KL-1\}:\cL_k(\pi) \text{
holds}\right\},\\
\cQ_i(\pi)&=\cQ_i^-(\pi) \cup \cQ_i^+(\pi),
\end{align*}
and define similarly
$\cQ_i^-(\pi'),\cQ_i^+(\pi'),\cQ_i(\pi')$ for $\pi'$.
On the event $\cG_i(\pi) \cap \cG_i(\pi')$ we must have
\[|\cQ_i^-(\pi) \cap \cQ_i^-(\pi')| \geq \frac{KL}{3},
\qquad |\cQ_i^+(\pi) \cap \cQ_i^+(\pi')| \geq \frac{KL}{3}.\]

For any $m \in \{i-KL,\ldots,i+KL-1\}$,
let $\cF_m(\pi)$ denote the information set given by $\cQ_i(\pi)$ and
the list of boundary variables
\[\{\Gamma_{l+\frac{1}{2}}(\pi):l \in \cQ_i(\pi) \setminus [m-L,m+L]\}.\]
Let us lower bound
\[P(\Gamma_{m+\frac{1}{2}}(\pi)=\varnothing \mid \cL_m(\pi),\,\cF_m(\pi)).\]
Consider any $\pi$ such that $\cL_m(\pi)$ holds --- i.e.\ both $\cA_m$ and
$\cC_m(\pi)$ hold --- and
$\Gamma_{m+\frac{1}{2}}(\pi)=\Gamma \neq \varnothing$.
Recall that each $(k,l) \in \Gamma$ denotes an index pair for which
$X_{s(k)}$ matches with $Y_{t(l)}$, and exactly one of $k,l$ belongs to each
set $\{1,\ldots,m\}$ and $\{m+1,\ldots,n\}$. Let us write
\begin{equation}\label{eq:Gammasorted}
\Gamma=\{(k_{-N},l_{-N}),\ldots,(k_{-1},l_{-1}),(k_1,l_1),\ldots,(k_N,l_N)\}
\end{equation}
where we sort $k_{-N}<\ldots<k_{-1}<k_1<\ldots<k_N$ in
increasing order. Importantly, the number of tuples $(k,l) \in \Gamma$ such
that $k \leq m$ and $l>m$ is equal to the number of tuples
such that $k>m$ and $l \leq m$, because $m$ indicates the same
ranked index for the $X$ and $Y$ values, and $\pi$ is a bijection.
Thus there are exactly $N$ indices $k_{-N}<\ldots<k_{-1} \leq m$
and exactly $N$ indices $m<k_1<\ldots<k_N$.
We construct from $\pi$ the matching $\tilde \pi$ 
that keeps all matches except those in \eqref{eq:Gammasorted} the
same, and matches the indices $k_{-N}<\ldots<k_N$ of
\eqref{eq:Gammasorted} to $\{l_{-N},\ldots,l_N\}$ in sorted
order. Then by construction, $\Gamma_{m+\frac{1}{2}}(\tilde \pi)=\varnothing$,
so $\cL_m(\tilde \pi)$ also holds. Furthermore, since $\cL_m(\pi)$ holds,
$\tilde \pi$ has the same matches with range $\leq L$ as $\pi$,
i.e.\ $|\pi_{st}(j)-j|
\leq L$ if and only if $|\tilde \pi_{st}(j)-j| \leq L$.
Then $\cL_l(\pi)=\cL_l(\tilde \pi)$ for all $l \in [n]$, so
also $\cQ_i(\pi)=\cQ_i(\tilde \pi)$. Furthermore,
$\Gamma_{l+\frac{1}{2}}(\pi)=\Gamma_{l+\frac{1}{2}}(\tilde \pi)$
for every $l \in \cQ_i(\pi) \setminus [m-L,m+L]$, so
\[\cF_m(\pi)=\cF_m(\tilde \pi).\]
Applying the condition for each $(k,\pi_{st}(k)) \in \Gamma_{m+\frac{1}{2}}$
that
\begin{align*}
    n|X_{s(k)} - Y_{t(\widetilde{\pi}_{st}(k))}|
    \le & \; n|X_{s(k)} -X_{s(m)}|
    + n|X_{s(m)} - Y_{t(m)}|
    + n|Y_{t(m)}- Y_{t(\widetilde{\pi}_{st}(k))}| \\
    \le & \;\frac{3L}{2\Lambda_{\min}} + L +  \frac{3L}{2\Lambda_{\min}} \leq
\frac{4L}{\Lambda_{\min}}
\end{align*}
which holds by definition of $\cA_m \cap \cC_m(\pi) \cap \cC_m(\tilde \pi)=
\cL_m(\pi) \cap \cL_m(\tilde \pi)$, and
the bounds $g(|\eps|)+V_{\min} \geq V(\eps) \geq V_{\min}$, we have
\begin{align}
\label{eq:boundary_map_bound}
    \frac{\exp(-H(\pi))}{\exp(-H(\widetilde{\pi}))} 
    &= \exp\left(
    \sum_{(k, \pi_{st}(k))\in \Gamma} 
    V(n(X_{s(k)}-Y_{t(\widetilde{\pi}_{st}(k))}))
    - V(n(X_{s(k)} - Y_{t(\pi_{st}(k))}))\right) \nonumber\\
    &\le \exp\big(|\Gamma| g(4L/\Lambda_{\min})\big)
\end{align}
For each fixed such boundary condition $\Gamma$, this mapping from $\pi$ to
$\tilde \pi$ is injective, as we may recover $\pi$ by reassigning the matches
\eqref{eq:Gammasorted} defining $\Gamma$. Thus this shows that
\begin{align*}
\frac{P(\Gamma_m(\pi)=\Gamma \mid \cL_m(\pi),\cF_m(\pi))}
{P(\Gamma_m(\pi)=\varnothing \mid \cL_m(\pi),\cF_m(\pi))}
    \le \exp\big(|\Gamma| g(4L/\Lambda_{\min})\big).
\end{align*}
Summing over all possible $\Gamma \neq \varnothing$ that are compatible
with $\cL_m(\pi)$, of which there are at most $\binom{L}{N}^4
(N!)^2$ possibilities with $|\Gamma|=2N$, we then have
\begin{align*}
&\frac{1-P(\Gamma_m(\pi)=\varnothing \mid \cL_m(\pi),\cF_m(\pi))}
{P(\Gamma_m(\pi)=\varnothing \mid \cL_m(\pi),\cF_m(\pi))}
\leq \sum_{\Gamma \neq \varnothing} \exp\big(|\Gamma| g(4L/\Lambda_{\min})\big)\\
&\leq \sum_{N=1}^L {\binom{L}{N}}^4 (N!)^2 \exp(2N
[g(4L/\Lambda_{\min})]) \le \exp(CL[g(4L/\Lambda_{\min})+\log L])
\le \exp(C'L\,g(4L/\Lambda_{\min})),
\end{align*}
the last two bounds holding for constants $C>0$ by
Stirling's approximation and $C'>0$ by the lower bound
$g(|\eps|) \geq c|\eps|^\delta$.
Thus
\begin{align}
P(\Gamma_m(\pi)=\varnothing \mid \cL_m(\pi),\cF_m(\pi))
&\geq \frac{1}{1+\exp(C'L\,g(4L/\Lambda_{\min}))}\notag\\
&\geq \exp(-C''L\,g(4L/\Lambda_{\min}))=:\iota(L).
\label{eq:emptyboundarylowerbound}
\end{align}
We note that the above arguments apply equally for the posterior distribution of
$\pi$ conditioned on $\Gamma_{i\pm(KL+\frac{1}{2})}=\varnothing$, so
\begin{equation}\label{eq:emptyboundaryconditionallowerbound}
P(\Gamma_m(\pi)=\varnothing \mid \cL_m(\pi),\cF_m(\pi),\Gamma_{i \pm
(KL+\frac{1}{2})}(\pi)=\varnothing) \geq \iota(L).
\end{equation}

Now, conditional on $\cQ_i(\pi)$ and $\cQ_i(\pi')$,
let $J^+(\pi,\pi')$ denote the length of the longest sequence of common sites
$\{k_1,\ldots,k_{J^+(\pi,\pi')}\} \subset \cQ_i^+(\pi) \cap \cQ_i^+(\pi')$
where each increment $k_{j+1}-k_j \geq L$. Define the event
$\cE_j \equiv
\cE_j(\pi,\pi')=\{\Gamma_{k_j}(\pi)=\Gamma_{k_j}(\pi')=\varnothing\}$. Then
\begin{align*}
&P\Big(\exists\,m \in \cQ_i^+(\pi) \cap \cQ_i^+(\pi'):
\Gamma_m(\pi)=\Gamma_m(\pi')=\varnothing \;\Big|\;
\cQ_i^+(\pi),\cQ_i^+(\pi')\Big)\\
&\geq \underbrace{P(\cE_1 \mid \cQ_i^+(\pi),\cQ_i^+(\pi'))}_{:=\iota_1}
+\underbrace{P(\cE_2 \mid \cQ_i^+(\pi),\cQ_i^+(\pi'),\cE_1^c)}_{:=\iota_2}
P(\cE_1^c \mid \cQ_i^+(\pi),\cQ_i^+(\pi'))\\
&\hspace{1in}
+\underbrace{P(\cE_3 \mid
\cQ_i^+(\pi),\cQ_i^+(\pi'),\cE_1^c,\cE_2^c)}_{:=\iota_3}P(\cE_1^c,\cE_2^c
\mid \cQ_i^+(\pi),\cQ_i^+(\pi'))+\ldots\\
&=\sum_{j=1}^{J^+(\pi,\pi')}
\iota_j(1-\iota_1)\ldots(1-\iota_{j-1})
=1-\prod_{j=1}^{J^+(\pi,\pi')}(1-\iota_j)
\geq 1-(1-\iota(L)^2)^{J^+(\pi,\pi')},
\end{align*}
the last inequality applying $\iota_j \geq \iota(L)^2$ for each
$j=1,\ldots,J^+(\pi,\pi')$ as is implied by
\eqref{eq:emptyboundarylowerbound}. Note that on the event
$\cG_i(\pi) \cap \cG_i(\pi')$ where
$|\cQ_i^+(\pi) \cap \cQ_i^+(\pi')| \geq KL/3$, we have
$J^+(\pi,\pi') \geq K/3$. Hence
\begin{align}
    P\Big(\exists\,m \in \cQ_i^+(\pi) \cap \cQ_i^+(\pi'):
\Gamma_m(\pi)=\Gamma_m(\pi')=\varnothing \;\Big|\;
\cG_i(\pi) \cap \cG_i(\pi')\Big)
\geq 1-(1-\iota(L)^2)^{K/3}.
\label{eq:emptyboundarylowerboundpipiprime}
\end{align}
Applying \eqref{eq:emptyboundaryconditionallowerbound}
in place of \eqref{eq:emptyboundarylowerbound}, also
\begin{align*}
&P\Big(\exists\,m \in \cQ_i^+(\pi) \cap \cQ_i^+(\pi'):
\Gamma_m(\pi)=\Gamma_m(\pi')=\varnothing \;\Big|\; \cG_i(\pi) \cap \cG_i(\pi'),
\Gamma_{i \pm (KL+\frac{1}{2})}(\pi)=\varnothing\Big)\\
&\geq 1-(1-\iota(L)^2)^{K/3},
\end{align*}
and the analogous bounds hold for $\cQ_i^-(\pi) \cap \cQ_i^-(\pi')$.
Let
\begin{align*}
M_- &\equiv M_-(\pi,\pi')=\min\{m:m \in \cQ_i^-(\pi) \cap \cQ_i^-(\pi),
\,\Gamma_m(\pi)=\Gamma_m(\pi')=\varnothing\},\\
M_+ &\equiv M_+(\pi,\pi')=\max\{m:m \in \cQ_i^+(\pi) \cap \cQ_i^+(\pi),
\,\Gamma_m(\pi)=\Gamma_m(\pi')=\varnothing\},
\end{align*}
with the conventions $M_-=0$ or $M_+=0$ if no such index exists. Then this
shows
\begin{align}
&P\Big(M_- \neq 0 \text{ and } M_+ \neq 0 \;\Big|\;
\cG_i(\pi) \cap \cG_i(\pi')\Big)
\geq 1-2(1-\iota(L)^2)^{K/3},\label{eq:commonboundary1}\\
&P\Big(M_- \neq 0 \text{ and } M_+ \neq 0 \;\Big|\; \cG_i(\pi) \cap \cG_i(\pi'),
\Gamma_{i \pm (KL+\frac{1}{2})}(\pi)=\varnothing\Big) \geq 1-2(1-\iota(L)^2)^{K/3}.
\label{eq:commonboundary2}
\end{align}

Now let $P_i^\text{sort}$ and $\widehat P_i^\text{sort}$ be the posterior
probabilities for the match of $X_{s(i)}$ under the actual posterior and under
the local posterior of Algorithm \ref{alg:exact}. Then,
denoting by $\pi,\pi'$ the two i.i.d.\ draws from $P(\pi \mid X,Y)$,
\begin{align*}
\TV\big(\widehat P_i^\text{sort},P_i^\text{sort}\big)
&=\frac{1}{2}\sum_{j=1}^n \Big|P(\pi_{st}(i)=j \mid \Gamma_{i \pm (KL+\frac{1}{2})}(\pi)=\varnothing)
-P(\pi_{st}'(i)=j)\Big|\\
&\leq \frac{1}{2}\sum_{j=1}^n \Bigg|
\sum_{m \neq 0}\sum_{m' \neq 0}
P(\pi_{st}(i)=j,M_-=m,M_+=m' \mid
\Gamma_{i \pm (KL+\frac{1}{2})}(\pi)=\varnothing)\\
&\hspace{2in}-P(\pi_{st}'(i)=j,M_-=m,M_+=m')\Bigg|\\
&\hspace{0.5in}+P(M_-=0 \text{ or } M_+=0)
+P(M_-=0 \text{ or } M_+=0 \mid \Gamma_{i \pm (KL+\frac{1}{2})}(\pi)=\varnothing).
\end{align*}
For any fixed $m,m' \neq 0$, note that
\begin{align*}
&P(\pi_{st}(i)=j \mid M_-=m,M_+=m',
\Gamma_{i \pm (KL+\frac{1}{2})}(\pi)=\varnothing)
=P(\pi_{st}(i)=j \mid M_-=m,M_+=m')\\
&=P(\pi_{st}'(i)=j \mid M_-=m,M_+=m')
=P(\pi_{st}'(i)=j \mid M_-=m,M_+=m',\Gamma_{i \pm (KL+\frac{1}{2})}(\pi)=\varnothing),
\end{align*}
where the first and last equalities hold because
$\pi(i)$ and $\pi'(i)$ are independent of $\Gamma_{i \pm (KL+\frac{1}{2})}(\pi)$
conditional on $M_-=m \neq 0$ and $M_+=m' \neq 0$ (by Lemma \ref{lemma:markov_chain}), and the second equality holds
by exchangeability of $(\pi,\pi')$ under this conditional law. Then
\begin{align*}
\TV\big(\widehat P_i^\text{sort},P_i^\text{sort}\big)
&\leq \frac{1}{2}\sum_{j=1}^n \Bigg|
\sum_{m \neq 0}\sum_{m' \neq 0}
P(\pi_{st}'(i)=j,M_-=m,M_+=m' \mid \Gamma_{i \pm
(KL+\frac{1}{2})}(\pi)=\varnothing)\\
&\hspace{2in}-P(\pi_{st}'(i)=j,M_-=m,M_+=m')\Bigg|\\
&\hspace{0.5in}+P(M_-=0 \text{ or } M_+=0)
+P(M_-=0 \text{ or } M_+=0 \mid \Gamma_{i \pm (KL+\frac{1}{2})}(\pi)=\varnothing)\\
&\leq \frac{1}{2}\sum_{j=1}^n \Big|
P(\pi_{st}'(i)=j \mid \Gamma_{i \pm (KL+\frac{1}{2})}(\pi)=\varnothing)
-P(\pi_{st}'(i)=j)\Big|\\
&\hspace{0.5in}+2P(M_-=0 \text{ or } M_+=0)
+2P(M_-=0 \text{ or } M_+=0 \mid \Gamma_{i \pm (KL+\frac{1}{2})}(\pi)=\varnothing)\\
&=2P(M_-=0 \text{ or } M_+=0)
+2P(M_-=0 \text{ or } M_+=0 \mid \Gamma_{i \pm (KL+\frac{1}{2})}(\pi)=\varnothing),
\end{align*}
the last equality holding since $\pi'$ is independent of $\pi$. Finally, the
bounds (\ref{eq:commonboundary1}--\ref{eq:commonboundary2}) show
\begin{align*}
P(M_-=0 \text{ or } M_+=0)
&\leq 2(1-\iota(L)^2)^{K/3}+P(\cG_i(\pi)^c)+P(\cG_i(\pi')^c)\\
P(M_-=0 \text{ or } M_+=0 \mid \Gamma_{i \pm (KL+\frac{1}{2})}(\pi)=\varnothing)
&\leq 2(1-\iota(L)^2)^{K/3} \\
&\quad+P(\cG_i(\pi)^c
\mid \Gamma_{i \pm (KL+\frac{1}{2})}(\pi)=\varnothing)+P(\cG_i(\pi')^c).
\end{align*}
Then applying Lemma \ref{lemma:locality_G_whp} to average over
sites, w.h.p.\
\[\frac{1}{n}\sum_{i=1}^n
\TV\big(\widehat P_i^\text{sort},P_i^\text{sort}\big)
\leq C(1-\iota(L)^2)^{K/3}+CL^{-\delta}\]
which is equivalent to the statement of the theorem.
\end{proof}

\subsection{Proof of Theorem \ref{thm:exact_asymptotics}}

\subsubsection{Local weak convergence to a limiting point process}
\label{section:PPP_conv_exact}

We recall some definitions about weak convergence of point processes; for a
fuller discussion, we refer to \cite[Chapter 3]{resnick2008extreme}, or \cite{lastPenrose2018lectures, daley2003introductionVol1,daley2008introductionVol2}.

Consider a metric space $E$ with its Borel $\sigma$-algebra $\cE$. A \emph{point
measure} on $E$ is a measure $m$ such that there exists a countable collection
of points $\inbraces{x_i \in E : i \in I \subseteq \NN}$ with $m = \sum_{i \in
I} \delta_{x_i}$.
Let $\mathbf{M} = \mathbf{M}(E)$ be the set of all point
measures on $E$. Let $\cM = \cM(E)$ be the smallest $\sigma$-algebra making all
evaluation maps $m \mapsto m(A)$ measurable for all $A \in \cE$.
A \emph{point process} $\eta$ is a random element of $(\mathbf{M}, \cM)$
(i.e.\ measurable map from some underlying probability space $(\Omega,\cF)$ to
$(\mathbf{M},\cM)$).
We may identify $\eta$ with a countable
collection of random points $\inbraces{X_i: i \in I}$ in $E$, so that $\eta =
\sum_{i \in I} \delta_{X_i}$. The \emph{intensity measure} $\mu$ of a point
process $\eta$ is the measure on $E$ given by $\mu(A) = \EE \eta(A)$ for all $A
\in \cE$. We say that $\eta$ is \emph{locally-finite} if $\PP\inparen{\eta(A) <
\infty} = 1$ for every bounded $A \in \cE$.
\begin{definition}\label{def:PPP}
    Let $\mu$ be a non-atomic measure on $E$. A Poisson point process with
intensity measure $\mu$ is a point process $\eta$ on $E$ so that for any
collection of pairwise disjoint sets $A_1,\dots,A_k \subseteq E$, the variables $\eta(A_1),\dots,\eta(A_k)$ are independent with $\eta(A_\ell) \sim \Poisson(\mu(A_\ell))$. We write $\eta \sim \PPP(\mu)$.

    In the case when $E$ is Euclidean and the intensity measure is of the form
$\mu(\cdot) = \lambda \cdot \leb(\cdot)$, where $\lambda \geq 0$ is a rate
parameter and $\leb$ is Lebesgue measure, we say that $\eta$ is a
\emph{homogeneous} Poisson point process with rate $\lambda$ and write
$\eta \sim \PPP(\lambda)$.
\end{definition}

\begin{definition}\label{def:LaplaceFunctional}
    Let $\eta$ be a point process on $E$, and let
$\RR_+(E)$ denote the space of nonnegative real-valued measurable functions on
$E$.
The Laplace functional of $\eta$ is the mapping $\Psi_\eta : \RR_+(E) \to [0,1]$ given by $
        \Psi_\eta(u) = \EE\, { \exp\inbraces{-\int_{E} u(x) \, \eta(\ud x) }  }
    $,
    where the expectation is over the randomness of $\eta$.
\end{definition}
It is a standard result that a point process $\eta$ has the law $\eta \sim \PPP(\mu)$ if and only if 
\begin{align}
    \Psi_\eta(u) = \exp \inbraces{ - \int \insquare{1 - e^{-u(x)}} \mu(\ud x) }  , \quad u \in \RR_+(E).
\end{align}
When $\eta \sim \PPP(\mu)$ and $\mu$ is itself a random measure, the law of $\eta$ is a \emph{mixed Poisson point process} (or more generally a Cox process). The Laplace functional satisfies 
\begin{align}\label{eq:mixedPPPlaplace}
    \Psi_\eta(u) = \EE_{\mu} \exp \inbraces{ - \int \insquare{1 - e^{-u(x)}} \mu(\ud x) }, \quad u \in \RR_+(E),
\end{align}
where the expectation is over the random measure $\mu$.

Let $C_c^+(E)$ be the space of continuous, compactly supported, and nonnegative
real-valued functions on $E$. Any point process is a random Radon measure so
that $\eta(u) = \int u(x) \eta(\ud x)$ is defined a.s.\ for $u \in C_c^+(E)$.
For $\inbraces{\eta_n},
\eta $ in $\mathbf{M}$, we write $\eta_n \convVaguely \eta$ (vague convergence)
when $\eta_n(u) \rightarrow \eta(u)$ for all $u \in C_c^+(E)$. This generates
the topology of vague convergence which is metrizable for $\mathbf{M}$ to be a
Polish space \cite[Proposition 3.17]{resnick2008extreme}. This allows notions of
weak convergence of point processes to be defined. Let $C_b(\mathbf{M})$
be the space of continuous (with respect to this metric of vague convergence)
and bounded real-valued functions on $\mathbf{M}$.
\begin{definition}\label{def:convToPPP} 
    Let $\inbraces{\eta_n}$ and $\eta$ be point processes on $E$.  The sequence $\inbraces{\eta_n}$ converges weakly to $\eta$, written $\eta_n \convWeakly \eta$, if for all $\phi \in C_b(\mathbf{M})$, $\EE \phi(\eta_n) \to \EE \phi(\eta)$.
\end{definition}

\begin{proposition}{\cite[Proposition 3.19]{resnick2008extreme}}\label{prop:pointProcessConvergence_Equivalences}
Let $\inbraces{\eta_n}$ and $\eta$ be point processes on $E$. The following two assertions are equivalent: (i) $\eta_n \convWeakly \eta$ and (ii) $\Psi_{\eta_n}(u) \to \Psi_{\eta}(u)$ for all $u \in C_c^+(E)$.
\end{proposition}

For our purposes,
the following criterion for weak convergence of point processes on Euclidean spaces is useful.
\begin{proposition}\label{prop:pointProcess_convergenceByRectangles}
    Let $\inbraces{\eta_n}$ and $\eta$ be point processes on $E = \RR^2$. Suppose that
    \begin{enumerate}[(i)]
        \item For all compact sets $U \subset \RR^2$, $\sup_{n} \EE \insquare{\eta_n(U)} < \infty$, and
        \item For every rational rectangle simple function $s = \sum_{j=1}^{m}
a_j \boldsymbol{1}_{R_j}$ where $a_j \geq 0$, $m \in \NN$, and $R_j$ are
disjoint rectangles with rational endpoints, we have $\Psi_{\eta_n}(s) \to \Psi_{\eta}(s)$.
    \end{enumerate}
    Then for all $u \in C_c^+(\RR^2)$, $\Psi_{\eta_n}(u) \to \Psi_{\eta}(u)$.
\end{proposition}
\begin{proof}
    Fix a compact $U \subset \RR^2$. Let $u$ and $\widetilde{u}$ be two
nonnegative bounded measurable functions whose union of supports is contained in
$U$. We first establish a Lipschitz property for $\Psi_n$ and $\Psi$. For each
$n$, $\abs{\exp(-\eta_n(u)) - \exp(-\eta_n(\widetilde{u}))} \leq \eta_n\inparen{\abs{u - \widetilde{u}}}$.
    Taking expectations, we obtain $\sup_{n} \abs{\Psi_{\eta_n}(u) - \Psi_{\eta_n}(\widetilde{u})} \leq \norm{u - \widetilde{u}}_{\infty} \sup_{n }\EE \eta_n(U)$. Similarly we have $\abs{\Psi_{\eta}(u) - \Psi_{\eta}(\widetilde{u})} \leq \norm{u - \widetilde{u}}_{\infty} \EE \eta(U)$.

    Fix $u \in C_c^+(\RR^2)$ and let $U$ denote the compact support of $u$.
Since $u$ is continuous on $U$, it is uniformly continuous and bounded. Then for
every $\epsilon > 0$, there exists a rational rectangle simple function $s$ with
$s \geq 0$, $\text{supp}(s) \subset U$,
and $\norm{u - s}_{\infty} < \min \inbraces{\epsilon/(3 C_U), \epsilon/(3 \EE \eta (U)) }$, where $C_U = \sup_{n} \EE \eta_n(U)$. There exists $n_0(\epsilon)$ such that for all $n \geq n_0(\epsilon)$, $\abs{\Psi_{\eta_n}(s) - \Psi_{\eta}(s)} \leq \epsilon/3$. It follows from the preceeding discussion that for $n \geq n_0$,
    \begin{equation*}
        \abs{\Psi_{\eta_n}(u) - \Psi_{\eta}(u)} \leq \abs{\Psi_{\eta_n}(u) - \Psi_{\eta_n}(s)} + \abs{\Psi_{\eta_n}(s) - \Psi_{\eta}(s)} + \abs{\Psi_{\eta}(u) - \Psi_{\eta}(s)} \leq \epsilon. \qedhere
    \end{equation*}
\end{proof}

Using the above characterization, we establish a notion of local weak
convergence of the point process $\{n(\bar X_i-\bar X_I,\bar Y_i-\bar X_I)\}_{i=1}^n$, conditional on $\{(\bar X_i,\bar Y_i)\}_{i=1}^n$ with randomness
induced by a uniform random choice of
index $I \in [n]$. (This is analogous to local weak
convergence for random graphs \cite{aldous2004objective} where the distribution
of the local neighborhood is induced by a random choice of vertex conditional
on the realization of the graph.)

The local weak limit is given by
$\inparen{\XSetPPP_\Lambda,\YSetPPP_\Lambda,\pi^*_\Lambda}$ as 
constructed in Definition \ref{def:exactMatching_PPP}, which we identify
implicitly with the point process on $\R^2$ containing the random points
$\{(\XpointPPP,\pi_\Lambda^*(\XpointPPP)):\XpointPPP \in \XSetPPP_\Lambda\}$.

\begin{lemma}\label{lemma:ExactM_convergenceToPPP2} 
Conditional on $\cD_n = \{(\bar X_i, \bar Y_i)\}_{i=1}^n$,
let $I \in [n]$ be a uniform random index.
Then almost surely as $n \to \infty$,
\begin{align*}
    \inbraces{n(\bar X_i- \bar X_I,\bar Y_i-\bar X_I)}_{i = 1}^{n}
    \overset{\text{w}}{\longrightarrow}
\inparen{\XSetPPP_\Lambda,\YSetPPP_\Lambda,\pi_\Lambda^*}.
\end{align*}
\end{lemma}
\begin{proof}
Let $R_1,\dots,R_m$ be disjoint bounded rectangles in $\RR^2$ with $R_\ell =
B_\ell \times C_\ell$, let $a_1,\ldots,a_m \geq 0$, and
let $s=\sum_{\ell=1}^m a_\ell \1_{R_\ell}$.
Let $\eta$ denote the point process on $\R^2$ representing
$\inparen{\XSetPPP_\Lambda,\YSetPPP_\Lambda,\pi_\Lambda^*}$. We first compute
the Laplace functional $\Psi_\eta(s)$: Write $\eta=\eta_0+\bar \eta$ where
$\eta_0$ consists of the single point $(0,\pi_\Lambda^*(0)) \equiv (0,\eps)$,
$\bar \eta$ consists of all remaining points, and these are independent. Then
\begin{align*}
\Psi_\eta(s)&=
\E \exp\left({-}\int s(x,y)\eta_0(\ud x \ud y)-\int s(x,y)\bar\eta(\ud x \ud
y)\right)\\
&=\E_{\eps \sim q}[e^{-s(0,\eps)}]\cdot
\underbrace{\E\exp\left({-}\sum_{\ell=1}^m a_\ell \sum_{\XpointPPP \in \XSetPPP_\Lambda
\setminus \{0\}} \1_{\{\XpointPPP \in B_\ell\}}\1_{\{\pi_\Lambda^*(\XpointPPP)
\in C_\ell\}}\right)}_{=\Psi_{\bar \eta}(s)}.
\end{align*}
Conditional on $\XSetPPP_\Lambda$, since $\{\pi_\Lambda^*(\XpointPPP):\XpointPPP
\in \XSetPPP_\Lambda \setminus \{0\}\}$ are independent
with $\pi_\Lambda^*(\XpointPPP) \sim q(\cdot-\XpointPPP)$, we have
\begin{align*}
&\E\left[\exp\left({-}\sum_{\ell=1}^m a_\ell \sum_{\XpointPPP \in
\XSetPPP_\Lambda \setminus \{0\}} \1_{\{\XpointPPP \in
B_\ell\}}\1_{\{\pi_\Lambda^*(\XpointPPP) \in C_\ell\}}\right)
\;\Bigg|\;\XSetPPP_\Lambda \right]\\
&=\prod_{\XpointPPP \in \XSetPPP_\Lambda \setminus \{0\}}
\int \exp\left({-}\sum_{\ell=1}^m 
a_\ell\1_{\{\XpointPPP \in B_\ell\}}\1_{\{y \in C_\ell\}}\right)
q(y-\XpointPPP) \ud y.
\end{align*}
Then, since the law of $\XSetPPP_\Lambda \setminus \{0\}$ is a mixed Poisson
point process, applying \eqref{eq:mixedPPPlaplace} to take the expectation over
$\XSetPPP_\Lambda$ gives
\begin{align*}
\Psi_{\bar \eta}(s)&=\E_{X \sim \Lambda}
\exp\left\{{-}\int \left[1-
\int \exp\left({-}\sum_{\ell=1}^m 
a_\ell\1_{\{x \in B_\ell\}}\1_{\{y \in C_\ell\}}\right)
q(y-x) \ud y\right]\Lambda(X)\ud x\right\}\\
&=\E_{X \sim \Lambda}\exp\left\{{-}\Lambda(X)\iint \left[1-
\exp\left({-}\sum_{\ell=1}^m 
a_\ell\1_{\{x \in B_\ell\}}\1_{\{y \in C_\ell\}}\right)\right]q(y-x)\ud x\,\ud y
\right\}\\
&=\E_{X \sim \Lambda}\exp\left\{\Lambda(X)\sum_{\ell=1}^m
(e^{-a_\ell}-1)\iint_{R_\ell} q(y-x)\ud x\,\ud y\right\},
\end{align*}
the last equality using that $R_1,\ldots,R_m$ are disjoint subsets of $\R^2$.
In light of \eqref{eq:mixedPPPlaplace} and
Proposition \ref{prop:pointProcess_convergenceByRectangles}, note that
this shows the bivariate process
$(\XSetPPP_\Lambda \setminus \{0\},\YSetPPP_\Lambda \setminus
\{\YpointPPP_0\},\pi_\Lambda^*)$ is also a mixed Poisson point process with
random intensity measure $\mu_X(\ud x\,\ud y)=
\Lambda(X)q(y-x)\ud x\,\ud y$ and $X \sim \Lambda$. Then
\[\Psi_\eta(s)=\E_{\eps \sim q}[e^{-s(0,\eps)}] \cdot
\E_{X \sim \Lambda}\exp\left\{\Lambda(X)\sum_{\ell=1}^m
(e^{-a_\ell}-1)\iint_{R_\ell} q(y-x)\ud x\,\ud y\right\}.\]

Now let $\eta_n$ be the point process with points $\{n(\bar X_i-\bar X_I,\bar
Y_i-\bar X_I)\}_{i=1}^n$.
Define $\cN^{(i)}(R_\ell) = |\{j \in [n]:j \neq i,\,
    n(\bar X_j- \bar X_i, \bar Y_j- \bar X_i) \in R_\ell\}|$
and set 
\begin{align}\label{eq:ExactM_convergence_cNdef}
   \cN^{(i)} \coloneqq \inparen{\cN^{(i)}(R_1),
   \ldots, \cN^{(i)}(R_m)},
\qquad a=(a_1,\ldots,a_m), \qquad \eps_i=n(\bar Y_i-\bar X_i).
\end{align}
Then
\begin{align*}
\Psi_{\eta_n}(s)&:=
\E\left[\exp\left({-}\int s(x,y)\eta_n(\ud x \ud y)\right)
\;\Bigg|\;\cD_n\right]=\frac{1}{n}\sum_{i=1}^n \exp\left(
-s(0,\eps_i)-a^\top \cN^{(i)}\right)
\end{align*}
We first show that $\E \Psi_{\eta_n}(s)$ (with expectation over $\cD_n$)
converges to $\Psi_\eta(s)$, and then show that $\Psi_{\eta_n}(s)$
concentrates around its mean. Note that
\[\E\Psi_{\eta_n}(s)=\E\Big[e^{-s(0,\eps_i)}\,\E[e^{-a^\top \xi} \mid \bar
X_i]^{n-1}\Big]\]
where the right side is the same for all $i \in [n]$, and where
$\xi$ denotes a vector in
$\inbraces{e_1,\dots,e_m} \cup \inbraces{0}$ with distribution
conditional on $\bar X_i$ given by
\begin{align*}
    \PP\insquare{\xi = 0 \mid \bar X_i } = 1-\sum_{\ell=1}^m
\rho_{\ell,n}(\bar X_i), \qquad 
    \PP\insquare{ \xi = e_\ell \mid \bar X_i } = \rho_{\ell,n}(\bar X_i)
\quad\text{for } 1 \leq \ell \leq m,
\end{align*}
\begin{align}
    \rho_{\ell,n}(\bar X_i) &= \P\inparen{ n \inparen{\bar X_j- \bar X_i, \bar
Y_j- \bar X_i} \in R_\ell \mid \bar X_i} \text{ for any } j \neq i \nonumber\\
    &=
    \frac{1}{\cZ_n}\int_{(\bar X_i, \bar X_i) + \frac{R_\ell}{n}}
    p_n(x,y)\, \ud x \, \ud y \nonumber\\
    &= \frac{1}{n^2 \cZ_n} \int_{B_\ell} \ud u
    \int_{ C_\ell} \ud \eps 
    \sqrt{\Lambda\inparen{\bar X_i + \frac{x}{n}}
    \Lambda\inparen{\bar X_i + \frac{y}{n}}} \exp\inparen{
    -V(y-x)
    }.\label{eq:ExactM_convergence_rho_ell}
\end{align}
By Proposition \ref{prop:datamodel}, for any fixed $\bar X_i \in (0,1)$,
as $n \to \infty$,
\[n\rho_{\ell,n}(\bar X_i)
\to \Lambda(\bar X_i) \iint_{R_\ell} q(y-x)\ud x\,\ud y,\]
and the law of $(\bar X_i,\eps_i)$ converges to that of
$(X,\eps) \sim \Lambda \times q$. Therefore, by the dominated convergence
theorem,
\begin{align}
    \E \Psi_{\eta_n}(s)&= 
\E\Big[e^{-s(0,\eps_i)}\,\E[e^{-a^\top \xi} \mid \bar
X_i]^{n-1}\Big]\notag\\
    &=\E\left[e^{-s(0,\eps_i)}
\left(1 + \sum_{\ell=1}^m \inparen{  
    e^{-t_\ell} - 1 }  \rho_{\ell,n}(\bar X_i)\right)^{n-1}\right]\nonumber\\
    &\to
\E_{\eps \sim q}[e^{-s(0,\eps)}] \cdot \E_{X \sim \Lambda}\left[
    \exp\inparen{ \Lambda(X) \sum_{\ell=1}^m (e^{-t_\ell}-1)
    \iint_{R_\ell} q(y-x) \, \ud x \, \ud y}\right]
=\Psi_\eta(s).\label{eq:mean_laplace_transform_poisson}
\end{align}

Next, we show concentration of $\Psi_{\eta_n}(s)=n^{-1}\sum_{i=1}^n
\exp(-s(0,\eps_i)-a^\top \cN^{(i)})$ with respect to $\cD_n$.
Let $\bigcup_{\ell=1}^m R_\ell \subset W:=[-R,R]^2$ for some $0 < R < \infty$.
Define the following quantity
\begin{align*}
    D_j \coloneqq \left|\inbraces{i \in [n] \colon i \neq j,\,
    n(\bar X_j-\bar X_i,\bar Y_j-\bar X_i) \in W}\right| \; \text{for } 
    j=1,\cdots, n
\end{align*}
and the event $G_n \coloneqq \{\max_j D_j<3 \log n\}$.
Conditional on $(\bar X_j,\bar Y_j)$, $D_j \sim \text{Binomial}(n-1,p_j)$
where $p_j \le \frac{2R\Lambda_{\max}}{n}$.
Applying a union bound and Chernoff bound gives
\begin{align*}
    \P(G_n^c) = \sum_{i=1}^n  \P\inparen{D_i \ge 3 \log n} \le n^{-3}
\sum_{i=1}^n \E e^{D_i}
    \le n^{-2}e^{8R\Lambda_{\max}}.
\end{align*}
If $\{(\bar X_1,\bar Y_1),\ldots,(\bar X_n,\bar Y_n)\}$ and
$\{(\bar X_1',\bar Y_1'),\ldots,(\bar X_n',\bar Y_n')\}$ both satisfy $G_n$
and differ in only one pair, say $(\bar X_j,\bar Y_j) \neq (\bar X_j',\bar
Y_j')$, then the corresponding count vectors $\cN^{(i)}$ and ${\cN^{(i)}}'$
for $i \neq j$ differ only if $n(\bar X_j-\bar X_i, \bar Y_j-\bar X_i)$ and/or 
$n(\bar X_j'-\bar X_i',\bar Y_j'-\bar X_i')$ belong to $W$. Hence,
$\Psi_{\eta_n}(s)$ differs by at most $\frac{1+6\log n}{n}$.
Then by a conditional version of the bounded differences inequality restricted
to $G_n$ (see \cite{Combes2024}), for constants $C,c>0$,
\begin{align*}
    \P\inparen{ 
    \abs{\Psi_{\eta_n}(s) - \E [\Psi_{\eta_n}(s) \mid G_n ]} \ge \varepsilon} \le  \frac{C}{n^2} + \exp
    \inparen{-\frac{cn \inparen{\eps - \frac{C\log n}{n^2}}_{+}^2}{(\log n)^2}} 
\end{align*}
This is summable in $n$, so by the Borel-Cantelli lemma, 
$\Psi_{\eta_n}(s)-\E [\Psi_{\eta_n}(s) \mid G_n ]\to 0$ a.s. Furthermore,
\begin{align*}
    \abs{\E \Psi_{\eta_n}(s) - \E [\Psi_{\eta_n}(s) \mid G_n ]} = 
    \P(G_n^c) \abs{ \E [\Psi_{\eta_n}(s) \mid G_n ]
    - \E [\Psi_{\eta_n}(s) \mid G_n^c]} \le \frac{2 e^{8R\Lambda_{\max}}}{n^2}.
\end{align*}
Combining with \eqref{eq:mean_laplace_transform_poisson}, we obtain
$\Psi_{\eta_n}(s) \to \Psi_\eta(s)$ a.s. Thus on an event of probability 1,
$\Psi_{\eta_n}(s) \to \Psi_\eta(s)$ for every such simple function $s$
defined by rational $a_1,\ldots,a_m \geq 0$ and $R_1,\ldots,R_m$ with rational
endpoints, implying by Propositions
\ref{prop:pointProcessConvergence_Equivalences} and
\ref{prop:pointProcess_convergenceByRectangles} that $\eta_n \to
\eta$ weakly a.s.
\end{proof}

\subsubsection{Local approximation of Algorithm \ref{alg:exact}}
Algorithm \ref{alg:exact} involves a global sorting of the $X$ and $Y$ values.
In this section we consider a related procedure, Algorithm
\ref{alg:exact_localApprox}, for posterior approximation that uses knowledge of
the true bijection $\pi^*$. The main difference with Algorithm \ref{alg:exact} is that the
sorting step is replaced by a determination of which local windows to match
using a local approximation of the flow of $\pi^*$. Consequently,
the posterior probability vector $\widetilde P_i$ computed for
each $X_i$ depends only on $\{n(\bar X_j-X_i,\bar
Y_j-X_i)\}_{j=1}^n \in [-D,D]^2$, and hence is amenable to an analysis using the
local weak convergence established in the preceding section. We use this
algorithm purely as a theoretical proof device.

For each index $i \in [n]$, define
\begin{align}
\mathbf{L}_n(i) &\coloneqq \sum_{j=1}^n
\1\{X_j \leq X_i, Y_{\pi^*(j)}>Y_{\pi^*(i)}\},\label{eq:finite_n_flux_global}\\
\mathbf{R}_n(i) &\coloneqq \sum_{j=1}^n \boldsymbol{1}\inbraces{X_j>X_i, 
Y_{\pi^*(j)} \leq Y_{\pi^*(i)}}
\end{align}
Thus $\mathbf{L}_n(i)$ is the number of points $\{X_j:X_j \leq X_i\}$
that match to some point $\{Y_j:Y_j>Y_{\pi^*(i)}\}$ under $\pi^*$, and
$\mathbf{R}_n(i)$ is the number of points $\{X_j:X_j>X_i\}$
that match to some point $\{Y_j:Y_j\leq Y_{\pi^*(i)}\}$. 
Fixing a (large) locality parameter $D>0$, define their local approximations by
\begin{align}
\mathbf{L}_n^D(i) &\coloneqq 
\sum_{j=1}^n \1\inbraces{X_j \in \left[X_i-\frac{D}{n},X_i\right],\,
Y_{\pi^*(j)} \in \left(Y_{\pi^*(i)}, X_i+\frac{D}{n}\right]},\label{eq:finite_n_flux_local}\\
\mathbf{R}_n^D(i) &\coloneqq
\sum_{j=1}^n \boldsymbol{1}\inbraces{X_j \in \bigg(X_i,X_i + \frac{D}{n}\bigg], 
Y_{\pi^*(j)} \in \bigg[X_i - \frac{D}{n}, Y_{\pi^*(i)}\bigg]}.
\end{align}
We may then define the flow of $\pi^*$ at $i$, and its local approximation, as
\[\mathbf{F}_n(i) \coloneqq \mathbf{L}_n(i) - \mathbf{R}_n(i),
\qquad \mathbf{F}_n^D(i) \coloneqq \mathbf{L}_n^D(i) - \mathbf{R}_n^D(i).\]

\RestyleAlgo{ruled}
\begin{algorithm}[hbt!]
\DontPrintSemicolon
\caption{Posterior approximation in exact matching by flow and reordering}
\label{alg:exact_localApprox}
\For{$i \gets 1$ \KwTo $n$}{
    Compute $\mathbf{F}_n^D(i)=\mathbf{L}_n^D(i)-\mathbf{R}_n^D(i) \in \Z$.\;
Index the $M$ points to the left and right of $X_i$ as
$X_{\tilde s(-M)}<\ldots<X_{\tilde s(0)} \equiv X_i
<\ldots<X_{\tilde s(M)}$.\;
Index the $M-\mathbf{F}_n^D(i)$ and $M+\mathbf{F}_n^D(i)$
points to the left and right of $Y_{\pi^*(i)}$ as
$Y_{\tilde t(-M+\mathbf{F}_n^D(i))}<\ldots<Y_{\tilde t(0)} \equiv
Y_{\pi^*(i)}<\ldots<Y_{\tilde t(M+\mathbf{F}_n^D(i))}$.\;
    \uIf{any $X_{\tilde s(-M)},\ldots,X_{\tilde s(M)},Y_{\tilde
t(-M+\mathbf{F}_n^D(i))},\ldots,Y_{\tilde t(M+\mathbf{F}_n^D(i))} \notin 
[X_i-\frac{D}{n},X_i+\frac{D}{n}]$}
{Let $\widetilde P_i$ be arbitrary
(say, $\widetilde P_i(j)=1$ if $j=\pi^*(i)$ and 0 otherwise).\;}
\uElse{
    Let $\SSS$ be the set of all bijective maps $\pi:\{-M,\ldots,M\} \to
\{-M+\mathbf{F}_n^D(i),\ldots,M+\mathbf{F}_n^D(i)\}$.\;
    Compute the local posterior law over $\pi \in \SSS$ given by
    \begin{align*}
    P^{(i)}(\pi \mid X,Y) \, \propto \,
    \exp\Bigg(-\sum_{k=-M}^M V(n(X_{\tilde s(k)}-Y_{\tilde t(\pi(k))}))\Bigg).
    \end{align*}\;
Let
    \[\widetilde P_{i}(j)=\begin{cases}
    P^{(i)}(\pi(0)=k) \mid X,Y) & \text{ if } j=\tilde t(k) \text{
for some } k \in \{-M+\mathbf{F}_n^D(i),\ldots,M+\mathbf{F}_n^D(i)\} \\
    0 & \text{ otherwise}\end{cases}\]
}
}
\Return $\widetilde P_1,\ldots,\widetilde P_n$.
\end{algorithm}

\begin{lemma}\label{lemma:Ptilde_F_Ftilde}
There exist constants $C,c>0$ such that for any $D>0$,
\begin{align*}
\frac{1}{n}\sum_{i=1}^n \1\{\mathbf{L}_n^D(i) \neq \mathbf{L}_n(i)\} 
\le Ce^{-cD^{1+\delta}},
\quad \frac{1}{n}\sum_{i=1}^n \1\{\mathbf{R}_n^D(i) \neq \mathbf{R}_n(i)\} 
\le Ce^{-cD^{1+\delta}} \text{ w.h.p.}
\end{align*}
\end{lemma}
\begin{proof}
We have
\begin{align}
\frac{1}{n}\sum_{i=1}^n \boldsymbol{1}\inbraces{
    \mathbf{L}_n^D(i) \neq \mathbf{L}_n(i)}
&\leq \frac{1}{n}\sum_{i=1}^n \boldsymbol{1}\inbraces{
\exists j:X_j<X_i-\frac{D}{n},\,Y_{\pi^*(j)}>Y_{\pi^*(i)}}\notag\\
&\hspace{0.2in}+\frac{1}{n}\sum_{i=1}^n \boldsymbol{1}\inbraces{
\exists j:Y_{\pi^*(j)}>X_i+\frac{D}{n},\,X_j\leq X_i}.
\label{eq:flux_approximation1}
\end{align}
Writing $\eps_i=n(Y_{\pi^*(i)}-X_i)$,
the first term of \eqref{eq:flux_approximation1} is further upper bounded by
\begin{align}
&\frac{1}{n} \sum_{i=1}^n \boldsymbol{1}\inbraces{
\eps_i \leq {-}\frac{D}{2} \text{ or }
\exists j \colon X_j-X_i<{-}\frac{D}{n},\,\eps_j \geq {-}\frac{D}{2}+n(X_i - X_j)}
\nonumber \\
&\le \frac{1}{n} \sum_{i=1}^n \1\left\{\eps_i \leq {-}\frac{D}{2}\right\}
+\sum_{j=1}^n \frac{1}{n} \sum_{i:\,X_i>X_j+\frac{D}{n}}
\boldsymbol{1}\inbraces{\eps_j \geq {-}\frac{D}{2}+n(X_i - X_j)}
\label{eq:flux_approximation2}
\end{align}
Fixing any $\kappa>1$,
by the tail bound \eqref{eq:epstailbound} and a Binomial tail inequality,
the first term of \eqref{eq:flux_approximation2} is bounded as
\[\frac{1}{n} \sum_{i=1}^n \1\left\{\eps_i \leq {-}\frac{D}{2}\right\}
\leq Ce^{-cD^{1+\delta}} \text{ w.h.p.}\]
For the second term of \eqref{eq:flux_approximation2}, 
let $S_j$ be the set of $\floor{(\log n)^\kappa}$ values in
$\{X_1,\ldots,X_n\}$ that are larger than and closest to
$X_j$. By an argument similar to \eqref{eq:quantilebound}, for some
constants $C,c,c_0>0$, with probability $1-Ce^{-c(\log n)^\kappa}$,
all points $X_i>X_j+\frac{D}{n}$ that do not belong to $S_j$ must satisfy
$X_i>X_j+c_0(\log n)^\kappa/n$. Furthermore,
with probability at least $1-Ce^{-c(\log n)^{\kappa(1+\delta)}}$,
we have $\max_j |\eps_j|<c_0(\log n)^\kappa$. Thus, w.h.p.\ the
second term of \eqref{eq:flux_approximation2} is equal to
\[V:=\sum_{j=1}^n \underbrace{\frac{1}{n} \sum_{i:\,X_i \in S_j}
\boldsymbol{1}\inbraces{\eps_j \geq {-}\frac{D}{2}+n(X_i - X_j)}}_{:=V_j}\]
Note that by the tail bound \eqref{eq:epstailbound} for $\eps_j$,
\[p_j:=\E[V_j \mid X] \leq \frac{1}{n}\sum_{i:X_i>X_j+\frac{D}{n}}
Ce^{-c[-D/2+n(X_i-X_j)]^{1+\delta}}.\]
Thus
\begin{align*}
\sum_{j=1}^n p_j &\leq \frac{1}{n}\sum_{i,j=1}^n
\1\left\{X_i>X_j+\frac{D}{n}\right\}Ce^{-c[-D/2+n(X_i-X_j)]^{1+\delta}}\\
&=\frac{1}{n}\sum_{i<j}
\underbrace{\left(\1\left\{\bar X_i>\bar
X_j+\frac{D}{n}\right\}Ce^{-c[-D/2+n(\bar X_i-\bar X_j)]^{1+\delta}}
+\1\left\{\bar X_j>\bar X_i+\frac{D}{n}\right\}Ce^{-c[-D/2+n(\bar X_j-\bar
X_i)]^{1+\delta}}\right)}_{U_{ij}}.
\end{align*}
Consider the Hoeffding decomposition $U_{ij}=\E U_{ij}+U_i+U_j+\bar U_{ij}$
where $U_i=\E[U_{ij} \mid \bar X_i]-\E U_{ij}$ and $\E[\bar U_{ij} \mid \bar
X_i]=\E[\bar U_{ij} \mid \bar X_j]=0$. It is readily checked that for some
constants $C,c>0$,
\[\E U_{ij} \leq \frac{C}{n}e^{-cD^{1+\delta}},
\quad |U_i| \leq \frac{C}{n}e^{-cD^{1+\delta}} \text{ a.s.},
\quad \E|\bar U_{ij}| \leq \frac{C}{n}e^{-cD^{1+\delta}},
\quad |\bar U_{ij}| \leq Ce^{-cD^{1+\delta}} \text{ a.s.}.\]
Then $\frac{1}{n}\sum_{i<j}(\E U_{ij}+U_i+U_j) \leq
Ce^{-cD^{1+\delta}}$, and
the tail inequality of \cite[Proposition 2.3]{arcones1993limit} for degenerate
U-statistics shows
\[\P\left[\frac{1}{n}\sum_{i<j} \bar U_{ij}>t\right]
\leq C\exp\left(-\frac{ct}{e^{-cD^{1+\delta}}n^{-1/2}+t^{1/3}n^{-1/3}}\right),\]
hence also $\frac{1}{n}\sum_{i<j} \bar U_{ij} \leq C'e^{-c'D^{1+\delta}}$ w.h.p.
Thus, for some constants $C,c>0$,
\begin{equation}\label{eq:pjbounded}
\sum_{j=1}^n p_j \leq Ce^{-cD^{1+\delta}} \text{ w.h.p.}
\end{equation}
Conditional on $X$, note that
\[\Var[V_j \mid X]
=\frac{1}{n^2}\sum_{i:X_i \in S_j} 
\Var\left[\boldsymbol{1}\inbraces{\eps_j \geq {-}\frac{D}{2}+n(X_i -
X_j)}\;\Bigg|\;X\right] \leq \frac{p_j}{n},\]
and $|V_j| \leq (\log n)^\kappa/n$ by definition of $S_j$. Then conditonal
on $X$ and on the event \eqref{eq:pjbounded}, Bernstein's inequality also shows
\[\sum_{j=1}^n (V_j-p_j) \leq Ce^{-cD^{1+\delta}} \text{ w.h.p.}\]
Combining these bounds, we obtain that for some constants $C,c>0$,
\eqref{eq:flux_approximation2} is at most $Ce^{-cD^{1+\delta}}$ w.h.p.
The second term of \eqref{eq:flux_approximation1} is upper bounded by
\[\frac{1}{n}\sum_{i=1}^n
\boldsymbol{1}\inbraces{\exists j \colon X_j \leq X_i,\,\eps_j>n(X_i-X_j)+D},\]
and similar arguments show this is at most $Ce^{-cD^{1+\delta}}$ w.h.p.
Thus $n^{-1}\sum_{i=1}^n \1\{\mathbf{L}_n^D(i) \neq \mathbf{L}_n(i)\}
\leq Ce^{-cD^{1+\delta}}$. Finally, symmetric arguments apply to bound
$\1\{\mathbf{R}_n^D \neq \mathbf{R}_n(i)\}$, showing the lemma.
\end{proof}

\begin{corollary}\label{cor:TV_Ptilde_Phat}
Let $\widehat P_i(j)$ be the output of Algorithm \ref{alg:exact} with locality
parameter $M \geq 1$, and let $\widetilde P_i(j)$ be the
output of Algorithm \ref{alg:exact_localApprox} with locality parameters
$M \geq 1$ and $D=D(M)$ for some large enough constant $D(M)>0$. Then
there exist constants $C,c>0$ not depending on $D,M$ such that
    \[ \frac{1}{n}\sum_{i=1}^{n} \TV(\widehat{P}_i,\widetilde{P}_i) \leq
Ce^{-cD^{1+\delta}}+Ce^{-cM} \text{ w.h.p.}\]
\end{corollary}
\begin{proof}
Suppose $i=s(a)$ and $\pi^*(i)=t(b)$, i.e.\ $X_i,Y_{\pi^*(i)}$
are the $a^\text{th}$ and $b^\text{th}$ sorted values. For any bijection
$\pi:[n] \to [n]$, note that
\[\mathbf{F}_n^\pi(i):=\sum_{j=1}^n \1\{X_j \leq X_{(a)},Y_{\pi(j)}>Y_{(b)}\}
-\sum_{j=1}^n \1\{X_j>X_{(a)},Y_{\pi(j)}\leq Y_{(b)}\}=a-b,\]
because this holds for the bijection $\pi(s(1))=t(1),\ldots,\pi(s(n))=t(n)$ that
matches points in sorted order, and it is readily checked that
$\mathbf{F}_n^\pi$ is invariant under any transposition of two indices
$(\pi(i),\pi(j)) \leftrightarrow (\pi(j),\pi(i))$. Then in particular,
$\mathbf{F}_n(i)=\mathbf{F}_n^{\pi^*}(i)=a-b$. In Algorithm
\ref{alg:exact_localApprox} we have $\tilde s(0)=i=s(a)$
and $\tilde t(0)=\pi^*(i)=t(b)$, so on the event that
$\mathbf{F}_n^D(i)=\mathbf{F}_n(i)$, this implies
$\{\tilde t(-M+\mathbf{F}_n^D(i)),
\ldots \tilde t(M+\mathbf{F}_n^D(i))\}
=\{t(a-M),\ldots,t(a+M)\}$.
Thus Algorithm~\ref{alg:exact_localApprox} matches
the $M$ points to the left and right of $X_{s(a)}$ with the $M$
points to the left and right of $Y_{t(a)}$ which has the same sorted rank
as $X_{s(a)}$.

Thus, for each $i=s(a) \in [n]$, the outputs
$\widehat P_i$ and $\widetilde P_i$ coincide on the event
\begin{align*}
\{\mathbf{F}_n(i)=\mathbf{F}_n^D(i)\} 
&\cap \left\{X_{s(a-M)},X_{s(a+M)}
\in \left[X_i-\frac{D}{n},X_i+\frac{D}{n}\right]\right\}\\
&\cap \left\{Y_{s(a-M)},Y_{s(a+M)}
\in \left[X_i-\frac{D}{n},X_i+\frac{D}{n}\right]\right\}
=:\cE_0(i) \cap \cE_1(i) \cap \cE_2(i)
\end{align*}
By Lemma \ref{lemma:Ptilde_F_Ftilde}, $n^{-1}\sum_{i=1}^n \1\{\cE_0(i)^c\}
\leq Ce^{-cD^{1+\delta}}$ w.h.p.\ over $(X,Y)$.
By Lemma \ref{lemma:finiteLine_card_B1(D)_bound} applied with $M$ in place of
$L$, as long as $D>3M/2\Lambda_{\min}$, we have
$n^{-1}\sum_{i=1}^n \1\{\cE_1(i)^c\} \leq Ce^{-cM}$ w.h.p.
Similarly by Lemmas \ref{lemma:finiteLine_card_B1(D)_bound} and
\ref{lemma:finiteLine_card_B2(D)_bound}, for $D>3M/2\Lambda_{\min}+M$,
we have
$n^{-1}\sum_{i=1}^n \1\{\cE_2(i)^c\} \leq Ce^{-cM}$ w.h.p.
Combining these statements shows that
\[\frac{1}{n}\sum_{i=1}^n
\TV(\widehat{P}_i, \widetilde{P}_i) \leq \frac{1}{n} \sum_{i=1}^n \1\{\widehat
P_i \neq \widetilde P_i\} \leq Ce^{-cD^{1+\delta}}+Ce^{-cM}.\qedhere\]
\end{proof}

\subsubsection{Proof of Proposition \ref{prop:flux_finiteness_PPP}}

\begin{proof}
On the probability 1 event that $\XSetPPP_\Lambda$ and
$\YSetPPP_\Lambda$ are locally finite,
recall our indexing of these points
with $\XpointPPP_0=0$ and $\YpointPPP_0=\pi_\Lambda^*(0)$.
Fix some constant $D>0$, and define the following quantities:
\begin{align}
 \mathsf{L} \coloneqq \sum_{i \in \Z:i \le 0} \sum_{j \in \Z:j>0} \1\{\pi_\Lambda^*
(\XpointPPP_i)=\YpointPPP_j\}, \quad
\mathsf{L}^D \coloneqq 
\sum_{i \in \Z: \XpointPPP_i \in [-D,0]} \;
\sum_{j \in \Z: \YpointPPP_j \in (\YpointPPP_0,D]} \1\{\pi_\Lambda^*
(\XpointPPP_i)=\YpointPPP_j\}.\label{eq:PPP_flux_local}
\end{align}
Here $\mathsf{L}^D<\infty$ a.s.\ since
$\XSetPPP_\Lambda$ and $\YSetPPP_\Lambda$ are locally finite a.s., 
and $\mathsf{L}=\lim_{D \to \infty} \mathsf{L}^D$.

Let us first show that $\mathsf{L}<\infty$ a.s.
Recall $\mathbf{L}_n(i)$ and $\mathbf{L}_n^D(i)$ from
\eqref{eq:finite_n_flux_global} and \eqref{eq:finite_n_flux_local},
which may be written in equivalent forms in terms of
$X_i=\bar X_{\pi^*(i)}$ and $Y_i=\bar Y_i$ as
\begin{align*}
\mathbf{L}_n(i)&=\sum_{j=1}^n \1\{n(\bar X_j-\bar X_{\pi^*(i)}) \leq 0, n(\bar
Y_j-\bar Y_{\pi^*(i)})>0\},\\
\mathbf{L}_n^D(i)&=\sum_{j=1}^n 
\1\inbraces{n(\bar X_j-\bar X_{\pi^*(i)}) \in [-D,0],\,
n(\bar Y_j-\bar Y_{\pi^*(i)}) \in (0,D]}.
\end{align*}
By weak convergence of the point process (Proposition 
\ref{lemma:ExactM_convergenceToPPP2}),
conditional on $\cD_n=\{(\bar X_i, \bar Y_i)\}_{i=1}^n$
and over a uniformly random
choice of index $I \in \{1,\ldots,n\}$, $\mathbf{L}_n^D(I) \to
\mathsf{L}^D$ weakly a.s. Since 
the laws of $\mathbf{L}_n^D(I)$ and
$\mathsf{L}^D$ are integer-valued, this implies $\TV\inparen{
\mathbf{L}_n^D(I), \mathsf{L}^D \mid \cD_n} \to 0$ a.s.
Coupling by the same realization of the index $I$,
Lemma \ref{lemma:Ptilde_F_Ftilde} implies
\[\TV\inparen{\mathbf{L}_n^D(I),\mathbf{L}_n(I) \mid \cD_n}
\leq \frac{1}{n}\sum_{i=1}^n \1\{\mathbf{L}_n^D(i) \neq
\mathbf{L}_n(i)\} \leq Ce^{-cD^{1+\delta}}\]
a.s.\ for all large $n$. Therefore, for any $D_1,D_2>0$,
\begin{align*}
    &\TV\inparen{\mathsf{L}^{D_1}, \mathsf{L}^{D_2}}\\
    &\le \limsup_{n \to \infty} 
\TV\inparen{\mathsf{L}^{D_1},\mathbf{L}_n^{D_1}(I) \mid \cD_n}
    +  \TV\inparen{\mathsf{L}^{D_2}(I), \mathbf{L}_n^{D_2}(I) \mid \cD_n} 
    +  \TV\inparen{\mathbf{L}_n^{D_1}(I), \mathbf{L}_n^{D_2}(I) \mid \cD_n}\\
&\leq Ce^{-cD_1^{1+\delta}}+Ce^{-cD_2^{1+\delta}}.
\end{align*}
Thus $\{\mathsf{L}^D\}_{D \geq 1}$ forms a sequence of probability
distributions over $\{0,1,2,\ldots\}$ (excluding $\infty$)
that is Cauchy in TV. By completeness of
the space of such distributions under the TV metric,
there exists a limiting probability distribution $\mathcal{L}$ 
supported on $\{0,1,2,\ldots\}$ (excluding $\infty$) for which
$\lim_{D \to \infty} \TV\inparen{\mathsf{L}^D,\mathcal{L}} \to 0$.
Since $\mathsf{L}^D \to \mathsf{L}$ a.s.,
$\mathcal{L}$ must be the law of $\mathsf{L}$, implying that
$\mathsf{L}<\infty$ a.s. Defining $\mathsf{R} \coloneqq \sum_{i \in \Z:i > 0}
\sum_{j \in \Z:j \leq 0} \1\{\pi_\Lambda^*
(\XpointPPP_i)=\YpointPPP_j\}$, the same arguments show
$\mathsf{R}<\infty$ a.s.

In the notation of Proposition \ref{prop:flux_finiteness_PPP}, the above
quantities $\mathsf{L},\mathsf{R}$ are
$\mathsf{L}=\mathsf{L}_0(\pi_\Lambda^*)$
and $\mathsf{R}=\mathsf{R}_0(\pi_\Lambda^*)$. For each $a \in \Z$, define
\[\mathsf{F}_a(\pi_\Lambda^*)
=\mathsf{L}_a(\pi_\Lambda^*) - \mathsf{R}_a(\pi_\Lambda^*).\]
It remains to show that if
$\mathsf{F}_a(\pi_\Lambda^*)$ is finite at either $a-1$ or $a$, then it is
finite at both values and
$\mathsf{F}_{a-1}(\pi_\Lambda^*)=\mathsf{F}_a(\pi_\Lambda^*)$.
The only matchings under $\pi_\Lambda^*$
that could contribute differently to 
$\mathsf{L}_{a-1}(\pi_\Lambda^*)$ and $\mathsf{L}_{a}(\pi_\Lambda^*)$ 
(or to $\mathsf{R}_{a-1}(\pi_\Lambda^*)$ and 
$\mathsf{R}_{a}(\pi_\Lambda^*)$) are 
$(a,\pi_\Lambda^*(a))$ and $(\pi_\Lambda^{*-1}(a),a)$. There are only five cases to consider:
\begin{enumerate}[(i)]
    \item $\pi_\Lambda^*(a)<a$ and 
    $\pi_\Lambda^{*-1}(a)<a$,
    \item $\pi_\Lambda^*(a)<a$ and 
    $\pi_\Lambda^{*-1}(a)>a$,
    \item $\pi_\Lambda^*(a)>a$ and 
    $\pi_\Lambda^{*-1}(a)<a$,
    \item $\pi_\Lambda^*(a)>a$ and 
    $\pi_\Lambda^{*-1}(a)>a$,
    \item $\pi_\Lambda^*(a)=a$.
\end{enumerate}
In case (i), $(a,\pi_\Lambda^*(a))$ contributes $+1$ to 
$\mathsf{L}_{a-1}$ but this is canceled by a $+1$ contribution to
$\mathsf{R}_{a-1}$ by
$(\pi_\Lambda^{*-1}(a),a)$; the two matchings contribute nothing to 
$\mathsf{L}_a$ and $\mathsf{R}_a$.
Therefore $\mathsf{F}_{a-1}(\pi_\Lambda^*)=\mathsf{F}_a(\pi_\Lambda^*)$.
The rest of the cases are verified analogously, concluding the proof.
\end{proof}


\subsubsection{Proof of Proposition \ref{prop:exact_infvolume}}
\begin{proof}
Let $M' \geq M>0$ be two (large) locality parameters.
Let $\widehat P_i^M,\widehat P_i^{M'}$
denote the outputs of Algorithm \ref{alg:exact} applied
with parameters $M$ and $M'$.
Theorem \ref{thm:exact_alg} implies that a.s.\ for all large $n$,
\[\frac{1}{n}\sum_{i=1}^n \TV(\widehat P_i^M,\widehat P_i^{M'})
\leq \frac{1}{n}\sum_{i=1}^n \TV(\widehat P_i,\widehat P_i^M)
+\TV(\widehat P_i,\widehat P_i^{M'}) \leq o_M(1)\]
where $o_M(1)$ denotes an error that vanishes as $M',M \to \infty$.
Choose $D(M,M')>0$ sufficiently large where $D(M,M') \to \infty$ as $M \to
\infty$, and let $\widetilde P_i^{M,D},\widetilde
P_i^{M',D}$ denote the outputs of
Algorithm \ref{alg:exact_localApprox} applied with locality parameters $M,M'$
and $D=D(M,M')$. Then Corollary \ref{cor:TV_Ptilde_Phat} implies
that also a.s.\ for all large $n$,
\[\frac{1}{n}\sum_{i=1}^n \TV(\widetilde P_i^{M,D},\widetilde P_i^{M',D})
\leq o_M(1).\]

By definition of Algorithm \ref{alg:exact_localApprox},
$\TV(\widetilde P_i^M,\widetilde P_i^{M'})$ is a function of the point process
$\{n(\bar X_j-\bar X_i,\bar Y_j-\bar X_i)\}_{j=1}^n \cap [-D,D]^2$ and,
conditional on the number of such points, depends
continuously on the locations of these points. Then by the
local weak convergence of Lemma \ref{lemma:ExactM_convergenceToPPP2}
and the continuous mapping theorem for point processes
\cite[pg 152]{resnick2008extreme}, almost surely
\[\frac{1}{n}\sum_{i=1}^n \TV(\widetilde P_i^{M,D},\widetilde P_i^{M',D})
\to \E_\Lambda[\TV(Q_{M,D}^0,Q_{M',D}^0)].\]
Here $Q_{M,D}^0,Q_{M',D}^0 \in \cP(\YSetPPP_\Lambda)$ are the probability
vectors obtained by applying the procedure of Algorithm \ref{alg:exact_localApprox}
to $\{n(\XpointPPP,\pi_\Lambda^*(\XpointPPP))\}_{\XpointPPP \in
\XSetPPP_\Lambda} \cap [-D,D]^2$. These coincide with
$Q_M^0=(Q_M(\pi(0)=j))_{j \in \Z}$ and $Q_{M'}^0=(Q_{M'}(\pi(0)=j))_{j \in \Z}$
under the following event:
Recall the approximation $\mathsf{L}^D$ defined in \eqref{eq:PPP_flux_local}
for $\mathsf{L}=\mathsf{L}_0(\pi_\Lambda^*)$. Let
$\mathsf{R}^D$ denote the analogous approximation for
$\mathsf{R}=\mathsf{R}_0(\pi_\Lambda^*)$, and let
\[\mathsf{F}^D=\mathsf{R}^D-\mathsf{L}^D\]
be the corresponding approximation of $\mathsf{F}(\pi_\Lambda^*)$.
Then $Q_{M,D}^0=Q_M^0$ and $Q_{M',D}^0=Q_{M'}^0$ on the event
\[\mathsf{F}^D=\mathsf{F}(\pi_\Lambda^*)
\text{ and } \XpointPPP_{{-}M'},\ldots,\XpointPPP_{M'},
\YpointPPP_{{-}M'+\mathsf{F}^D},\ldots,\YpointPPP_{M'+\mathsf{F}^D}
\in [-D,D].\]
Since $\lim_{D \to \infty}
\mathsf{F}^D \to \mathsf{F}(\pi_\Lambda^*)$ a.s., we have
$\lim_{D \to \infty} \P_\Lambda[\mathsf{F}^D \neq \mathsf{F}(\pi_\Lambda^*)]=0$.
Since the law of $\mathsf{F}(\pi_\Lambda^*)$ does not depend on $D$, this
implies also that $\lim_{D \to \infty} \P_\Lambda[|\mathsf{F}^D| \geq D/2]=0$.
Then choosing $D=D(M,M')$ large enough --- say $D \geq 3M'$ --- a standard
Poisson tail bound shows that
\[\lim_{M' \to \infty} \P_\Lambda[\XpointPPP_{{-}M'},\ldots,\XpointPPP_{M'},
\YpointPPP_{{-}M'+\mathsf{F}^D},\ldots,\YpointPPP_{M'+\mathsf{F}^D}
\in [-D,D]]=1.\] 
Thus
\[\P_\Lambda[Q_{M,D}^0 \neq Q_M^0 \text{ or } Q_{M',D}^0 \neq Q_{M'}^0]
=o_M(1),\]
implying by the above that
\[\E_\Lambda[\TV(Q_M^0,Q_{M'}^0)]=o_M(1).\]

Thus the sequence $\{Q_M^0\}_{M \geq 1}$ is Cauchy in TV and
$\P_\Lambda$-probability, i.e.\ for any fixed $\eps>0$, there exists $M_0 \equiv
M_0(\eps)>0$ such that
\begin{align}\label{eq:PPP_UniqueLimit_UniformTVControl}
\PP_\Lambda[\TV(Q_M^0,Q_{M'}^0) \geq \eps]
\leq \frac{1}{\eps}\,\E_\Lambda[\TV(Q_M^0,Q_{M'}^0)] \leq \eps
\text{ for all } M,M' \geq M_0(\eps).
\end{align}
There is then a subsequence $M_1,M_2,\ldots$ such that
$\PP_\Lambda[\TV(Q_{M_k}^0,Q_{M_{k+1}}^0) \geq 2^{-k}] \leq 2^{-k}$,
so $$\lim_{k \to \infty} \sup_{j,j'>k} \TV(Q_{M_j}^0,Q_{M_j'}^0)=0$$ a.s.\ by the
Borel-Cantelli lemma. Thus, on an event with $\P_\Lambda$ probability 1,
$\{Q_{M_k}^0\}_{k \geq 1}$ is Cauchy in TV and hence has a limit
$\ProbOnPPP \in \cP(\YSetPPP_\Lambda)$ in TV. Then
\[\lim_{M \to \infty} \P_\Lambda[\TV(\ProbOnPPP,Q_M^0) \geq \eps]
=\lim_{M \to \infty} \lim_{k \to \infty}
\P_\Lambda[\TV(Q_{M_k}^0,Q_M^0) \geq \eps]=0,\]
so $\ProbOnPPP$ is a limit in probability of $\{Q_M^0\}_{M \geq 1}$, as desired.
For uniqueness, note simply that if $\ProbOnPPP'$ is another such limit in
probability, then for any $\eps>0$,
\[\P_\Lambda[\TV(\ProbOnPPP,\ProbOnPPP') \geq \eps]
\leq \lim_{M \to \infty}
\P_\Lambda[\TV(\ProbOnPPP,Q_M^0) \geq \eps/2]
+\P_\Lambda[\TV(\ProbOnPPP',Q_M^0) \geq \eps/2]=0,\]
so $\P_\Lambda[\ProbOnPPP=\ProbOnPPP']=1$.
\end{proof}

We now conclude the proof of Theorem \ref{thm:exact_asymptotics}.\\

\begin{proof}[Proof of Theorem \ref{thm:exact_asymptotics}]
Fix any locality parameters $M>0$ and $D>0$, and let $\widehat P_i^M$,
$\widetilde P_i^{M,D}$, and $Q_{M,D}^0$ be defined as above. We have 
\begin{align*}
 \abs{\frac{1}{n}\sum_{i=1}^{n} f(P_i,\pi^*(i))
 -\E_\Lambda[f(\ProbOnPPP,\pi_\Lambda^*(0))]}
 \le  (\text{I}) + (\text{II}) + (\text{III}) + (\text{IV}) +(\text{V}),
\end{align*}
where
\begin{align*}
    (\text{I}) &= \abs{\frac{1}{n}\sum_{i=1}^{n} f(P_i,\pi^*(i)) -
\frac{1}{n}\sum_{i=1}^{n} f(\widehat{P}^M_i,\pi^*(i))} \\
    (\text{II}) &= \abs{\frac{1}{n}\sum_{i=1}^{n} f(\widehat{P}^M_i,\pi^*(i)) -
\frac{1}{n}\sum_{i=1}^{n} f(\widetilde{P}^{M,D}_i,\pi^*(i))} \\
    (\text{III}) &= \abs{\frac{1}{n}\sum_{i=1}^{n}
f(\widetilde{P}^{M,D}_i,\pi^*(i)) - \E_\Lambda[f(Q_{M,D}^0,\pi_\Lambda^*(0))]}\\
    (\text{IV}) &= \abs{\E_\Lambda[f(Q_{M,D}^0,\pi_{\Lambda^*}(0))] - 
\E_\Lambda[f(Q_M^0,\pi_{\Lambda^*}(0))]},\\
(\text{V}) &= \abs{\E_\Lambda[f(Q_M^0,\pi_{\Lambda^*}(0))]
-\E_\Lambda[f(\ProbOnPPP,\pi_{\Lambda^*}(0))]}.
\end{align*}
Since $f(\cdot)$ is TV-continuous,
by the preceding arguments, for large $M>0$ and $D=D(M)>0$, almost surely 
for all large $n$ we have $(\mathrm{I})+(\mathrm{II}) \leq o_M(1)$,
and also $(\mathrm{IV})+(\mathrm{V}) \leq o_M(1)$.
Since $f(\widetilde{P}^{M,D}_i,\pi^*(i))$ is a bounded function
of the point process
$\{n(\bar X_j-\bar X_i,\bar Y_j-\bar X_i)\}_{j=1}^n \cap [-D,D]^2$ that varies
continuously in the locations of these points, by the local weak convergence of
Lemma \ref{lemma:ExactM_convergenceToPPP2} and the continuous mapping
theorem, we have $\lim_{n \to \infty} (\mathrm{III})=0$ a.s.
Then taking $M \to \infty$ concludes the proof.
\end{proof}

\section{Proofs for partial matching}

We now show Theorems \ref{thm:soft_alg} and \ref{thm:partial_asymptotics} on
the partial matching model. The analyses are analogous to, and simpler than,
those in the exact matching setting.
We will assume Assumption \ref{asmpt:pairwiseGenerative} and the setting of the
partial matching model \eqref{eq:partialMatching_oberservation}
throughout the lemmas and analyses of this section.

\subsection{Proof of Theorem \ref{thm:soft_alg}}

\subsubsection{Mean potential}

\begin{lemma}\label{lemma:mean_potential_partial_matching} 
There exists a constant $C>0$ such that
\begin{align*}
    \frac{1}{n}\,
    E\insquare{\sum_{i \in [N_X]:\pi(i) \neq \varnothing}
    V\inparen{n\inparen{X_i - Y_{\pi(i)}} }
    \,\Bigg\rvert\, X, Y} \leq C \text{ w.h.p.}
\end{align*}
\end{lemma}
\begin{proof}
Similar to exact matching, we consider the model 
\eqref{eq:general_model_XY_density_with_beta} with $\beta V$
in place of $V$. The Hamiltonian of the posterior distribution
in this model is
    \[H_n(\pi \mid X, Y; \beta)
    =\sum_{i \in [N_X]:\pi(i) \neq \varnothing} \beta V(n(X_i-Y_{\pi(i)}))
    -\sum_{i \in [N_X]:\pi(i)=\varnothing} U_n(X_i)
    -\sum_{j \in [N_Y]:\pi^{-1}(j)=\varnothing} U_n(Y_j).\]
Denote the corresponding free energy as
\begin{align}
    \cF_n(\beta) 
    \coloneqq \frac{1}{n} \log \cZ(\beta), 
    \qquad 
    \cZ(\beta) \coloneqq \sum_{\pi \in \SSS_n}
    \exp\inparen{-H_n\inparen{\pi \mid X, Y; \beta}}
\end{align}
where we recall that $\SSS_n$ is the space of all partial bijections from
$[N_X]$ to $[N_Y]$. Then
\begin{align}
    \cD_n(\beta) \coloneqq \frac{1}{n} 
    E_\beta \insquare{\sum_{i \in [N_X]:\pi(i) \neq \varnothing}
    V\inparen{n\inparen{X_i - Y_{\pi(i)}} }
    \,\bigg\rvert\, X, Y}={-}\cF_n'(\beta),
\end{align}
where $\cF_n(\beta)$ is convex, so the quantity to be bounded is
\begin{align}
\label{eq:relate_D_1_to_free_energy}
     \cD_n(1) 
     = - \cF_n'(1) 
     \le 2\inparen{\cF_n(1/2)-\cF_n(1)}.
\end{align}

Similar to the proof for Lemma \ref{lemma:mean-potential-O(1)-whp}
in exact matching, a lower bound for 
$\cF_n(1)$ is obtained by lower bounding 
the sum over $\pi \in \SSS_n$ by the single summand $\pi^*$:
\begin{align*}
     \cF_n(1)
     & \geq -\frac{1}{n} H_n(\pi^* \mid X, Y) \\
     & = -\frac{1}{n} 
     \sum_{i \in [N_X]:\pi^*(i) \neq \varnothing}  V(n(X_i-Y_{\pi^*(i)}))
     + \frac{1}{n} \sum_{i \in [N_X]:\pi^*(i)=\varnothing} U_n(X_i)
     + \frac{1}{n} \sum_{j \in [N_Y]:{\pi^*}^{-1}(j)=\varnothing} U_n(Y_j).
\end{align*}
From the argument of Proposition \ref{prop:datamodel}, it may be checked that
for all large $n$ and some constants $C,c>0$,
\begin{equation}\label{eq:pnuniformbounds}
C \geq \sup_{x \in [0,1]} p_n(x)/\Lambda(x) 
\geq \inf_{x \in [0,1]} p_n(x)/\Lambda(x) \geq c.
\end{equation}
Then $U_n(x) \geq \log c\sqrt{\Lambda(x)} \geq {-}C$ for all $x \in [0,1]$,
so
\begin{align*}
 \cF_n(1) &\geq -\frac{1}{n} 
     \sum_{i \in \cS_{XY}}  V(n(\bar X_i-\bar Y_i))
-\frac{C(N_X+N_Y)}{n}
\end{align*}
Writing $\eps_i=n(\bar Y_i-\bar X_i)$, for any $t>0$,
\begin{align*}
    &\P\inparen{
    \frac{1}{n} 
    \sum_{i \in \cS_{XY}}  V(n(\bar X_i-\bar Y_i)) \ge t
    }
    \le e^{-tn/2}
    \E \exp\inparen{
     \frac{1}{2}\sum_{i \in \cS_{XY}} V(\eps_i)} \\
     & = e^{-tn/2}
     \E \insquare{\inparen{
     \E \exp\inparen{V(\varepsilon)/2}
     }^{\abs{\cS_{XY}}}}
     \le e^{-tn/2}
     \E \exp\inparen{CN} 
     \le \exp\inparen{-tn/2+C'n}.
\end{align*}
Choosing $t>0$ sufficiently large shows
\[\frac{1}{n}
\sum_{i \in \cS_{XY}}  V(n(\bar X_i-\bar Y_i)) \le C \text{ w.h.p.}\]
By a standard Poisson tail bound, $N_X+N_Y \leq N \leq Cn$ w.h.p.
Therefore, we obtain
\begin{align}
\label{eq:high_prob_bound_Fn1}
\cF_n(1) \geq -C \text{  w.h.p.}
\end{align}

To upper-bound $\cZ_n(1/2)$, let
\[\cQ_{X,Y} = \inbraces{(\cR_X, \cR_Y) \colon \cR_X \subseteq [N_X], 
\cR_Y \subseteq [N_Y], N_X - \abs{\cR_X} = N_Y - \abs{\cR_Y} }\]
be the set of all possible tuples $(\cR_X, \cR_Y)$ such that
$\cR_X$ is a subset of $[N_X]$, 
$\cR_Y$ is a subset of $[N_Y]$, and
$[N_X] \backslash \cR_X$ and $[N_Y] \backslash \cR_Y$ have equal 
cardinality. Then
\begin{align*}
    &\cZ_n(1/2)
    = \sum_{\pi \in \SSS_n}
    \exp\inparen{-H_n\inparen{\pi \mid X, Y;1/2}} \\
   =&\;\sum_{(\cR_X, \cR_Y) \in \cQ_{X,Y}}
    \exp\inparen{\sum_{i \in \cR_X} U_n(X_i)
    + \sum_{j \in \cR_Y} U_n(Y_j)}
    \sum_{\pi: [N_X] \backslash \cR_X \to [N_Y] \backslash \cR_Y} 
    \exp\inparen{-\frac{1}{2} \sum_{i \in [N_X] \backslash \cR_X}
    V(n(X_i-Y_{\pi(i)}))} \\
    \leq & \; \exp\inparen{C(N_X+N_Y)}
    \sum_{\cR_X, \cR_Y \in \cQ_{X,Y}} 
    \sum_{\pi: [N_X] \backslash \cR_X \to [N_Y] \backslash \cR_Y} 
    \exp\inparen{-\frac{1}{2} \sum_{i \in [N_X] \backslash \cR_X}
    V(n(X_i-Y_{\pi(i)}))}
\end{align*}
where the last inequality uses \eqref{eq:pnuniformbounds}
to bound $U_n(x) \le C$. Again upper-bounding the permanent of a
nonnegative matrix by the product of its row sums,
\begin{align*}
    & \sum_{\pi: [N_X] \backslash \cR_X \to [N_Y] \backslash \cR_Y} 
    \exp\inparen{-\frac{1}{2} \sum_{i \in [N_X] \backslash \cR_X}
    V(n(X_i-Y_{\pi(i)}))} \\
    & \le \prod_{i \in [N_X] \backslash \cR_X} 
    \sum_{j \in [N_Y] \backslash \cR_Y} 
    \exp\inparen{-\frac{1}{2} V(n(X_i-Y_j))} 
    \le 
\prod_{i \in \pi_X^*([N_X] \backslash \cR_X)} 
    \sum_{j=1}^N 
    \exp\inparen{-\frac{1}{2} V(n(\bar X_i-\bar Y_j))}.
\end{align*}
Thus
\begin{align}
\label{eq:expected_free_energy_at_1/2}
  \cF_n(1/2)
    \le \frac{CN}{n} +
   \frac{1}{n} \log \sum_{\cR_X, \cR_Y \in \cQ_{X,Y}} 
   \prod_{i \in \pi_X^*([N_X] \backslash \cR_X)} 
    \sum_{j=1}^N 
    \exp\inparen{-\frac{1}{2} V(n(\bar X_i-\bar Y_j))}.
\end{align}
To bound the second term on the right side,
we follow the argument in
\eqref{eq:non-(log n)^kappa-neighbor-bound}
and 
\eqref{eq:chernoff_bound_binomial_opt}:
Let $\cN_\kappa(\bar Y_j)$ denote the index set of the
$\floor{(\log n)^\kappa}$ nearest neighbors of
$\bar Y_j$ in $\{\bar Y_1,\ldots,\bar Y_N\}$. Then
w.h.p.\ for every $i \in [N]$,
\begin{align*}
\sum_{j=1}^N \exp\inparen{-\frac{1}{2} V(n(\bar X_i-\bar Y_j))}
    \le \exp\inparen{ C + \frac{n}{2} \abs{\bar X_i- \bar Y_i}}
+\sum_{j \in \cN_\kappa(\bar Y_i) }
    \exp\inparen{-\frac{1}{2} V(n(\bar X_i-\bar Y_j))}.
\end{align*}
This upper bound is at least 1, so we may further include the indices $i \notin
\pi_X^*([N_X] \setminus \cR_X)$ to obtain
\begin{align*}
    \cF_n(1/2)
    &\le \frac{CN}{n} +
    \frac{1}{n} \log \abs{\cQ_{X,Y}}
    +\frac{1}{n} \sum_{i=1}^N \log \insquare{
    \exp\inparen{C + \frac{n}{2} \abs{\bar X_i- \bar Y_i}}
+ \sum_{j \in \cN_\kappa(\bar Y_i)}
    \exp\inparen{-\frac{1}{2} V(n(\bar X_i-\bar Y_j))}}.
\end{align*}
The third term is at most a constant w.h.p.\ by the same argument as for
\eqref{eq:upper_bd_Fn1/2_bdd_diff_fxc} in exact matching.
The cardinality of $\cQ_{X,Y}$ can be bounded from above by
$2^{N_X}2^{N_Y} \leq 2^N$.
Thus, applying again the Poisson tail bound $N \leq Cn$ w.h.p.,
\begin{align}
\label{eq:high_prob_bound_Fn1/2}
     \cF_n(1/2) \le C \text{ w.h.p.}
\end{align}
Combining
\eqref{eq:high_prob_bound_Fn1} and 
\eqref{eq:high_prob_bound_Fn1/2} shows $\cD_n(1) \le C$ w.h.p.
\end{proof}

We next turn to showing a similar mean potential result under a local Gibbs
measure. For any $x \in [0,1]$ and partial matching $\pi$,
define its boundary variable at $x$ by
\begin{align}
    \Gamma_x(\pi) \coloneqq \{(k,\pi(k)) \colon
    k,\pi(k) \neq \varnothing,\,X_k \le x < Y_{\pi(k)}
    \text{ or }  X_k > x \geq Y_{\pi(k)}\}.
\end{align}

\begin{lemma}
\label{lemma:local_mean_potential_partialMatching}
Fix any $K,L>0$. There exists a constant $C>0$ not depending on $K,L$
such that
\begin{align*}
    \frac{1}{n}\,\sum_{l=1}^n
    E\insquare{\frac{1}{KL}\sum_{i \in [N_X]:\pi(i) \neq \varnothing,\,
X_i \in (\frac{l-KL}{n},\frac{l+KL}{n}]}
    V\inparen{n\inparen{X_i - Y_{\pi(i)}}}
    \,\Bigg\rvert\, X,Y,\Gamma_{\frac{l \pm KL}{n}}(\pi)=\varnothing} \leq C \text{
w.h.p.}
\end{align*}
\end{lemma}
\begin{proof}
Fix $l \in [n]$. Let
$\SSS^{(l)}$ be the set of all partial bijections $\pi$ from
$S_X^l=\inbraces{i:X_i \in (\tfrac{l-KL}{n}, \tfrac{l+KL}{n}]}$ to
$S_Y^l=\inbraces{j:Y_j \in (\tfrac{l-KL}{n}, \tfrac{l+KL}{n}]}$.
Define the local Gibbs measure over 
$\pi \in \SSS^{(l)}$ and the corresponding free energy
\begin{align}
\label{eq:localgibbsmeasure_partialmatching}
    P^l_\beta \inparen{\pi \mid X, Y} = \frac{1}{\cZ_l(\beta)} 
    \exp \inparen{-H_n(\pi \mid S_X^l, S_Y^l; \beta)},
    \quad \cF_l(\beta) = \frac{1}{KL} \log \cZ_l(\beta)
\end{align}
where 
\begin{align*}
    & H_n(\pi \mid S_X^l, S_Y^l;\beta) = \beta \sum_{i \in S_X^l:\pi(i) \in
S_Y^l}  V(n(X_i-Y_{\pi(i)}))
    -\sum_{i \in S_X^l:\pi(i)=\varnothing} U_n(X_i)
    -\sum_{j \in S_Y^l:\pi^{-1}(j)=\varnothing} U_n(Y_j).
\end{align*}
Then the quantity to be bounded is
$\frac{1}{n}\sum_{l=1}^n \cD_n^l$, where
\[\cD_n^l:=
E_{\beta=1}^l\left[\frac{1}{KL} \sum_{i \in S_X^l:\pi(i) \in S_Y^l} 
V(n(X_i-Y_{\pi(i)})) \bigg\rvert X,Y\right]
={-}\cF_l'(1) \le 2 \inparen{\cF_l(1/2)-\cF_l(1)}.\]

We upper-bound $\frac{1}{n}\sum_{i=1}^{n}\cF_l(1/2)$ as follows:
Let
\[\cQ^l=\inbraces{(\cR_X,\cR_Y) \colon \cR_X \subseteq S_X^l, \cR_Y \subseteq
S_Y^l, \abs{ S_X^l \backslash \cR_X} = \abs{ S_Y^l \backslash \cR_Y}},\]
which has cardinality $|\cQ^l| \leq 2^{|S_X^l|}2^{|S_Y^l|}$.
Then following the preceding free energy upper bound for the global Gibbs
measure, we have
\begin{align*}
    \frac{1}{n} \sum_{l=1}^{n} \cF_l(1/2) & \le \frac{1}{n}
\sum_{l=1}^{n}\frac{C(|S_X^l|+|S_Y^l|)}{KL} \\
    &\quad +\frac{1}{n} \sum_{l=1}^{n} \frac{1}{KL}\sum_{i \in \pi_X^*(S_X^l)} \log
\insquare{\exp\inparen{C+\frac{n \abs{\bar X_i - \bar Y_i}}{2}} +\sum_{j\in
\cN_\kappa(\bar Y_i)} \exp\inparen{-\frac{1}{2}V\inparen{n\abs{\bar X_i - \bar
Y_j}}}}\\
&\leq \frac{CN}{n}
+\frac{C}{n} \sum_{i=1}^N \log
\insquare{\exp\inparen{C+\frac{n \abs{\bar X_i - \bar Y_i}}{2}} +\sum_{j\in
\cN_\kappa(\bar Y_i)} \exp\inparen{-\frac{1}{2}V\inparen{n\abs{\bar X_i - \bar
Y_j}}}}\\
&\leq C' \text{ w.h.p.}
\end{align*}

We may lower-bound $\frac{1}{n} \sum_{l=1}^{n} \cF_l(1)$ by lower-bounding the
sum over $\pi \in \SSS^{(l)}$ by the single summand $\pi^*$ restricted to
$(S_X^l,S_Y^l)$:
\begin{align*}
    \frac{1}{n} \sum_{l=1}^{n} \cF_l(1) &\geq \frac{1}{n} \sum_{l=1}^{n}
\frac{1}{KL} \sum_{i \in S_X^l: \pi^*(i) \in S_Y^l} -V\inparen{n(X_i -
Y_{\pi^*(i)})} \\
    &\qquad + \sum_{i \in N_X^l: \pi^*(i)=\varnothing \text{ or } \pi^*(i)
\notin S_Y^l} U_n(X_i) + \sum_{j \in S_Y^l: {\pi^*}^{-1}(j) = \varnothing \text{ or } 
    {\pi^*}^{-1}(j) \notin S_X^l} U_n(Y_j).
\end{align*}
Recalling the lower bound $U_n(x) \ge {-}C$ and applying the same concentration
over $\eps_i = n(\bar Y_i - \bar X_i)$ as in the preceding
argument for the global Gibbs measure, we obtain
\[ \frac{1}{n} \sum_{l=1}^{n} \cF_l(1) \geq -C \text{ w.h.p.},\]
and the lemma follows from combining these upper and lower bounds.
\end{proof}

\subsubsection{Regularity of the point processes and locality of the Gibbs measure}

Fixing locality parameters $K,L > 0$ and recalling the upper bound $\Lambda(x) \le \Lambda_{\max}$, for each $x \in [0,1]$ we define the $(X,Y)$-dependent
event
\begin{align*}
\cA_x &= \inbraces{\sum_{i \in [N_X]} \1 \inbraces{ |X_i-x| \leq \frac{L}{n}} 
\leq \frac{9L\Lambda_{\max}}{(1-p)^2},\,
\; \sum_{j \in [N_Y]} \1 \inbraces{ |Y_j-x| \leq \frac{L}{n}} 
\leq \frac{9L\Lambda_{\max}}{(1-p)^2}}
\end{align*}
Conditional on $(X,Y)$, we define the $\pi$-dependent events
\begin{align*}
  \cC_x &\equiv \cC_x(\pi) = \inbraces{|X_k-Y_m| \leq L/n \text{ for all }
(k,m) \in \Gamma_x(\pi)},\\
    \cL_x &\equiv \cL_x(\pi) = \cA_x \cap 
     \cC_x(\pi).
\end{align*}
For each $i \in [N_X]$ and $k \in \Z$, denote
\[X_i^k=\frac{\lfloor nX_i \rfloor+k}{n}.\]
Note that $\widehat P_i(j)$ computed in Algorithm \ref{alg:partial} is precisely
$P(\pi(i)=j \mid X,Y,\Gamma_{X_i^{\pm KL}}=\varnothing)$, for each $j \in [N_Y]
\cup \{\varnothing\}$.
Define
\begin{align*}
    \cG_i(\pi) =
    \inbraces{\sum_{k=0}^{KL-1}
    \boldsymbol{1}\{\cL_{X_i^{-k}}(\pi)\} 
    \ge \frac{2KL}{3} } \cap 
    \inbraces{\sum_{k=1}^{KL}
    \boldsymbol{1}\{\cL_{X_i^k}(\pi)\} 
    \ge \frac{2KL}{3} }.
\end{align*}
In this section, we show the following lemma.

\begin{lemma}
\label{lemma:locality_partialMatching}
Fix the locality parameters $K,L>0$. For a constant $C>0$ not depending on
$K,L$, w.h.p.\ over $(X,Y)$,
\begin{align}
\frac{1}{N_X}\sum_{i=1}^{N_X}P\inparen{\cG_i^c(\pi)
    \mid X, Y} \le C L^{-\delta}, \qquad 
    \frac{1}{N_X}\sum_{i=1}^{N_X} P\inparen{\cG_i^c(\pi)
    \;\Big|\; X, Y,\Gamma_{X_i^{\pm KL}}(\pi)=\varnothing} \le C L^{-\delta}.
\end{align}
\end{lemma}

We first show that $\cA_x$ holds with high probability for most sites $X_i^k$.
\begin{lemma}
\label{lemma:localityeventA_partialMatching}
For constants $C,c>0$ not depending on $K,L$,
    \[\frac{1}{N_X} \sum_{i=1}^{N_X} \frac{1}{KL} \sum_{k=-KL}^{KL}
\1\{\cA_{X_i^k}^c\} \le Ce^{-cL}.\]
\end{lemma}
\begin{proof}
Fix any $\kappa>1$, and
let $\cN_i^k$ be the indices of the $\floor{(\log n)^\kappa}$ points
in $\{X_1,\ldots,X_{N_X}\}$ that are closest to $X_i^k$. By a binomial tail
bound, w.h.p.\ all points $X_j \in X_i^k \pm \frac{L}{n}$ belong to $\cN_i^k$.
On this event,
\begin{align}
    & \frac{1}{N_X}\sum_{i=1}^{N_X} 
\frac{1}{KL} \sum_{k=-KL}^{KL} \1\{\cA_{X_i^k}^c\}\notag\\
&= \frac{1}{N_X} \sum_{i=1}^{N_X}  
\frac{1}{KL} \sum_{k=-KL}^{KL}
    \boldsymbol{1}\bigg\{
    \left\lvert \left\{
    j \in \cN_i^k \colon X_j \in X_i^k \pm \frac{L}{n}
    \right\} \right\rvert\geq \frac{9\Lambda_{\max}L}{(1-p)^2}\bigg\}
    \label{eq:ub_termC}
\end{align}
Conditioning on the count $N_X$, 
\eqref{eq:ub_termC} satisfies a bounded difference property
in the i.i.d.\ variables $X_1,\cdots,X_{N_X}$,
since changing any $X_i$ changes its value by
at most $\frac{C(\log n)^{\kappa}}{n}$. 
Applying the Chernoff bound and moment generating function inequality
$\E[e^B] \leq e^{4np}$ for $B \sim
\text{Binomial}(n,p)$, its expectation conditional on $N_X$ is also
bounded from above by
\begin{align*}
    &\frac{1}{N_X} \sum_{i=1}^{N_X} 
\frac{1}{KL} \sum_{k=-KL}^{KL}
    \P\Bigg[
\left\lvert \left\{
    j \in [N_X] \colon X_j \in X_i^k \pm \frac{L}{n}
    \right\} \right\rvert \geq \frac{9\Lambda_{\max}L}{(1-p)^2}
\;\Bigg|\; N_X\Bigg]\\
    &\le
    C\exp\inparen{-\frac{9 \Lambda_{\max} L}{(1-p)^2} + \frac{8\Lambda_{\max} L
N_X}{n}}.
\end{align*}
Since $N_X \leq N \leq Cn$ by the
Poisson tail bound, this is at most $Ce^{-cL}$ w.h.p., so the lemma
follows from the bounded differences inequality.
\end{proof}

We next establish that most sites satisfy $\cC_x(\pi)$, using the mean potential results obtained previously in Lemmas \ref{lemma:mean_potential_partial_matching} and \ref{lemma:local_mean_potential_partialMatching}.
\begin{lemma}
\label{lemma:localityeventC_partialMatching}
Let
\[\mathscr{S} (\pi) = \inbraces{x \in [0,1] \colon
\sum_{k=0}^{KL-1} 
    \1\{\cC^c_{\frac{\lfloor nx\rfloor-k}{n} }(\pi)\} 
    \ge \frac{KL}{6} \text{ or } \sum_{k=1}^{KL}
    \1\{\cC^c_{\frac{\lfloor nx\rfloor+k}{n} }(\pi)\} 
    \ge \frac{KL}{6} }\]
There exists a constant $C>0$ not depending on $K,L$
such that w.h.p.
\[\frac{1}{N_X} \sum_{i=1}^{N_X} P\inparen{X_i
\in \mathscr{S}(\pi) \mid X,Y} \le C L^{-\delta/2}, \quad 
\frac{1}{N_X} \sum_{i=1}^{N_X} P\inparen{X_i
\in \mathscr{S}(\pi) \mid X,Y,\Gamma_{X_i^{\pm KL}}(\pi)=\varnothing}
\le C L^{-\delta/2}.\]
\end{lemma}
\begin{proof}
For the global Gibbs measure, we have
\begin{align*}
     &\frac{1}{n} \sum_{l=1}^{n}  \frac{1}{KL} \sum_{i=-KL}^{KL}
\1\inbraces{\cC^c_{\frac{l+i}{n}}(\pi)} 
     \le \frac{C}{n} \sum_{l=1}^{n}  \1\inbraces{\cC^c_{\frac{l}{n}}(\pi)} \\
   & \le \frac{C}{n} \sum_{l=1}^{n}  \sum_{k \in [N_X]: \pi(k) \neq \varnothing}
    \1\inbraces{X_k \le \tfrac{l}{n} < Y_{\pi(k)} \text{ or } X_k > \tfrac{l}{n} \ge Y_{\pi(k)}}
    \1\inbraces{\abs{Y_{\pi(k)} - X_k} \geq L/n} \\
    &\le \frac{C}{n} \sum_{k \in [N_X]: \pi(k) \neq \varnothing}
n\abs{Y_{\pi(k)} - X_k} \1\inbraces{n\abs{Y_{\pi(k)} - X_k} \geq L}
\end{align*}
Applying \eqref{eq:first_mmoment_tail_estimate} and Lemma
\ref{lemma:mean_potential_partial_matching},
\begin{align}
& \frac{1}{n} \sum_{l=1}^{n}  \frac{1}{KL} \sum_{i=-KL}^{KL}
P[\cC^c_{\frac{l+i}{n}}(\pi) \mid X,Y]
\le \frac{C}{ L^{\delta}} E\inparen{\frac{1}{n}\sum_{k \in [N_X]: \pi(k) \neq
\varnothing}  (n \abs{Y_{\pi(k)} - X_k})^{1+\delta}\;\Bigg|\; X,Y} \nonumber \\
& \le C' L^{-\delta} \text{ w.h.p.} 
\end{align}
This and Markov's inequality imply that
\begin{equation}\label{eq:globalgoodsitebound}
\frac{1}{n} \sum_{l=1}^{n}
P\inparen{\tfrac{l}{n} \in \mathscr{S}(\pi) \mid X,Y} \le C L^{-\delta} \text{ w.h.p.}
\end{equation}
Then letting $N_l = |\{i \colon X_i \in [\tfrac{l}{n}, \tfrac{l+1}{n})\}|$ and
noting $\sum_l N_l = N_X$, we have
\begin{align*}
&\frac{1}{N_X} \sum_{i=1}^{N_X} P\inparen{X_i \in \mathscr{S}(\pi) \mid X,Y} 
 \le \frac{1}{N_X} \sum_{l=1}^{n}  \underbrace{P\inparen{\tfrac{l}{n}
\in \mathscr{S}(\pi) \mid X,Y}}_{:=p_l} N_l 
\le L^{-\delta/2} + \frac{1}{N_X} \sum_{l \in [n]:\,p_l \geq L^{-\delta/2}} N_l.
\end{align*}
The number of indices $l \in [n]$ where $p_l \geq L^{-\delta/2}$ is bounded as
\[\sum_{l=1}^n \1\inbraces{p_l \ge L^{-\delta/2}} \le
L^{\delta/2}\sum_{l=1}^n p_l \le C L^{-\delta/2}n \text{ w.h.p.}\]
There are at most
$\binom{n}{CL^{-\delta/2}n} \le 2^n$ subsets of such indices $l \in [n]$.
Then by a Binomial tail bound and union bound over all such subsets,
there exists a constant $C>0$ such that 
\[\sup_{S \subseteq [n]:|S| \leq CL^{-\delta/2}n} \sum_{l \in S} N_l
\leq C'L^{-\delta/2} N_X \text{ w.h.p.}\]
Hence we obtain 
\[\frac{1}{N_X}\sum_{l=1}^{N_X} P\inparen{X_i \in \mathscr{S}(\pi) \mid X,Y} \le CL^{-\delta/2} \text{ w.h.p.}\]

For the local Gibbs measure, for each fixed $l \in [n]$, on the
event $\Gamma_{\frac{l \pm KL}{n}}(\pi)=\varnothing$ we have similarly
\begin{align*}
    &\sum_{i=-KL}^{KL} \1\inbraces{\cC_{\frac{l+i}{n}}^c(\pi)}\\
     &\leq \sum_{i=-KL}^{KL} \sum_{k \in [N_X]: \pi(k) \neq \varnothing,\,
X_k \in (\frac{l-KL}{n},\frac{l+KL}{n}]}
\1\inbraces{\abs{Y_{\pi(k)}-X_k} \geq \tfrac{L}{n}} \1\inbraces{X_k \le \tfrac{l+i}{n} < Y_{\pi(k)} \text{ or }X_k > \tfrac{l+i}{n} \geq Y_{\pi(k)} } \\
    &\le 
\sum_{k \in [N_X]: \pi(k) \neq \varnothing,\,
X_k \in (\frac{l-KL}{n},\frac{l+KL}{n}]}
n \abs{Y_{\pi(k)}-X_k}\1\inbraces{n \abs{Y_{\pi(k)}-X_k}\geq  L}.
\end{align*}
Then applying \eqref{eq:first_mmoment_tail_estimate} and Lemma
\ref{lemma:local_mean_potential_partialMatching}, and then
averaging over $l \in [n]$,
\begin{align*}
     &\frac{1}{n}\sum_{l=1}^{n}\frac{1}{KL} \sum_{i=-KL}^{KL}
P[\cC_{(l+i)/n}^c(\pi) \mid X,Y,\Gamma_{\frac{l \pm KL}{n}}(\pi)=\varnothing]\\
     &\le  \frac{C}{n L^{\delta}}\sum_{l=1}^{n} E\Bigg[\frac{1}{KL} 
\sum_{k \in [N_X]: \pi(k) \neq \varnothing,\,
X_k \in (\frac{l-KL}{n},\frac{l+KL}{n}]}
(n \abs{Y_{\pi(k)}-X_k})^{1+\delta} \;\Bigg|\; X,Y,\Gamma_{\frac{l \pm
KL}{n}}(\pi)=\varnothing \Bigg]\\
     &\le C'L^{-\delta} \text{ w.h.p.}
\end{align*}
Then again by Markov's inequality, we obtain analogously to
\eqref{eq:globalgoodsitebound} that
\[\frac{1}{n} \sum_{l=1}^n P(\tfrac{l}{n} \in \mathscr{S}(\pi) \mid
X,Y,\Gamma_{\frac{l \pm KL}{n}}(\pi)=\varnothing) \le C L^{-\delta}.\]
The rest of the proof is identical to that under the global Gibbs measure.
\end{proof}

\begin{proof}[Proof of Lemma \ref{lemma:locality_partialMatching}]
We have
\begin{align*}
 \frac{1}{N_X}\sum_{i=1}^{N_X} \1\{\cG_i^c(\pi)\} &\leq
\frac{1}{N_X}\sum_{i=1}^{N_X} \1\inbraces{\sum_{k=0}^{KL-1}
    \1\inbraces{\cL_{X_i^{-k}}^c(\pi)} 
    \ge \frac{KL}{3} \text{ or } \sum_{k=1}^{KL}
    \1\inbraces{\cL_{X_i^k}^c(\pi)} \ge
\frac{KL}{3}}\\
&\leq \frac{1}{N_X}\sum_{i=1}^{N_X} \Bigg(\1\inbraces{\sum_{k=0}^{KL-1}
    \1\inbraces{\cA_{X_i^{-k}}^c}
    \ge \frac{KL}{6}}+
\1\inbraces{\sum_{k=1}^{KL} \1\inbraces{\cA_{X_i^k}^c} \ge
\frac{KL}{6}}\\
&\hspace{1in}+\1\inbraces{\sum_{k=0}^{KL-1}
    \1\inbraces{\cC_{X_i^{-k}}^c(\pi)}
    \ge \frac{KL}{6} \text{ or } \sum_{k=1}^{KL}
    \1\inbraces{\cC_{X_i^k}^c(\pi)} \ge
\frac{KL}{6}}\Bigg)\\
&\leq \frac{1}{N_X}\sum_{i=1}^{N_X} \frac{6}{KL} \sum_{k=-KL}^{KL}
\1\inbraces{\cA_{X_i^k}^c}
+\frac{1}{N_X}\sum_{i=1}^{N_X} \1\{X_i \in \mathscr{S}(\pi)\}.
\end{align*}
The result then follows from Lemmas \ref{lemma:localityeventA_partialMatching} and
\ref{lemma:localityeventC_partialMatching}.
\end{proof}

\subsubsection{Correlation decay}

We now prove Theorem \ref{thm:soft_alg} by establishing correlation
decay for the posterior law.

\begin{proof}[Proof of Theorem \ref{thm:soft_alg}]

Write as shorthand $P(\cdot) = P(\,\cdot \mid X,Y)$. 
We follow a similar coupling argument as in the proof of Theorem
\ref{thm:exact_alg}.
Let $\pi,\pi'$ be two independent draws from
the posterior $P(\pi \mid X,Y)$. Conditional on $(X,Y)$,
fix $i \in [N_X]$ and define the lists of sites
\begin{align*}
\cQ_i^-(\pi)&=\left\{k \in \{-KL-1,\ldots,0\}:\cL_{X_i^k}(\pi) \text{
holds}\right\},\\
\cQ_i^+(\pi)&=\left\{k \in \{1,\ldots,KL\}:\cL_{X_i^k}(\pi) \text{ holds}\right\},\\
\cQ_i(\pi)&=\cQ_i^-(\pi) \cup \cQ_i^+(\pi),
\end{align*}
and define similarly
$\cQ_i^-(\pi'),\cQ_i^+(\pi'),\cQ_i(\pi')$ for $\pi'$.
On the event $\cG_i(\pi) \cap \cG_i(\pi')$ we must have
\[|\cQ_i^-(\pi) \cap \cQ_i^-(\pi')| \geq \frac{KL}{3},
\qquad |\cQ_i^+(\pi) \cap \cQ_i^+(\pi')| \geq \frac{KL}{3}.\]

For any $m \in \{-KL-1,\ldots,KL\}$,
let $\cF_m(\pi)$ denote the information set given by $\cQ_i(\pi)$ and
the list of boundary variables
\[\{\Gamma_{X_i^l}(\pi):l \in \cQ_i(\pi) \setminus [m-L,m+L]\}.\]
Let us lower bound
\[P(\Gamma_{X_i^m}(\pi)
=\varnothing \mid \cL_{X_i^m}(\pi),\,\cF_m(\pi)).\]
Consider any $\pi$ such that $\cL_{X_i^m}(\pi)$
holds, and $\Gamma_{X_i^m}(\pi)=\Gamma \neq
\varnothing$. We construct from $\pi$ the partial bijection $\tilde \pi$ 
that keeps all matches except those in $\Gamma$ the
same, and sets $\tilde \pi(k) = \varnothing$ 
and $\tilde \pi^{-1}(m) = \varnothing$ for each $(k,m) \in \Gamma$.
Then by construction, $\Gamma_{X_i^m}(\tilde \pi)=\varnothing$
so $\cL_{X_i^m}(\tilde \pi)$ holds. Furthermore, $\pi$ and $\tilde \pi$ have the
same long-range matches, i.e.\ $|X_j-Y_{\pi(j)}|>L/n$ if and only
if $|X_j-Y_{\tilde\pi(j)}|>L/n$, so $\cF_m(\pi)=\cF_m(\tilde \pi)$.
Recalling the lower bounds $V(\eps) \geq -C$, and $U_n(x) \geq {-}C$,
\begin{align}
\label{eq:boundary_map_bound_partialMatching}
    \frac{\exp(-H_n(\pi))}{\exp(-H_n(\widetilde{\pi}))} 
    &= \exp\inparen{-\sum_{(k,m) \in \Gamma} U_n(X_k) + U_n(Y_m) + V(n\abs{X_k -
Y_m}) } \le \exp\inparen{C|\Gamma|}.
\end{align}
For each fixed boundary condition $\Gamma \neq \varnothing$,
this mapping from $\pi$ to $\tilde \pi$ is injective. On the event
$\cA_{X_i^m} \subset \cL_{X_i^m}(\pi)$, there are at most $L'=CL$ points
$\{X_1,\ldots,X_{N_X},Y_1,\ldots,Y_{N_Y}\}$ in $[X_i^{m-L},X_i^{m+L}]$,
so the number of possible boundary conditions $\Gamma \neq \varnothing$
satisfying $\cL_{X_i^m}$ is at most $e^{CL\log L}$. Then, following the same
arguments as in Theorem \ref{thm:exact_alg}, we get
\begin{align}
P(\Gamma_{X_i^m}(\pi)=\varnothing \mid \cL_{X_i^m}(\pi),\cF_m(\pi))
\geq \exp(-CL\log L)=:\iota(L).
\label{eq:emptyboundarylowerbound_partialMatching}
\end{align}
The same arguments apply for the posterior distribution of
$\pi$ conditioned also on $\Gamma_{X_i^{\pm KL}}=\varnothing$, so
\begin{equation}\label{eq:emptyboundaryconditionallowerbound_partialMatching}
P(\Gamma_{X_i^m}(\pi)=\varnothing \mid
\cL_{X_i^m}(\pi),\cF_m(\pi),\Gamma_{X_i^{\pm KL}}(\pi)=\varnothing)
\geq \iota(L).
\end{equation}
Setting
\begin{align*}
M_- &\equiv M_-(\pi,\pi')=\min\{m:m \in \cQ_i^-(\pi) \cap \cQ_i^-(\pi'),
\,\Gamma_{X_i^m}(\pi)=\Gamma_{X_i^m}(\pi')=\varnothing\},\\
M_+ &\equiv M_+(\pi,\pi')=\max\{m:m \in \cQ_i^+(\pi) \cap \cQ_i^+(\pi'),
\,\Gamma_{X_i^m}(\pi)=\Gamma_{X_i^m}(\pi')=\varnothing\}
\end{align*}
with the conventions $M_-=0$ or $M_+=0$ if no such indices exist, we obtain
as in the argument of Theorem \ref{thm:exact_alg} that
\begin{align*}
\TV\big(\widehat P_i,P_i\big)
&\le 2P(M_-=0 \text{ or } M_+=0)
+2P(M_-=0 \text{ or } M_+=0 \mid \Gamma_{X_i^{\pm KL}}(\pi)=\varnothing)\\
&\le C\Big((1-\iota(L)^2)^{K/3}+P(\cG_i^c(\pi))+P(\cG_i^c(\pi)
\mid \Gamma_{X_i^{\pm KL}}(\pi)=\varnothing)\Big).
\end{align*}
Then the theorem follows from
applying Lemma \ref{lemma:locality_partialMatching} to average over sites
$i \in [N_X]$.
\end{proof}

\subsection{Proof of Theorem \ref{thm:partial_asymptotics}}

\subsubsection{Local weak convergence to a limiting point process}

Let $\cD_N=\{N,\cS_{XY},\cS_X,\cS_Y,\cS_\varnothing,\{(\bar X_i, \bar
Y_i)\}_{i=1}^N,\pi_X^*,\pi_Y^*\}$.
Conditional on $\cD_N$, let
$\bar I=\pi_X^*(I) \sim \Unif\inbraces{\cS_{XY} \cup
\cS_{X}}$ be the index of the random centering point in
$\{\bar X_1,\ldots,\bar X_N\}$, and consider the point process on $\R^2$
\[\eta_N=\eta_N^{XY}+\eta_N^X+\eta_N^Y+\eta_N^\varnothing,
\qquad
\eta_N^\mark=\sum_{i \in \cS_\mark} \delta_{n(\bar X_i - \bar X_{\bar I},
\bar Y_i - \bar X_{\bar I})} \text{ for } \mark \in \{XY,X,Y,\varnothing\}\]
with randomness induced by the choice of $\bar I$.
We define the limit point processes
\[\eta =\eta^{XY}+ \eta^X+ \eta^Y+
\eta^\varnothing\]
as follows: (i) Sample $X \sim \Lambda$. (ii) Sample $\eps \sim V(\cdot)$, and
sample the origin mark $\mark_0 \in \inbraces{XY,X}$ with
$\PP\insquare{\mark_0=XY}=p$ and $\PP\insquare{\mark_0=X}=1-p$.
(iii) Let
\[\mu_X(\ud x\,\ud y)=\frac{\Lambda(X)}{(1-p)^2}q(y-x)\ud x\,\ud y,
\quad (p_{XY},p_X,p_Y,p_\varnothing)
=\Big(p^2,p(1-p),p(1-p),(1-p)^2\Big),\]
and sample independent Poisson point processes
\begin{align*}
\inparen{\eta^{XY}}' \sim \PPP(p_{XY} \mu_X), \quad
\inparen{\eta^{X}}' \sim \PPP(p_X \mu_X), \quad \eta^{Y} \sim
\PPP(p_Y \mu_X), \quad \eta^{\varnothing} \sim \PPP(p_\varnothing
\mu_X).
\end{align*}
(iv) Set
\[\begin{cases}
\eta^{XY}=\inparen{\eta^{XY}}'+\delta_{(0,\eps)},\qquad
\eta^X=\inparen{\eta^X}'
& \text{ if } \mark_0 = XY,\\
\eta^X=\inparen{\eta^X}'+\delta_{(0,\eps)},\qquad
\eta^{XY}=\inparen{\eta^{XY}}'
& \text{ if } \mark_0 = X. \end{cases}\]

\begin{lemma}\label{lemma:softM_Unmarked}
Conditional on $\cD_N$, over the uniform random index $\bar I \sim
\Unif\{\cS_{XY} \cup \cS_X\}$, almost surely as $n \to \infty$,
\[(\eta_N^{XY},\eta_N^X,\eta_N^Y,\eta_N^\varnothing)
\convWeakly (\eta^{XY},\eta^X,\eta^Y,\eta^\varnothing).\]
\end{lemma}
\begin{proof}
By Propositions \ref{prop:pointProcessConvergence_Equivalences}
and \ref{prop:pointProcess_convergenceByRectangles}, it is equivalent to
show that for any rational rectangle simple functions
$s^{XY},s^X,s^Y,s^\varnothing \in C_c^+(\RR^2)$, we have
almost surely as $n \to \infty$,
\begin{align*}
&\Psi_{\eta_N}(s):=\E\left[\exp\left\{\sum_{\mark \in \{XY,X,Y,\varnothing\}}
s^\mark(x,y)\eta_N^\mark(\ud x\,\ud y)\right\}\;\Bigg|\;\cD_N\right]\\
&\to \Psi_{\eta}(s)
:=\E\exp\left\{\sum_{\mark \in \{XY,X,Y,\varnothing\}}
s^\mark(x,y) \eta^\mark(\ud x\,\ud y)\right\}
\end{align*}

Suppose $s^\mark=\sum_{\ell=1}^m a_\ell^\mark \boldsymbol{1}_{R_\ell^\mark}$,
where we may assume without loss of generality $m$ is the same for each $\mark
\in \{XY,X,Y,\varnothing\}$.
Define $\cN_\mark^{(i)}(R_\ell^\mark)=|\{j \in \cS_\mark:j \neq i,\,
n(\bar X_j-\bar X_i,\bar Y_j -\bar X_i) \in R_\ell^\mark\}|$, and set
\[\cN_\mark^{(i)}=\Big(\cN_\mark^{(i)}(R_1^\mark),\ldots,\cN_\mark^{(i)}(R_m^\mark)\Big),
\qquad a^\mark=(a_1^\mark,\ldots,a_m^\mark),
\qquad \eps_i=n(\bar Y_i-\bar X_i).\]
Then
\begin{align*}
\Psi_{\eta_N}(s) = \frac{1}{|\cS_{XY}|+|\cS_X|}
\sum_{i \in \cS_{XY} \cup \cS_X} \exp\left\{
{-}s^{XY}(0,\eps_i)\1_{i \in \cS_{XY}}
-s^X(0,\eps_i)\1_{i \in \cS_X}
-\sum_{\mark \in \{XY,X,Y,\varnothing\}} (a^\mark)^\top \cN_\mark^{(i)}\right\}.
\end{align*}
Conditional on $N,\cS_{XY},\cS_X,\cS_Y,\cS_\varnothing$
and on $\bar X_i$ for any fixed $i \in \cS_{XY} \cup \cS_X$, note that
$\{n(\bar X_j-\bar X_i,\bar Y_j-\bar X_i)\}_{j \in \cS_\mark \setminus \{i\}}$
are i.i.d.\ for each $\mark \in \{XY,X,Y,\varnothing\}$. Then
\begin{align*}
&\E[\Psi_{\eta_N}(s) \mid N,\cS_{XY},\cS_X,\cS_Y,\cS_\varnothing]\\
&=\frac{|\cS_{XY}|}{|\cS_{XY}|+|\cS_X|}
\E\left[e^{-s^{XY}(0,\eps_i)}\E\left[e^{{-}(a^\mark)^\top \xi^{XY}}
\;\Big|\;\bar X_i\right]^{|\cS_{XY}|-1}
\prod_{\mark \in \{X,Y,\varnothing\}}
\E\left[e^{{-}(a^\mark)^\top \xi^\mark}\;\Big|\; \bar
X_i\right]^{|\cS_\mark|}\right]\\
&\qquad+\frac{|\cS_X|}{|\cS_{XY}|+|\cS_X|}
\E\left[e^{-s^X(0,\eps_i)}\E\left[e^{{-}(a^\mark)^\top \xi^X}
\;\Big|\;\bar X_i\right]^{|\cS_X|-1}
\prod_{\mark \in \{XY,Y,\varnothing\}}
\E\left[e^{{-}(a^\mark)^\top \xi^\mark}\;\Big|\; \bar
X_i\right]^{|\cS_\mark|}\right],
\end{align*}
where $\xi^\mark$ is a vector in $\inbraces{e_1,\dots,e_m} \cup \inbraces{0}$
with distribution conditional on $\bar X_i$ given by
    \begin{align*}
        \PP\insquare{\xi^\mark = 0 \;\Big|\; \bar X_i } = 1-\sum_{\ell=1}^m
\rho_{\ell,n}^\mark(\bar X_i),
\qquad \PP\insquare{ \xi^\mark = e_\ell \;\Big|\; \bar X_i} =
\rho_{\ell,n}^\mark(\bar X_i) \quad\text{for } 1 \leq \ell \leq m,
    \end{align*}
and $\rho_{\ell,n}^\mark(\bar X_i)$ is as defined in
\eqref{eq:ExactM_convergence_rho_ell} for the rectangle $R_\ell^\dagger$.
Noting that
\[\frac{1}{n}\Big(|\cS_{XY}|,|\cS_X|,|\cS_Y|,|\cS_\varnothing|\Big)
\to
\frac{1}{(1-p)^2}(p_{XY},p_X,p_Y,p_\varnothing)\]
and applying a dominated convergence argument similar to
\eqref{eq:mean_laplace_transform_poisson}, we obtain
a.s.\ as $n \to \infty$
\begin{align*}
&\E[\Psi_{\eta_N}(s) \mid N,\cS_{XY},\cS_X,\cS_Y,\cS_\varnothing]\\
&\to \Big(p\,\E_{\eps \sim q} [e^{-s^{XY}(0,\eps)}]
+(1-p)\E_{\eps \sim q} [e^{-s^X(0,\eps)}]\Big)\\
&\hspace{1in}\times \E_{X \sim \Lambda}\left[
\exp\left(\frac{\Lambda(X)}{(1-p)^2}\sum_{\dagger \in \{XY,X,Y,\varnothing\}}
p_{\dagger} \sum_{\ell=1}^m (e^{-t_\ell^\dagger}-1)
\iint_{R_\ell^\dagger} q(y-x) \, \ud x \, \ud y \right)\right]
\end{align*}
The right side is precisely $\Psi_{ \eta}(s)$.
A similar concentration argument as in Lemma
\ref{lemma:ExactM_convergenceToPPP2} for
exact matching shows that $\Psi_{\eta_N}(s)
-\E[\Psi_{\eta_N}(s) \mid N,\cS_{XY},\cS_X,\cS_Y,\cS_\varnothing] \to 0$ a.s.,
establishing the lemma.
\end{proof}

\begin{corollary}
\label{cor:partial_matching_PPP_convergence}
Conditional on $\cD_N$, let $I$ be a uniform random index in $[N_X]$,
and consider the joint law of
\begin{equation}\label{eq:softtripleN}
\Big(\{n(X_i-X_I,Y_{\pi^*(i)}-X_I)\}_{i \in \dom(\pi^*)},
\{n(X_i-X_I)\}_{i \in [N_X] \setminus \dom(\pi^*)},
\{n(Y_j-X_I)\}_{j \in [N_Y] \setminus \range(\pi^*)}\Big)
\end{equation}
as a point process on $\R^2$ and two point processes on $\R$.
Identify analogously
$(\XSetPPP_{\Lambda,p},\YSetPPP_{\Lambda,p},\pi_{\Lambda,p}^*)$
of Definition \ref{def:partial_PPP}
as the joint law of the point processes
\begin{equation}\label{eq:softtriplelimit}
\Big(\{(\XpointPPP,\pi_{\Lambda,p}^*(\XpointPPP))\}_{\XpointPPP \in
\dom(\pi_{\Lambda,p}^*)},\XSetPPP_{\Lambda,p}
\setminus \dom(\pi_{\Lambda,p}^*),
\YSetPPP_{\Lambda,p} \setminus \range(\pi_{\Lambda,p}^*)\Big).
\end{equation}
Then conditional on $\cD_N$, over the uniform random choice of $I \in [N_X]$,
a.s.\ as $n \to \infty$,
\begin{align*}
&\Big(\{n(X_i-X_I,Y_{\pi^*(i)}-X_I)\}_{i \in \dom(\pi^*)},
\{n(X_i-X_I)\}_{i \in [N_X] \setminus \dom(\pi^*)},
\{n(Y_j-X_I)\}_{j \in [N_Y] \setminus \range(\pi^*)}\Big)\\
&\overset{\text{w}}{\longrightarrow} 
(\XSetPPP_{\Lambda,p},\YSetPPP_{\Lambda,p},\pi_{\Lambda,p}^*).
\end{align*}
\end{corollary}
\begin{proof}
\eqref{eq:softtripleN} is the tuple consisting of $\eta_N^{XY}$,
the $X$-marginal of $\eta_N^X$, and the $Y$-marginal of $\eta_N^Y$.
By computing the Laplace functional following a similar argument as in
Lemma \ref{lemma:ExactM_convergenceToPPP2},
one may check that \eqref{eq:softtriplelimit}
has the same joint law as the tuple consisting
of $\eta^{XY}$, the $X$-marginal of $\eta^X$,
and the $Y$-marginal of $\eta^Y$. Then the result follows directly from
Lemma \ref{lemma:softM_Unmarked}.
\end{proof}

\subsubsection{Proof of Proposition \ref{prop:partial_infvolume}}

\begin{proof}[Proof of Proposition \ref{prop:partial_infvolume}]
Let $M' \geq M > 0$ be two locality parameters, and
let $\widehat{P}_i^M, \widehat P_i^{M'}$
denote the corresponding outputs of Algorithm \ref{alg:partial}.
Theorem \ref{thm:soft_alg} implies that a.s. for all large $n$,
\[\frac{1}{N_X} \sum_{i=1}^{N_X} \TV(\widehat{P}_i^M, \widehat P_i^{M'}) \le 
\frac{1}{N_X} \sum_{i=1}^{N_X}\TV(P_i, \widehat P_i^{M}) + \TV(P_i, \widehat P_i^{M'}) 
\le o_M(1).\]
For any $i \in [N_X]$, by definition of Algorithm \ref{alg:partial}, 
$\TV(\widehat{P}_i^M, \widehat P_i^{M'})$ is a function of the point processes
\eqref{eq:softtripleN} and, conditional on the number of such points, depends
continuously on the locations of these points. Then by the local weak 
convergence of Corollary \ref{cor:partial_matching_PPP_convergence}
and the continuous mapping theorem, almost surely
\[\E_{\Lambda,p}[\TV(Q_{M}^0, Q_{M'}^0)] = \lim_{n \to \infty} \frac{1}{N_X}
\sum_{i=1}^{N_X} \TV(\widehat{P}_i^M, \widehat P_i^{M'})= o_M(1).\]
The rest of the proof is identical to that of Proposition \ref{prop:exact_infvolume}.
\end{proof}

\begin{proof}[Proof of Theorem \ref{thm:partial_asymptotics}]
Fix any locality parameter $M>0$, and let $\widehat{P}_i^M$ be the output of
Algorithm~\ref{alg:partial} applied with $M$. By the triangle inequality, we have
\begin{align*}
 \abs{ \frac{1}{N_X}\sum_{i=1}^{N_X} f(P_i,\pi^*(i))
 -\E_{\Lambda,p}[f(\ProbOnPPP,\pi^*_{\Lambda,p}(0))] }
 \le  (\text{I}) + (\text{II}) + (\text{III}) 
\end{align*}
where
\begin{align*}
 (\text{I}) &= \abs{ \frac{1}{N_X}\sum_{i=1}^{N_X} f(P_i,\pi^*(i)) -
 \frac{1}{N_X}\sum_{i=1}^{N_X} f(\widehat{P}_i^M,\pi^*(i))
 }, \\
  (\text{II}) &= \abs{\frac{1}{N_X}\sum_{i=1}^{N_X} f(\widehat{P}_i^M,\pi^*(i))
 - \E_{\Lambda,p}[f(Q^0_M,\pi^*_{\Lambda,p}(0))] }, \\
  (\text{III}) &= \abs{ \E_{\Lambda,p}[f(Q_M^0,\pi^*_{\Lambda,p}(0))] 
 - \E_{\Lambda,p}[f(\ProbOnPPP,\pi^*_{\Lambda,p}(0))] }.
\end{align*}
Since $f(\cdot)$ is $\TV$-continuous, $(\text{I}) \leq o_M(1)$ a.s.\ for all
large $n$ by Theorem \ref{thm:soft_alg},
and $(\text{III}) \le o_M(1)$ by Proposition \ref{prop:partial_infvolume}.
By the local weak convergence of Corollary
\ref{cor:partial_matching_PPP_convergence} and the continuous mapping theorem,
$(\text{II}) \to 0$ a.s.\ as $n \to \infty$, showing the theorem.
\end{proof}

\subsection*{Acknowledgments}

This research was supported in part by NSF DMS-2142476 and a Sloan Research Fellowship.

\bibliographystyle{plain}
\bibliography{ref}

\end{document}